\documentclass[reqno]{amsart}

\usepackage{amsmath,amsthm,amsfonts,amssymb,young}
\usepackage{array, boldline, makecell, booktabs}
\usepackage{bm}
\usepackage{ytableau}
\usepackage{euscript, mathrsfs} 
\usepackage{tikz}
\usetikzlibrary{backgrounds}
\usepackage{amsmath}
\usepackage{mathtools}
\usepackage[colorlinks]{hyperref}
\usepackage{cleveref}
\usepackage[margin=1in]{geometry}
\numberwithin{equation}{section}    
\usepackage[all]{xy}
\usepackage{graphicx,import}
\usepackage{amscd}
\usepackage{color}
\usepackage{enumerate}
\usepackage{fourier}
\usepackage{tikz}
\usepackage{cases}
\usepackage{cite}
\usepackage[T1]{fontenc}
\usepackage{tikz-cd}

\ytableausetup{boxsize=1em}

\newtheorem{thm}{Theorem}[section]
\newtheorem{lem}[thm]{Lemma}
\newtheorem{proposition}[thm]{Proposition}
\newtheorem{corollary}[thm]{Corollary}
\theoremstyle{definition}
\newtheorem{example}[thm]{Example}
\newtheorem{definition}[thm]{Definition}

\newtheorem{rmk}[thm]{Remark}

\DeclareMathOperator{\midd}{midd}
\DeclareMathOperator{\tol}{tol}
\DeclareMathOperator{\inc}{Inc}
\DeclareMathOperator{\std}{std}
\DeclareMathOperator{\Aug}{Aug}
\DeclareMathOperator{\SP}{SP}
\DeclareMathOperator{\FG}{FG}

\DeclareMathOperator{\wt}{wt}
\DeclareMathOperator{\HW}{HW}
\DeclareMathOperator{\oc}{oc}
\DeclareMathOperator{\TL}{TL}
\DeclareMathOperator{\tab}{tab}
\DeclareMathOperator{\ps}{ps}
\DeclareMathOperator{\pss}{pss}
\DeclareMathOperator{\nss}{nss}
\DeclareMathOperator{\flip}{flip}
\DeclareMathOperator{\ns}{ns}
\DeclareMathOperator{\SSYT}{SSYT}
\DeclareMathOperator{\ROHS}{ROHS}
\DeclareMathOperator{\hw}{hw}
\DeclareMathOperator{\KL}{KL}
\DeclareMathOperator{\cind}{cind}
\DeclareMathOperator{\rind}{rind}
\DeclareMathOperator{\red}{red}
\DeclareMathOperator{\SYT}{SYT}
\DeclareMathOperator{\LR}{LR}
\DeclareMathOperator{\SSOT}{SSOT}
\DeclareMathOperator{\GSSOT}{GSSOT}
\DeclareMathOperator{\SOT}{SOT}
\DeclareMathOperator{\GSOT}{GSOT}
\DeclareMathOperator{\SSROT}{SSROT}
\DeclareMathOperator{\energy}{energy}

\DeclareMathOperator{\vac}{vac}
\DeclareMathOperator{\Par}{Par}
\DeclareMathOperator{\second}{second}
\DeclareMathOperator{\lseq}{lseq}

\newcommand{\sboxone}{\scalebox{0.4}{\ydiagram{1}}}
\newcommand{\sboxtwo}{\scalebox{0.4}{\ydiagram{2}}}
\newcommand{\sboxeleven}{\raisebox{0.2em}{\scalebox{0.3}{\ydiagram{1,1}}}}

\newsavebox\youngA
\newsavebox\youngB
\newsavebox\youngC
\newsavebox\youngD
\newsavebox\youngE
\newsavebox\youngF
\newsavebox\youngG
\newsavebox\youngH
\savebox\youngA{\begin{ytableau} \none[1] & \none[1] & \none[1]\\ \none[2] & \none[2]\\ \none[3] \end{ytableau}}%
\savebox\youngB{\begin{ytableau} \none[1] & \none[1] & \none[1] & \none[2]\\ \none[2] \\ \none[3] \end{ytableau}}%  
\savebox\youngC{\begin{ytableau} \none[1] & \none[1] & \none[1] & \none[3]\\ \none[2] & \none[2] \end{ytableau}}%
\savebox\youngD{\begin{ytableau} \none[1] & \none[1] & \none[1]\\ \none[2] & \none[2] &\none[3]\\ \end{ytableau}}%      
\savebox\youngE{\begin{ytableau} \none[1] & \none[1] & \none[1] & \none[2] &\none[3] \\ \none[2] \\ \end{ytableau}}%
\savebox\youngF{\begin{ytableau} \none[1] & \none[1] & \none[1] &\none[2] \\ \none[2] &\none[3]\\ \end{ytableau}}%  
\savebox\youngG{\begin{ytableau} \none[1] & \none[1] & \none[1] &\none[2] &\none[2] \\ \none[3] \end{ytableau}}%
\savebox\youngH{\begin{ytableau} \none[1] & \none[1] & \none[1] & \none[2] & \none[2] &\none[3] \end{ytableau}}%  

\title{Lusztig $q$-weight multiplicities and Kirillov-Reshetikhin crystals}

\author{Hyeonjae Choi}
\address{Department of Mathematical Sciences \\ Seoul National University \\Seoul 151-247 \\ Korea}
\email{2pun4ani@snu.ac.kr}

\author{Donghyun Kim}
\address{Department of Mathematical Sciences \\ Seoul National University \\Seoul 151-247 \\ Korea}
\email{hyun920310@snu.ac.kr}

\author{Seung Jin Lee}
\address{Department of Mathematical Sciences and Research Institute of
 Mathematics, Seoul National University\\Seoul 151-247 \\ Korea}
\email{lsjin@snu.ac.kr}

\begin{document}

\begin{abstract}
Lusztig $q$-weight multiplicities extend the Kostka-Foulkes polynomials to a broader range of Lie types. 
In this work, we investigate these multiplicities through the framework of Kirillov-Reshetikhin crystals. 
Specifically, for type $C$ with dominant weights and type $B$ with dominant spin weights, we present a combinatorial formula for Lusztig $q$-weight multiplicities in terms of energy functions of Kirillov-Reshetikhin crystals, generalizing the charge statistic on semistandard Young tableaux for type $A$.
Additionally, we introduce level-restricted $q$-weight multiplicities for nonexceptional types, and prove positivity by providing their combinatorial formulas. 
\end{abstract}

\maketitle

\section{Introduction}\label{Sec: Intro}
The \emph{Kostka-Foulkes polynomial} $K_{\lambda,\mu}(q)$ is the coefficient of the modified Hall-Littlewood polynomials when expressed in the Schur basis. It has a combinatorial formula involving a beautiful statistic known as charge \cite{LasSch1978}:
\begin{equation}\label{eq: charge}
K_{\lambda,\mu}(q) = \sum_{T \in \SSYT(\lambda,\mu)} q^{\text{charge}(T)}
\end{equation}
where the sum is taken over all semistandard Young tableaux of shape $\lambda$ with weight $\mu$. The Kostka-Foulkes polynomial has its origins in geometry and representation theory, with significant connections to flag varieties and Springer fibers \cite{Springer1977,DCP1981,Tan1982,GP1992}. 
Due to its ubiquity and rich structure, the study of the Kostka-Foulkes polynomial has been an active area of research for decades.
Several generalizations of the Kostka-Foulkes polynomials exist, including Macdonald polynomials and Lusztig $q$-weight multiplicities. 
Lusztig $q$-weight multiplicities generalize the Kostka-Foulkes polynomials to any Lie type.
In this paper, we focus on Lusztig $q$-weight multiplicities, particularly in types beyond $A$.

Let $\mathfrak{g}_n$ be a classical simple Lie algebra and let $L: R^{+}\rightarrow \mathbb{Z}$ be a function on the set of positive roots. 
For dominant weights $\lambda, \mu$, we define 
\begin{equation}\label{eq:q-analog}
    \operatorname{KL}^{\mathfrak{g}_n,L}_{\lambda,\mu}(q)=\sum_{w\in W}(-1)^w[e^{w(\lambda + \rho)-(\mu + \rho)}]\prod_{\alpha\in R^{+}}\frac{1}{1-q^{L(\alpha)}e^\alpha}
\end{equation}
where $[e^\beta]f$ denotes the coefficient of $e^\beta$ in $f$, $\rho = \frac{1}{2}\sum_{\alpha \in R^{+}}\alpha$, $W$ is the Weyl group of $\mathfrak{g}_n$ and $(-1)^{w}$ denotes $(-1)^{\ell(w)}$ where $\ell(w)$ is the length of $w$.
Notably, $\KL_{\lambda,\mu}^{\mathfrak{g}_n,L}(1)$ equals the weight multiplicity.
Although $\KL^{\mathfrak{g}_n,L}_{\lambda,\mu}(q)$ is not a polynomial with nonnegative coefficients in general, we examine two cases where this holds true.

First, we set $L\equiv1$, i.e., $L(\alpha)=1$ for all $\alpha\in R^{+}$. 
In this case, the polynomial $\KL^{\mathfrak{g}_n,L\equiv1}_{\lambda,\mu}(q)$, which we denote simply as $\KL^{\mathfrak{g}_n}_{\lambda,\mu}(q)$, is known as the \textit{Lusztig $q$-weight multiplicity} \cite{Lusztig1983}. 
The Lusztig $q$-weight multiplicity has been extensively studied from both geometric and algebraic perspectives \cite{Broer1993,Kir99,Bry1989}. 
In particular, its nonnegativity follows from the theory of the affine Kazhdan-Lusztig polynomial \cite{Lusztig1983}.

On the other hand, relatively few combinatorial results exist regarding the positivity of these multiplicities.
When $\mathfrak{g}_n$ is of type $A$, i.e., $\mathfrak{g}_n = A_{n-1}$, the Lusztig $q$-weight multiplicity $\KL^{A_{n-1}}_{\lambda,\mu}(q)$ coincides with the Kostka-Foulkes polynomial.
The combinatorics of the Kostka-Foulkes polynomial are well understood via the charge formula of Lascoux and Sch\"utzenberger \eqref{eq: charge}.
A natural question is whether the charge statistic can be generalized beyond type $A$ to provide a positive combinatorial formula for the Lusztig $q$-weight multiplicities.
This has remained a long-standing open problem for decades, with only partial results known, even when restricting to specific Lie types.

In this paper, we provide a combinatorial formula for $\KL^{C_n}_{\lambda,\mu}(q)$ and $\operatorname{KL}^{B_n}_{\alpha,\beta}(q)$, for dominant weights $\lambda,\mu$ and dominant spin weights $\alpha$ and $\beta$, using the energy function in the column KR crystals (Theorem \ref{thm: C lusztig} and \ref{thm: B lusztig}).
Note that the energy function in the column KR crystals for affine type $A$ is essentially the same as the charge statistic \cite{NY1997}.

We introduce another $q$-version of the weight multiplicity for nonexceptional types, which we call the \textit{level-restricted $q$-weight multiplicity}.
We define a function $L_A: R^+\rightarrow \mathbb{Z}$ as follows:
\begin{align*}
    L_A(\alpha)=\begin{cases*}
        1 \qquad \text{if $\alpha=\varepsilon_i-\varepsilon_j$ for some $i<j$},\\
        0 \qquad \text{otherwise}.
    \end{cases*}
\end{align*}
We denote the level-restricted $q$-weight multiplicity by $\KL^{\mathfrak{g}_n,L_A}_{\lambda,\mu}(q)$.
We provide a combinatorial formula for $\KL^{\mathfrak{g}_n,L_A}_{\lambda,\mu}(q)$ for nonexceptional types (Theorem \ref{thm: level formula}). 
Notably, our results refine the $X=K={^\infty }\KL$ theorem for tensor products of row KR crystals \cite{LS2007, S05} and provide new combinatorial objects for weight multiplicities:
$\GSSOT$ and $\SSROT$ (see Definition \ref{def: gssot} and \ref{def: ssrot}).
The following table summarizes the main results:
\begin{table}[h]\label{table: summary}
    \centering
        \begin{tabular}{|c!{\vrule width 1.7pt}c|c|}
        \hline
        Type & Lusztig $q$-weight multiplicity & Level-restricted $q$-weight multiplicity\\\hline
        $B$&column KR crystal of type $D_{N+1}^{(2)} (\text{spin weight})$ &row KR crystal of type $D_{N+1}^{(2)}$\\\hline
        $C$&column KR crystal of type $B_{N}^{(1)}$&row KR crystal of type $C_{N}^{(1)}$\\\hline
        $D$&?&row KR crystal of type $B_{N}^{(1)}$\\\hline
        \end{tabular}
        \caption{Summary of main results.}
\end{table}

Along the way to the proof, we make extensive use of the splitting map \cite{LS2007, S05, LOS2012}, a generalization of 
the standardization map by Lascoux, which transforms a semistandard Young tableau (SSYT) into a standard Young tableau (SYT) while preserving the cocharge \cite{Las1989}. The splitting map sends an element in $\otimes_i B^{r_i,s_i}$ to an element in $(B^{1,1})^{\otimes \sum_i r_i s_i}$ while maintaining the coenergy function and the classical crystal structure.
Computing the energy function is generally complex, but when an element resides in $B = (B^{1,1})^{\otimes n}$, the computation becomes significantly simpler.
Thus, the splitting map serves as a crucial tool for simplifying the calculations of the energy functions.
We develop the associated techniques in Appendix \ref{Sec: append A}.

\subsection{Organization}
This paper is organized as follows:  
In Section \ref{Sec: Prel}, we provide the necessary background on crystal theory, with a particular focus on KR crystals.
In Section \ref{sec: Objects}, we introduce the combinatorial objects $\SSOT$, $\GSSOT$, and $\SSROT$, which will serve as key ingredients throughout this paper.
We establish their correspondence with certain classical highest weight elements in KR crystals.
In Section \ref{Sec: L}, we state and prove our main results concerning Lusztig $q$-weight multiplicities for type $C$ and type $B$ (spin weights), corresponding to the first column of Table \ref{table: summary}.
In Section \ref{Sec: Ltilde}, we prove results related to level-restricted $q$-weight multiplicities, as summarized in the second column of Table \ref{table: summary}.
Some open questions and directions for future work are listed in Section \ref{sec: future}.  

In Appendix \ref{Sec: append A}, we develop techniques exploiting the splitting map and prove several results used in Section \ref{Sec: L}. 
Appendix \ref{Sec: append B} contains the proof of Lemma \ref{lem: useful lemma}, which is used in Section \ref{sub: filtering x=k}.

\section*{acknowledgement}
The authors are grateful to Mark Shimozono and Jae-Hoon Kwon for fruitful conversations. D. Kim  was supported by the National Research Foundation of Korea (NRF) grant funded by the Korean government (MEST) (No. 2019R1A6A1A10073437) and individual NRF grants 2022R1I1A1A01070620. 
S. J. Lee was supported by the National Research Foundation of Korea (NRF) grant funded by the Korean government (MSIT) (No.0450-20240021)

\section{Preliminaries}\label{Sec: Prel}
\subsection{Partitions and orthogonal complements}

A \emph{partition} $\mu = (\mu_1, \mu_2, \dots, \mu_\ell)$ is a weakly decreasing sequence of positive integers. We say that $\mu$ is a partition of $n$, denoted by $\mu \vdash n$, if $|\mu| := \sum_{i} \mu_i = n$. The \emph{length} of $\mu$, denoted $\ell(\mu)$, is the number of parts in the partition. The set of partitions of length no greater than $n$ will be denoted by $\Par_n$ and we frequently regard $\lambda\in \Par_n$ as a vector of length $n$ by appending zeros if necessary. We represent partitions using \emph{Young diagrams} in English notation, where a partition is displayed as a finite collection of cells arranged in left-justified rows, with the number of cells in each row (from top to bottom) corresponding to the parts of the partition. The \emph{conjugate partition} of $\mu$, denoted by $\mu^t = (\mu^t_1, \mu^t_2, \dots)$ \footnote{We will not use the notation $\mu'$ for the conjugate partition of $\mu$.}, is obtained by reflecting the Young diagram of $\mu$ along its main diagonal.

For a vector $\lambda = (\lambda_1, \lambda_2, \dots, \lambda_n)$ of length $n$ and a number $g$, we define the \textit{orthogonal complement} of $\lambda$ with respect to $g$, denoted $\oc(\lambda, g)$, as $(g - \lambda_n, \dots, g - \lambda_2, g - \lambda_1)$. Additionally, we define $\overline{\oc}(\lambda, g)$ as 
$
(g - \lambda_1, g - \lambda_2, \dots, g -\lambda_n). 
$
Note that if $\lambda$ is a partition and $g \geq \lambda_1$ is a positive integer, then $\oc(\lambda, g)$ also represents a partition (possibly with extra zeros at the end).

\subsection{Crystals}
Let $\mathfrak{g}$ be a classical simple Lie algebra (of type $A_{n-1}$, $B_n$, $C_n$, or $D_n$). We use the following notations.
Let $\mathfrak{b}$ denote a Borel subalgebra and $\mathfrak{h}$ a Cartan subalgebra. 
The set of nodes of the Dynkin diagram is denoted by $J$. 
The simple roots are $\{\alpha_i : i \in J\}$, and the simple coroots are $\{\alpha_i^\vee : i \in J\}$. 
The set of positive roots with respect to $\mathfrak{b}$ is $R^+ \subset \mathfrak{h}^*$.
The fundamental weights are $\{\omega_i : i \in J\}$. 
The weight lattice is given by $P = \bigoplus_{i \in J} \mathbb{Z} \omega_i$, and the set of dominant weights is $P^+ = \bigoplus_{i \in J} \mathbb{Z}_{\geq 0} \omega_i$.
We denote the Weyl group by $W$ and the evaluation pairing by $\langle - , - \rangle$.

The fundamental weights are given by $\omega_i=(1^i,0^{n-i})$ for $0 \leq i \leq n-2$, $\omega_{n-1}^{B_n}=\omega_{n-1}^{C_n}=(1^{n-1},0)$, $\omega_{n-1}^{D_n}=(\frac{1}{2}^{n-1},-\frac{1}{2})$, $\omega_{n}^{B_n}=\omega_{n}^{D_n}=(\frac{1}{2}^{n-1},\frac{1}{2})$, and $\omega_n^{C_n} = (1^n)$. The half sum of positive roots (which is equal to the sum of fundamental weights) $\rho$ is given by $\rho^{B_n}=\frac{1}{2}(2n-1,2n-3,\dots,1),\rho^{C_n}=(n,n-1,\dots,1),$ and $\rho^{D_n}=(n-1,n-2,\dots,0)$. The set of dominant weights $P^+$ is given by: 
$\Par_n\cup\{\mu^{\sharp}: \mu\in \Par_n\}$ if $\mathfrak{g}$=$B_n$, $\Par_n$ if  $\mathfrak{g}$=$C_n$, and $\{(\mu_1,\dots,\mu_{n-1},\pm\mu_n): \mu\in \Par_n\}\cup \{(\mu^{\sharp}_1,\dots,\mu^{\sharp}_{n-1},\pm\mu^{\sharp}_n): \mu\in \Par_n\}$ if $\mathfrak{g}$=$D_n$, where $\mu^{\sharp}:=\mu+(\frac{1}{2})^n $ for $\mu\in \Par_n$.
We call $\mu\in P\cap (\mathbb{Z}+\frac{1}{2})^n$ a \emph{spin weight}.
Note that spin weights only exist when $\mathfrak{g}=B_n, D_n$.

In \cite{Lus90, Kashiwara90}, Kashiwara and Lusztig independently introduced crystal theory which provides a combinatorial way to understand the representation theory of quantum algebras $U_q(\mathfrak{g})$. 
Instead of the representation-theoretic view, we review the $\mathfrak{g}$-crystal axiomatically \cite{bump2017crystal}. A $\mathfrak{g}$-crystal is a nonempty set $B$ together with \textit{crystal operators} $e_i,f_i: B \to B \cup \{0\}$ for $i\in J$,  \textit{weight function}  wt : $B \to P$, and $\epsilon_i,\varphi_i : B \to \mathbb{Z}_{\geq 0}$, where 
\begin{align*}
    \varphi_i(x)=\max\{ k \in \mathbb{Z}_{\geq 0} | f_i^k(x) \neq 0\}, 
    \quad \epsilon_i(x)=\max\{ k \in \mathbb{Z}_{\geq 0} | e_i^k(x) \neq 0\}.
\end{align*}    
A $\mathfrak{g}$-crystal satisfies the following conditions:
\begin{enumerate}
    \item For $x,y \in B$ and $i \in J$, $e_i(x)=y$ if and only if $f_i(y)=x$. 
    In this case, $\text{wt}(y) =\text{wt}(x) + \alpha_i$.
    \item For all $x \in B$ and $i \in J$, $\varphi_i(x) = \langle \text{wt}(x),\alpha_i^\vee \rangle + \epsilon_i(x).$ 
\end{enumerate}
An element $u \in B$ is called a \textit{classical highest weight element} if $e_i(u)=0$ for all $i \in J$.
For a $\mathfrak{g}$-crystal $B$, we can associate it with a directed graph called the \textit{crystal graph} of $B$.
The vertex set is $B$, and we draw an edge $x \overset{i}\to y$ if $f_i(x) = y$.
Every $\mathfrak{g}$-crystal $B$ with a connected crystal graph has a unique classical highest weight element of weight $\lambda$ and in this case we have $B=B(\lambda)$, where $B(\lambda)$ is the crystal corresponding to the irreducible finite $U_q(\mathfrak{g})$-module of the highest weight $\lambda$.
%p38 u_\lambda such that wt(u_\lambda)=\lambda 

We introduce the \textit{tensor product} $B \otimes C$ of $\mathfrak{g}$-crystal $B,C$. 
As a set, $B\otimes C$ is the Cartesian product $B \times C$.
We define crystal operators as follows \footnote{ In this paper, the convention for tensor product is opposite to Kashiwara's original convention.}:
\begin{align}
    f_i(x\otimes y)=\begin{cases*}
        f_i(x)\otimes y \qquad \text{if } \varphi_i(y) \leq \epsilon_i(x),\\
        x \otimes f_i(y) \qquad \text{if }\varphi_i(y) > \epsilon_i(x),
    \end{cases*} \quad \text{and} \quad
    e_i(x\otimes y)=\begin{cases*}
        e_i(x)\otimes y \qquad \text{if } \varphi_i(y) < \epsilon_i(x),\\
        x \otimes e_i(y) \qquad \text{if }\varphi_i(y) \geq \epsilon_i(x).
    \end{cases*}\label{eq: ei operator}
\end{align}
Also, we have the following:
\begin{align}
    &\text{wt}(x\otimes y) = \text{wt}(x) + \text{wt}(y),\\
    &\varphi_i(x\otimes y) = \max(\varphi_i(x),\varphi_i(y)+\langle \text{wt}(x),\alpha_i^\vee \rangle ),\\
    &\epsilon_i(x\otimes y) = \max(\epsilon_i(y),\epsilon_i(x)-\langle \text{wt}(y),\alpha_i^\vee \rangle )\label{eq: epsilon}.
\end{align}

\subsection{KR crystals}
Let $\hat{\mathfrak{g}} \supset \hat{\mathfrak{g}}' \supset \mathfrak{g}$ be an affine Kac-Moody algebra with an index set (the set of nodes for the affine Dynkin diagram) $I = J \cup \{0\}$, its derived subalgebra, and the simple Lie algebra.  
Let $U_q(\hat{\mathfrak{g}}) \supset U_q'(\hat{\mathfrak{g}}) \supset U_q(\mathfrak{g})$ be the corresponding quantum algebras.  
Chari and Pressley classified finite-dimensional irreducible $U_q'(\hat{\mathfrak{g}})$-modules in \cite{CP1994, CP1998}.  
For every $(r, s) \in J \times \mathbb{Z}_{\geq 0}$, there exists a finite-dimensional irreducible $U_q'(\hat{\mathfrak{g}})$-module $W^{(r)}_s$, called a \emph{Kirillov-Reshetikhin (KR) module}. It is conjectured that KR modules $W^{(r)}_s$ admit a crystal basis $B^{r,s}$ \cite{Hat1998, Hat2001}, and this is known to be true for nonexceptional types \cite{OS2008}. These $B^{r,s}$ are now referred to as \emph{Kirillov-Reshetikhin (KR) crystals}. Based on the attachment of the zero node to the Dynkin diagram of $\mathfrak{g}$, we group KR crystals as follows \cite{SZ2006}:
\begin{enumerate}
    \item kind $\emptyset$ : $A_{N-1}^{(1)}$ 
    \item kind $\sboxone$ : $D_{N+1}^{(2)}, A_{2N}^{(2)\dagger}, B_{N}^{(1)\dagger}$ 
    \item kind $\sboxtwo$ : $A_{2N}^{(2)}, C_{N}^{(1)}, A_{2N-1}^{(2)\dagger}$
    \item kind $\sboxeleven$ : $B_{N}^{(1)}, A_{2N-1}^{(2)}, D_{N}^{(1)}$.
\end{enumerate}

We use the symbol $\diamond \in \{\emptyset, \scalebox{0.7}{\ydiagram{1}}, \scalebox{0.7}{\ydiagram{2}}, \scalebox{0.7}{\ydiagram{1,1}} \}$ to indicate the kind. For each kind, we choose a representative as follows:
\begin{equation}\label{eq: classification}
    \text{kind $\emptyset$ : $A^{(1)}_{N-1}$} \qquad
    \text{kind $\sboxone$ : $D^{(2)}_{N+1}$} \qquad
    \text{kind $\sboxtwo$ : $C^{(1)}_{N}$} \qquad
    \text{kind $\sboxeleven$ : $B^{(1)}_{N}$}.
\end{equation}
Then we denote $B^{r,s}(\diamond)$ as the KR crystal $B^{r,s}$ corresponding to the given type.
For example, $B^{r,s}(\sboxone)$ is the KR crystal of type $D_{N+1}^{(2)}$, and $B^{r,s}(\hspace{0.4mm}\sboxeleven\hspace{0.4mm})$ is the KR crystal of type $B^{(1)}_N$. Furthermore, throughout this paper, we let $N$ be a sufficiently large number.

The KR crystal \( B^{r,s} \) is an affine crystal equipped with additional operators \( e_0 \) and \( f_0 \). The values \( \varphi_0 \) and \( \epsilon_0 \) are defined similarly to their classical counterparts and satisfy the relations \( \mathrm{wt}(y) = \mathrm{wt}(x) + \alpha_0 \) and \( \varphi_0(x) = \langle \mathrm{wt}(x), \alpha_0^\vee \rangle + \epsilon_0(x) \) when \( y = e_0(x) \), equivalently \( f_0(y) = x \). By discarding the operators \( e_0 \) and \( f_0 \), \( B^{r,s} \) becomes a classical crystal (\( \mathfrak{g} \)-crystal) and decomposes as:
\begin{equation}\label{eqeqeqeq}
B^{r,s}(\diamond) = \bigoplus_{\lambda} B(\lambda),
\end{equation}
where the sum runs over all partitions \( \lambda \) obtainable from the \( r \times s \) rectangle by removing pieces of shape \( \diamond \). There are some exceptional values of \( r \) for which this decomposition does not hold, but these cases do not arise as \( N \) is sufficiently large.

Let $\mathcal{C}$ be the category of tensor products of KR crystals. Then
$\mathcal{C}$ has the following remarkable properties \cite{Hat1998, Hat2001}:
\begin{enumerate}
    \item For any $B_1,B_2 \in \mathcal{C}$, there is a unique affine crystal isomorphism $R=R_{B_2,B_1} : B_2 \otimes B_1 \to B_1 \otimes B_2$ called the \emph{combinatorial $R$-matrix}.
    \item There is a map $\overline{H}=\overline{H}_{B_2,B_1} : B_2 \otimes B_1 \to \mathbb{Z}$, called the \emph{local energy function}, such that $\overline{H}$ is constant on classical component ($J$-component), and 
    \begin{align*}
        \overline{H}(e_0(b_2 \otimes b_1)) = \overline{H}(b_2 \otimes b_1) + \begin{cases*}
            1 \quad &if LL\\
            -1 \quad &if RR\\
            0 \quad &otherwise
        \end{cases*}
    \end{align*}
    where for $b_2 \otimes b_1 \in B_2 \otimes B_1$ and $R(b_2\otimes b_1) = b_1' \otimes b_2' \in B_1 \otimes B_2$, LL (RR) indicates that $e_0$ acts on the left (right) factor both times. Moreover, such $\overline{H}$ is unique up to a global addictive constant where we choose a proper normalization.
    \item For $B\in\mathcal{C}$, there is a map $\overline{D}_B : B \to \mathbb{Z}$ called the \emph{energy function}\footnote{In some literature $\overline{D}_B$ is referred to as the coenergy function. 
    We follow the SageMath notation where $\overline{D}_B$ is referred to as the energy function.}, such that $\overline{D}_B$ is constant on a classical component. For a single KR crystal $B^{r,s}(\diamond)$ and its element $b$, the energy function is given by 
    \begin{equation}\label{eq: energy function for single KR crystla}
        \overline{D}_{B^{r,s}(\diamond)}(b)=\frac{rs-|\lambda|}{|\diamond|}
    \end{equation}
    when $b\in B(\lambda)$ in \eqref{eqeqeqeq}. Now we define the energy function on $B\in\mathcal{C}$ recursively as follows. Let $B_1, \dots , B_n \in \mathcal{C}$, and $b_i \in B_i$ for $i = 1,\dots,n$. 
    Set $B=B_n\otimes \cdots \otimes B_1$ and $b=b_n\otimes \cdots \otimes b_1\in B$.
    The energy function is given by
    \begin{align*}
        \overline{D}_B(b)=\sum_{1\leq i < j \leq n} \overline{H}(b_j^{(i+1)} \otimes b_i) + \sum_{j=1}^n \overline{D}_{B_j}(b_j^{(1)})
    \end{align*}
    where for $1 \leq i < j \leq n$, $b_j^{(i)} \in B_j$ is determined by applying the combinatorial $R$-matrix consecutively, so that:
    \begin{align*}
        B_j \otimes \cdots \otimes B_{i+1} \otimes B_i \to & B_{j-1} \otimes  \cdots \otimes B_i \otimes B_j,\\
        b_j \otimes \cdots \otimes b_{i+1} \otimes b_i \mapsto & b_{j-1}' \otimes \cdots \otimes b_i' \otimes b_j^{(i)}.
    \end{align*}
    We may write $\overline{D}$ instead of $\overline{D}_B$ when $B$ is obvious from the context. In general, computing $\overline{D}$ is challenging.
    However, when $B = (B^{1,1})^{\otimes n}$, the energy function is easy to compute, and this case will be discussed later.
\end{enumerate}

In this paper, we only consider row KR crystals, i.e., $B^{1,s}$, and column KR crystals, i.e., $B^{r,1}$.
For a nonnegative integer vector $\mu=(\mu_1,\mu_2,\dots,\mu_n)$, we employ the following notation:
\begin{equation*}
    B_{\mu}:=B^{1,\mu_n}\otimes\cdots\otimes B^{1,\mu_1} \qquad B^{t}_{\mu}:=B^{{\mu_n},1}\otimes\cdots\otimes B^{{\mu_1},1}.
\end{equation*}
We use the notation $B_{\mu}(\diamond)$ and $B^{t}_{\mu}(\diamond)$ to conveniently indicate the type of the KR crystal, according to \eqref{eq: classification}.

We define $\HW(B)$ to be the set of classical highest weight elements in $B$ and $\HW(B, \lambda) = \{ b \in \HW(B) : \wt(b) = \lambda \}$. 
For $b \in B$, the classical highest weight element in the classical component containing $b$ will be denoted by $\hw(b)$.

Roughly speaking, an element of $B^{r,1}$ or $B^{1,s}$ is a word consisting of letters $1, 2, \dots, N$, which we call \emph{unbarred letters}, and $\overline{1}, \overline{2}, \dots, \overline{N}$, which we call \emph{barred letters}, and $0$.  
Throughout this paper, we do not consider any word that contains $0$. Moreover, we assume $N$ is sufficiently large and only consider words consisting of letters $m$ or $\overline{m}$, for $m$ sufficiently smaller than $N$.
Additionally, we will use the following total order on letters:  
\begin{equation}\label{eq: corder def}
    1 \prec 2 \prec 3 \prec \cdots \prec \bar{3} \prec \bar{2} \prec \bar{1}.
\end{equation}

From \eqref{eqeqeqeq}, regarded as a classical crystal, $B^{1,s}$ is decomposed as a direct sum of $B(k \omega_1)$'s and $B^{r,1}$ is decomposed as a direct sum of $B(\omega_k)$'s as follows:
\begin{align}\label{eq: column row KR decompose}
    B^{1,s}(\diamond)=\begin{cases*}
    \bigoplus_{k=0}^{s}B((s-k) \omega_1) \quad &\text{if  $\diamond=\sboxone$},\\
    \bigoplus_{k=0}^{\lfloor\frac{s}{2}\rfloor}B((s-2k) \omega_1) \quad &\text{if $\diamond=\sboxtwo$,} \\
    B(s\omega_1) \quad &\text{if $\diamond=\sboxeleven$,}
    \end{cases*} \quad \quad B^{r,1}(\diamond)=\begin{cases*}
    \bigoplus_{k=0}^{r}B(\omega_{r-k}) \qquad &\text{if  $\diamond=\sboxone$,} \\
    B(\omega_{r}) \qquad &\text{if  $\diamond=\sboxtwo$,} \\
    \bigoplus_{k=0}^{\lfloor\frac{r}{2}\rfloor}B(\omega_{r-2k})
    \qquad &\text{if  $\diamond=\sboxeleven$. }
    \end{cases*} 
\end{align}
Each $B(k \omega_1)$ consists of a word $v = v_1 \dots v_k$ where $v_1 \preceq v_2 \preceq \dots \preceq v_k$ (the empty word is denoted by $\emptyset$).
Abusing the notation, we identify a word $v$ with a multi-set $\{v_1, v_2, \dots, v_k\}$ and vice versa.
In other words, we identify a word $112\bar{3}\bar{2}$ with a multi-set $\{1, 1, 2, \bar{3}, \bar{2}\}$.

The situation is more subtle for $B(\omega_k)$.
An element of $B(\omega_k)$ is a word $v = v_1 \dots v_k$ where $v_1 \prec v_2 \prec \dots \prec v_k$ satisfying an admissibility condition.
We say that a word $v$ is \emph{admissible} if, for every $i$ such that $i, \bar{i} \in v$, we have
\[
    |\{c \prec i : c \in v\}| + |\{\bar{i} \prec c : c \in v\}| < i - 1.
\]
For example, any word $v$ that contains $1$ and $\bar{1}$ is not admissible, and words $12\bar{2}$ and $13\bar{3}\bar{2}$ are not admissible.

Given any word $v = v_1 \dots v_k$ with $v_1 \prec v_2 \prec \dots \prec v_k$ that may not be admissible, we introduce a map $\red$ that associates an admissible word from $v$.
We go through the following process:
\begin{itemize}
    \item (Step 1) If $v$ is admissible, terminate the process. Otherwise, among $i$'s such that $i, \bar{i} \in v$, choose $i$ such that $(|\{c \prec i : c \in v\}| + |\{\bar{i} \prec c : c \in v\}|)-i$ is the largest (if there are more than one such $i$, we pick the smallest one).
    \item (Step 2) Delete $i$ and $\bar{i}$ from $v$ and go back to (Step 1).
\end{itemize}
We define the output to be $\red(v)$. 
For example, $\red(1\overline{1}) = \emptyset$ and $\red(123\overline{3}) = 12$. 
It is elementary to check that for an admissible word $v$ of length $k$ and a positive integer $r$, there exists a unique word $w = w_1 \dots w_{k + 2r}$ such that $w_1 \prec \dots \prec w_{k + 2r}$ and $\red(w) = v$.
Moreover, the map $\red$ commutes with classical crystal operators, where we regard a word $v = v_1 \dots v_k$ with $v_1 \prec v_2 \prec \dots \prec v_k$ is embedded into $B(\omega_1)^{\otimes k}$ as $v_k \otimes \dots \otimes v_1$.
For example, we have $\red(12\overline{2}) = 1$ and $\red(12\overline{1}) = 2$ with $e_1(\overline{1} \otimes 2 \otimes 1) = \overline{2} \otimes 2 \otimes 1$ and $e_1(2) = 1$.
The map $\red$ will be a key ingredient to connect our combinatorial objects to KR crystals (Lemma \ref{lem: cind}).

Now, we give an explicit description for the energy function of an element of $b\in (B^{1,1}(\diamond))^{\otimes n}$ denoted by $b_n \otimes b_{n-1} \otimes \cdots \otimes b_1$. Here each $b_i$ is either a letter or an empty word $\emptyset$.
Note that an empty word $\emptyset$ is allowed only when $\diamond=\sboxone$. 

When $\diamond=\sboxtwo$ or $\sboxeleven$, the energy function $\overline{D}(b)$ is given by $\sum_{i=1}^{n-1} (n-i)\overline{H}_{\diamond}(b_{i+1},b_i)$ where 
\begin{align*}
    \overline{H}_{\sboxeleven}(x,y) = \begin{cases}
    2 &\textrm{ if } x=\overline{1} \textrm{ and } y=1,\\
    1 &\textrm{ if } x \succ y \textrm{ and } (x,y)\neq (\overline{1},1), \\
    0 &\textrm{ if } x \preceq y.
\end{cases}\qquad  \qquad \overline{H}_{\sboxtwo}(x,y) = \begin{cases}
    1 &\textrm{ if } x \succ y,  \\
    0 &\textrm{ if } x \preceq y. 
\end{cases}
\end{align*}
On the other hand, if $\diamond=\sboxone$, the energy function $\overline{D}(b)$ is given by $\sum_{i=1}^{n-1} 2(n-i)\overline{H}_{\sboxone}(b_{i+1},b_i) + \vac$, where $\vac$ denotes the number of $\emptyset$ among $b_i$'s, and 
\begin{align*}
    \overline{H}_{\sboxone}(x,y) = \begin{cases}
    1 &\textrm{ if } x \succ y \textrm{ or } x=y=\emptyset, \\
    0 &\textrm{ otherwise, } 
\end{cases}
\end{align*}
under the order $\emptyset \prec 1\prec 2\prec 3 \prec \cdots  \prec \bar{3}\prec \bar{2}\prec \bar{1}$.
Note that for $\diamond=\emptyset$ (affine type $A^{(1)}_{N-1}$), the energy function essentially coincides with the charge statistic given by Lascoux and Sch\"utzenberger \cite{NY1997}.
%example?

\subsection{Splitting map}
In \cite{Las1989}, Lascoux introduced the \textit{standardization map}, which maps a semistandard Young tableau to a standard Young tableau while preserving the cocharge.
A generalized version of Lascoux's standardization map exists, mapping an element in $\otimes_i B^{r_i,s_i}$ to an element in $(B^{1,1})^{\otimes \sum_i r_i s_i}$ while preserving the energy function (up to a global constant), which we term the splitting map.

The \emph{splitting map} $S: \otimes_{i} B^{r_i, s_i} \to (B^{1,1})^{\otimes \sum_i r_i s_i}$ is a classical crystal embedding, an injection that preserves a classical crystal structure, that exists for any tensor products of KR crystals.
As we only consider crystals of the form $B_{\mu}$ or $B^{t}_{\mu}$, the corresponding splitting maps will be denoted by $S_{\mu}$ and $S^{t}_{\mu}$, respectively.
We now describe $S_{\mu}$ (when $\diamond = \sboxone, \sboxtwo$, and $\sboxeleven$) and $S^{t}_{\mu}$ (when $\diamond = \sboxone$ and $\sboxeleven$) based on \cite{LOS2012, LS2007, S05}.

We first explain $S_{\mu}$. Regardless of the kind $\diamond$, they can be uniformly described. For $a \geq 2$, we define a map $\bar{S}: B^{1,a} \to B^{1,1} \otimes B^{1,a-1}$ as follows \cite[Equation (6.12)]{LOS2012}: representing $T \in B^{1,a}$ as a word $v_1 \dots v_k$, we have
\[
    \bar{S}(T) = 
    \begin{cases}
        \bar{1} \otimes 1 v_1 \dots v_k  & \text{if } a \geq k+2, \\
        \emptyset \otimes v_1 \dots v_k & \text{if } a = k +1, \\
        v_1 \otimes v_2 \dots v_k  & \text{if } a = k.
    \end{cases}
\]
Then we define the splitting map $S_\mu: B_\mu \to B_{(1^{|\mu|})}$ as follows:
\begin{enumerate}
    \item Find the smallest $i$ such that $\mu_i > 1$.
    \item Apply a composition of combinatorial $R$-matrices so that $B^{1, \mu_i}$ is placed on the rightmost tensor factor.
    \item Apply the map $\bar{S}$ on the rightmost tensor factor $B^{1, \mu_i}$, leaving other factors unchanged.
    \item Repeat this process until we arrive at $(B^{1,1})^{\otimes |\mu|}$.
\end{enumerate}

\begin{lem}\label{lem: splitting preserves energy}\cite{LOS2012}
    For $b \in B_\mu$, we have $ \overline{D}\left( S_\mu(b) \right)=\overline{D}(b) $.
\end{lem}

\begin{rmk}\label{rmk: splitting R commute}
    It is a conjecture that the splitting map commutes with combinatorial $R$-matrices \cite[Remark 6.2]{LOS2012}.
\end{rmk}

Now we describe $S_\mu^t$ for $B^{t}_{\mu}(\diamond)$ when $\diamond = \sboxone$ or $\sboxeleven$.
For $a \geq 2$, we define a map $\hat{S}: B^{a,1} \to (B^{1,1})^{\otimes a}$ as follows: representing $T \in B^{a,1}$ as a word $v_1 \dots v_k$, we have
\[
    \hat{S}(T) = \begin{cases}
        v_k \otimes \dots \otimes v_1 \otimes \underbrace{\emptyset \otimes \dots \otimes \emptyset}_{a - k}  & \text{if } \diamond = \sboxone, \\
        v_k \otimes \dots \otimes v_1 \otimes \underbrace{\bar{1} \otimes 1 \otimes \dots \otimes \bar{1} \otimes 1}_{a - k}  & \text{if } \diamond = \sboxeleven.
    \end{cases}
\]
Note that when $\diamond = \sboxeleven$, $a - k$ is necessarily an even number. Then $S^{t}_{\mu}$ can be defined exactly in the same way as $S_{\mu}$ by replacing $\hat{S}$ in place of $\bar{S}$.

\begin{rmk}
The description of $\hat{S}$ is different from the one given in \cite[Proposition 6.1]{LOS2012}. It is not hard to check that they are equal.
\end{rmk}

For a nonnegative integer vector $\alpha=(\alpha_1,\dots,\alpha_n)$, we define $||\alpha||=\sum_{i=1}^{n}(i-1)\beta_i$, where $\beta$ is a rearrangement of $\alpha$ into weakly decreasing order.
\begin{lem}\cite{LOS2012}\label{lem: splitting preserves coenergy}
 We have 
 \begin{align*}
     \overline{D}(S^{t}_{\mu}(b))=\begin{cases*}
         \overline{D}(b) + 2\times (\frac{|\mu|(|\mu|-1)}{2} - ||\mu||) \qquad & \text{if $b\in B_\mu^t(\sboxone)$},  \\
         \overline{D}(b) + (\frac{|\mu|(|\mu|-1)}{2} - ||\mu||) \qquad & \text{if $b\in B_\mu^t(\hspace{0.4mm}\sboxeleven\hspace{0.4mm})$}. 
     \end{cases*}
 \end{align*}
\end{lem}

\section{Combinatorial objects and embedding into KR crystals}\label{sec: Objects}
In this section, we introduce combinatorial objects (Definition \ref{def: ssot}, \ref{def: gssot} and \ref{def: ssrot}) and relate them with certain classical highest weight elements of KR crystals (Lemma \ref{lem: cind} and \ref{lem: rind}).

\begin{definition} \label{def: ssot}
    Oscillating horizontal strip (ohs for short) $(\mu,\nu,\lambda)$ of length $r$ is a sequence of three partitions satisfying:
    \begin{itemize}
        \item $\nu / \mu$ and $\nu / \lambda$ are horizontal strips,
        \item $|\nu / \mu| + |\nu / \lambda| = r$.
    \end{itemize}
    We define the \emph{initial shape}, denoted by $I(\mu,\nu,\lambda)$, to be $\mu$, and the \emph{final shape}, denoted by $F(\mu,\nu,\lambda)$, to be $\lambda$.
    We say that an ohs $(\mu,\nu,\lambda)$ is $g$-bounded if $\nu_1 \leq g$. 
    
    \textit{Semistandard oscillating tableaux} (SSOT for short) of shape $\lambda$ with weight $\mu$ is a sequence of ohs' $T = (T_1, \dots, T_n)$, where each $T_i$ is an ohs of length $\mu_i$ and 
    \begin{itemize}
        \item $I(T_1) = \emptyset$ and $F(T_n) = \lambda$,
        \item $F(T_i) = I(T_{i+1})$ for $1 \leq i \leq n-1$.
    \end{itemize}
    We define $c(T)\in \mathbb{Z}$ to be the minimal number such that every $T_i$ is $c(T)$-bounded. 
    
    The set of all SSOT of shape $\lambda$ with weight $\mu$ is denoted by $\SSOT(\lambda,\mu)$, and we write $\SSOT_g(\lambda,\mu)$ for the subset of $\SSOT(\lambda,\mu)$ consisting of $\SSOT$ $T$ satisfying $c(T) \leq g$.
\end{definition}

There is a bijection between King tableaux of shape $\lambda$ with weight $\mu$ and $\SSOT(\oc(\lambda,g),\overline{\oc}(\mu,g))$ \cite{Lee2023}.

\begin{definition} \label{def: gssot}
    Generalized oscillating horizontal strip (gohs for short) $(\mu,\nu,\lambda)$ of length $r$ is a sequence of three partitions satisfying:
    \begin{itemize}
        \item $\nu / \mu$ and $\nu / \lambda$ are horizontal strips,
        \item $|\nu / \mu| + |\nu / \lambda| = r \quad \text{or} \quad r-1$.
    \end{itemize}
    We say that a gohs $(\mu, \nu, \lambda)$ is $g$-bounded ($g \in \mathbb{Z}/2$) when
    \begin{itemize}
        \item $\nu_1 \leq g$, if $|\nu / \mu| + |\nu / \lambda| = r$,
        \item $\nu_1 + \frac{1}{2} \leq g$, if $|\nu / \mu| + |\nu / \lambda| = r-1$.
    \end{itemize}
    
    \textit{Generalized semistandard oscillating tableaux} (GSSOT for short) of shape $\lambda$ with weight $\mu$ is $T = (T_1, \dots, T_n)$, where each $T_i$ is a gohs of length $\mu_i$, $I(T_1)=\emptyset$, $F(T_n)=\lambda$ and $F(T_i)=I(T_{i+1})$ for $1\leq i\leq n-1$. We define $c(T)\in \mathbb{Z}/2$ to be the minimal number such that every $T_i$ is $c(T)$-bounded. 
    
    The set of all GSSOT of shape $\lambda$ with weight $\mu$ is denoted by $\GSSOT(\lambda,\mu)$, and we write $\GSSOT_g(\lambda,\mu)$ for the subset of $\GSSOT(\lambda,\mu)$ consisting of $\GSSOT$ $T$ satisfying $c(T) \leq g$.
\end{definition}
\begin{definition} \label{def: ssrot}
    Reverse oscillating horizontal strip (rohs for short) $(\mu,\nu,\lambda)$ of length $r$ is a sequence of three partitions satisfying:
    \begin{itemize}
        \item $\mu / \nu$ and $\lambda / \nu$ are horizontal strips,
        \item $|\mu / \nu| + |\lambda / \nu| = r$.
    \end{itemize}
    We say that an rohs $(\mu,\nu,\lambda)$ is $g$-bounded ($g \in \mathbb{Z}/2$) if
    \begin{equation*}
        \mu_1 + (\lambda_1 - \nu_1) + \max(\mu_2, \lambda_2) \leq 2g.
    \end{equation*}
    \textit{Semistandard reverse oscillating tableaux} (SSROT for short) of shape $\lambda$ with weight $\mu$ is $T = (T_1, \dots, T_n)$, where each $T_i$ is an rohs of length $\mu_i$, $I(T_1)=\emptyset$, $F(T_n)=\lambda$ and $F(T_i)=I(T_{i+1})$ for $1\leq i\leq n-1$.  We define $c(T)\in \mathbb{Z}/2$ to be the minimal number such that every $T_i$ is $c(T)$-bounded. 
       
    The set of all SSROT of shape $\lambda$ with weight $\mu$ is denoted by $\SSROT(\lambda,\mu)$, and we write $\SSROT_g(\lambda,\mu)$ for the subset of $\SSROT(\lambda,\mu)$ consisting of $\SSROT$ $T$ satisfying $c(T) \leq g$.
\end{definition}

Given an ohs $(\mu,\nu,\lambda)$, we define $\cind(\mu,\nu,\lambda)$ to be
\begin{align*}
    \cind(\mu,\nu,\lambda) := \{i: \mu^{t}_i+1=\nu^{t}_i\} \cup \{\bar{i}: \lambda^{t}_i+1=\nu^{t}_i\}.
\end{align*}
We also define $\rind(\mu,\nu,\lambda)$ to be the multi-set such that there are $(\nu_i-\mu_i)$-many $i$'s and $(\nu_i-\lambda_i)$-many $\bar{i}$'s. 
In other words, $\cind(\mu,\nu,\lambda)$ (respectively $\rind(\mu,\nu,\lambda)$) records column indices (respectively row indices) from the ohs $(\mu,\nu,\lambda)$. We can define $\cind(\mu,\nu,\lambda)$ and $\rind(\mu,\nu,\lambda)$ similarly when $(\mu,\nu,\lambda)$ is an rohs.

\begin{example}\label{ex: ohs cind}
    Consider the following ohs
    \begin{align*}
        T = (\ydiagram{2,1}, \ydiagram{3,2}, \ydiagram{3}).
    \end{align*}
    Then $\cind(T) = \{2,3,\bar{2},\bar{1}\}$ and $\rind(T) = \{1,2,\bar{2},\bar{2}\}$. 
    We also identify $\cind(T)$ by $23\bar{2}\bar{1}$ and $\rind(T)$ by $12\bar{2}\bar{2}$.
\end{example}

Now, we interpret $\SSOT(\lambda,\mu)$, $\GSSOT(\lambda,\mu)$, and $\SSROT(\lambda,\mu)$ as sets of classical highest weight elements of specific KR crystals (Lemma \ref{lem: cind} and \ref{lem: rind}). Lemma \ref{lem: cind} is straightforward, while Lemma \ref{lem: rind} requires an auxiliary lemma (Lemma \ref{lem: rohs ohs connection}) for the proof.

\begin{lem}\label{lem: cind}
For a partition $\lambda$ and a nonnegative integer vector $\mu$, we have the following:
    \begin{enumerate}
        \item There exists a bijection 
        \begin{equation*}
            \phi_c: \SSOT(\lambda,\mu) \rightarrow \HW(B_\mu^t(\hspace{0.4mm} \sboxeleven \hspace{0.4mm}), \lambda^{t}).
        \end{equation*}
        \item There exists a bijection 
        \begin{equation*}
            \phi_c: \GSSOT(\lambda,\mu) \rightarrow \HW(B_\mu^t(\sboxone), \lambda^{t}).
        \end{equation*}
    \end{enumerate}
\end{lem}

\begin{proof}
    In all cases, for $T = (T_1, T_2, \dots, T_n)$ which is an $\SSOT$ or a $\GSSOT$, the map is given by
    \begin{equation*}
        \phi_c(T) = \red(\cind(T_n)) \otimes \dots \otimes \red(\cind(T_2)) \otimes \red(\cind(T_1)).
    \end{equation*}
\end{proof}

\begin{lem}\label{lem: rohs ohs connection}
For partitions $\lambda$ and $\mu$, there exists a bijection
\begin{equation*}
    \Gamma:\{\text{ohs $(\mu,\nu,\lambda)$ of length $r$}\} \rightarrow \{\text{rohs $(\mu,\zeta,\lambda)$ of length $(r-2k)$ for $0 \leq k \leq \lfloor\frac{r}{2}\rfloor$}\}.
\end{equation*}
\end{lem}

\begin{proof}
    Given an ohs $T = (\mu, \nu, \lambda)$, denote the multi-set $\rind(T)$ by $A$, and apply the following process:
    \begin{enumerate}
        \item Select $i$ such that $i, \bar{i} \in A$.
        \item Replace $i$ and $\bar{i}$ with $i-1$ and $\overline{i-1}$. If $i = 1$, remove $1$ and $\bar{1}$.
        \item Repeat until no more pairs can be selected, ensuring the modified pairs are not reused.
    \end{enumerate}
    Let the resulting multi-set be $B$. We claim there exists an rohs $(\mu, \zeta, \lambda)$ such that $\rind(\mu, \zeta, \lambda) = B$.

    Let $a_i$ (respectively $b_i$) denote the number of $i$'s in $A$ (respectively in $B$), and $a'_i$ (respectively $b'_i$) denote the number of $\bar{i}$'s in $A$ (respectively in $B$). To ensure $B$ corresponds to a valid rohs, we need to show
    \begin{equation*}
        b'_i \leq \mu_i - \mu_{i+1}, \quad\text{and}\quad b'_i - b'_{i+1} + b_{i+1} \leq \mu_i - \mu_{i+1}, \quad \forall i.
    \end{equation*}

    Since $A$ arises from the ohs, the following inequalities hold for all $i$
    \begin{align}
        a_{i+1} \leq \mu_i - \mu_{i+1},  \quad\text{and}\quad
        a_{i+1} - a_i + a'_i \leq \mu_i - \mu_{i+1}. 
    \end{align}

    From the construction of $B$, we have
    \begin{align*}
        b'_i &= a'_i - \min(a_i, a'_i) + \min(a_{i+1}, a'_{i+1}), \\
        b_i &= a_i - \min(a_i, a'_i) + \min(a_{i+1}, a'_{i+1}).
    \end{align*}

    If $a_i \leq a'_i$, then
    \begin{align*}
        b'_i &= a'_i - a_i + \min(a_{i+1}, a'_{i+1}) \leq a'_i - a_i + a_{i+1} \leq \mu_i - \mu_{i+1}, \\
        b'_i - b'_{i+1} + b_{i+1} &= a'_i - a_i + \min(a_{i+1}, a'_{i+1}) + a_{i+1} - a'_{i+1} \leq a'_i - a_i + a_{i+1} \leq \mu_i - \mu_{i+1}.
    \end{align*}

    If $a_i > a'_i$, then
    \begin{align*}
        b'_i &= \min(a_{i+1}, a'_{i+1}) \leq a_{i+1} \leq \mu_i - \mu_{i+1}, \\
        b'_i - b'_{i+1} + b_{i+1} &= \min(a_{i+1}, a'_{i+1}) + a_{i+1} - a'_{i+1} \leq a_{i+1}\leq \mu_i - \mu_{i+1}.
    \end{align*}

    Thus, there exists an rohs $(\mu, \zeta, \lambda)$ with $\rind(\mu, \zeta, \lambda) = B$ and its length equals $(r - 2 \min(a_1, a'_1))$.

    Conversely, given an rohs $(\mu, \zeta, \lambda)$ of length $r - 2k$, denote $\rind(\mu, \zeta, \lambda)$ by $B$ and apply the following process:
    \begin{enumerate}
        \item Select  $i$ such that $i, \bar{i} \in B$.
        \item Replace $i$ and $\bar{i}$ with $i+1$ and $\overline{i+1}$.
        \item Repeat until no more pairs can be selected ensuring the modified pairs are not reused.
        \item Finally, add $k$ pairs of $1$ and $\bar{1}$.
    \end{enumerate}

    Let the resulting multi-set be $A$. By similar reasoning, $A$ corresponds to $\rind(T)$ for some ohs $T = (\mu, \nu, \lambda)$ of length $r$.
\end{proof}

\begin{lem}\label{lem: rind}
For a partition $\lambda$ and a nonnegative integer vector $\mu$, we have the following:
\begin{enumerate}
    \item There exists a bijection 
    \begin{equation*}
        \phi_r: \SSOT(\lambda,\mu) \rightarrow \HW(B_\mu(\sboxtwo),\lambda)
    \end{equation*}
    such that $c(T) = \epsilon_0(\phi_r(T))$ for $T \in \SSOT(\lambda,\mu)$.
    
    \item There exists a bijection 
    \begin{equation*}
        \phi_r: \GSSOT(\lambda,\mu) \rightarrow \HW(B_\mu(\sboxone),\lambda)
    \end{equation*}
    such that $2c(T) = \epsilon_0(\phi_r(T))$ for $T \in \GSSOT(\lambda,\mu)$.
    
    \item There exists a bijection 
    \begin{equation*}
        \phi_r: \SSROT(\lambda,\mu) \rightarrow \HW(B_\mu(\hspace{0.4mm}\sboxeleven\hspace{0.4mm}),\lambda)
    \end{equation*}
    such that $2c(T) = \epsilon_0(\phi_r(T))$ for $T \in \SSROT(\lambda,\mu)$.
\end{enumerate}
\end{lem}

\begin{proof}
Denote $T = (T_1, \dots, T_n)$, which is one of $\SSOT$, $\GSSOT$, or $\SSROT$.

For (1) and (2), the map is given by
$
    \phi_r(T) = \rind(\Gamma(T_n)) \otimes \dots \otimes \rind(\Gamma(T_1)),
$
and for (3), the map is given by
$
    \phi_r(T) = \rind(T_n) \otimes \dots \otimes \rind(T_1).
$

By Lemma~\ref{lem: rohs ohs connection}, in all cases, $\phi_r(T)$ corresponds to an $\SSROT$ (by reading off row indices on each rohs) and therefore a classical highest weight element. The bijectivity follows directly from~\eqref{eq: column row KR decompose}. It remains to establish the claim for $\epsilon_0(\phi_r(T))$. Denote $T_n = (\zeta, \nu, \lambda)$ and $T' = (T_1, \dots, T_{n-1})$.

(1) By~\eqref{eq: epsilon}, we have
\begin{equation*}
    \epsilon_0(\phi_r(T)) = \max\big(\epsilon_0(\phi_r(T')), \epsilon_0(\rind(\Gamma(T_n))) - \langle \wt(\phi_r(T')), \alpha_0^\vee \rangle\big).
\end{equation*}
Note that $\langle \wt(\phi_r(T')), \alpha_0^\vee \rangle = \langle \zeta, \alpha_0^\vee \rangle = -\zeta_1$. For $S \in B^{1, \mu_n}(\sboxtwo)$, represented as a multi-set, we have
\begin{equation*}
    \epsilon_0(S) = \text{number of 1's in } S + \frac{\mu_n - |S|}{2}.
\end{equation*}
Thus, $\epsilon_0(\Gamma(T_n)) = \nu_1 - \zeta_1$, which implies
\begin{equation*}
    \epsilon_0(\rind(\Gamma(T_n))) - \langle \wt(\phi_r(T')), \alpha_0^\vee \rangle = \nu_1.
\end{equation*}
The result follows by induction.

(2) Here, $\langle \wt(\phi_r(T')), \alpha_0^\vee \rangle = \langle \zeta, \alpha_0^\vee \rangle = -2\zeta_1$. For $S \in B^{1, \mu_n}(\sboxone)$, represented as a multi-set, we have
\begin{equation*}
    \epsilon_0(S) = 2 \times (\text{number of 1's in } S) + (\mu_n - |S|).
\end{equation*}
Therefore,
\begin{align*}
    \epsilon_0(\Gamma(T_n)) =
    \begin{cases*}
        2(\nu_1 - \zeta_1), & if $|\nu / \zeta| + |\nu / \lambda| = \mu_n$, \\
        2(\nu_1 - \zeta_1) + 1, & if $|\nu / \zeta| + |\nu / \lambda| = \mu_n - 1$.
    \end{cases*}
\end{align*}
The result follows by induction.

(3) Here, $\langle \wt(\phi_r(T')), \alpha_0^\vee \rangle = \langle \zeta, \alpha_0^\vee \rangle = -\zeta_1 - \zeta_2$. For $S \in B^{1, \mu_n}(\hspace{0.4mm}\sboxeleven\hspace{0.4mm})$, represented as a multi-set, we have
\begin{equation*}
    \epsilon_0(S) = (\text{number of 1's in } S) + \max\big(0, (\text{number of 2's in } S) - (\text{number of $\bar{2}$'s in } S)\big).
\end{equation*}
The result again follows by induction.
\end{proof}

\begin{example}
Let $\lambda = (2)$, $\mu = (3,3)$. 
Consider the following elements in $\SSOT(\lambda,\mu)$ and $\GSSOT(\lambda,\mu)$:
\begin{align*}
    T^{(1)} &= \big((\emptyset, \ydiagram{2}, \ydiagram{1}), (\ydiagram{1}, \ydiagram{2,1}, \ydiagram{2})\big) \in \SSOT(\lambda,\mu), \\
    T^{(2)} &= \big((\emptyset, \ydiagram{2}, \ydiagram{2}), (\ydiagram{2}, \ydiagram{2,1}, \ydiagram{2})\big) \in \GSSOT(\lambda,\mu).
\end{align*}

For $T^{(1)}$, the $\cind$ of the first ohs is $\{1,2,\bar{2}\}$, and for the second ohs, it is $\{1,2,\bar{1}\}$. Since we have $\red(12\bar{2}) = 1$ and $\red(12\bar{1}) = 2$, $T^{(1)}$ maps to $2 \otimes 1 \in B^{3,1} \otimes B^{3,1}$ via $\phi_c$. Similarly, for $T^{(2)}$, we have $\cind$ for the first ohs as $\{1,2\}$ and for the second ohs as $\{1,\bar{1}\}$, which maps to $\emptyset \otimes 12 \in B^{3,1} \otimes B^{3,1}$ via $\phi_c$.

On the other hand, for $T^{(1)}$, the $\rind$ of the first ohs is $\{1,1,\bar{1}\}$, and for the second ohs, it is $\{1,2,\bar{2}\}$. We have $\phi_r(T^{(1)})=11\bar{1} \otimes 1$.
For $T^{(2)}$, we have $\rind$ for the first ohs as $\{1,1\}$ and for the second ohs as $\{2,\bar{2}\}$. We obtain $\phi_r(T^{(2)})=1\bar{1} \otimes 11$. Note that we have $c(T^{(2)})=\frac{5}{2}$ and $\epsilon_0(\phi_r(T^{(2)}))=5$.
\end{example}

\section{Lusztig $q$-weight multiplicities}\label{Sec: L}
In this section, we state and prove a combinatorial formula for Lusztig $q$-weight multiplicities for type $C$ and type $B$ with spin weights (Theorem \ref{thm: C lusztig} and \ref{thm: B lusztig}).

\begin{thm}\label{thm: C lusztig}
For $\lambda, \mu \in \Par_n$, we have
\begin{equation*}
    \KL^{C_n}_{\lambda,\mu}(q) = \sum_{T \in \SSOT_{g}(\oc(\lambda,g),\overline{\oc}(\mu,g))} q^{\overline{D}(\phi_c(T))}
\end{equation*}
where $g$ is any positive integer such that $g \geq \lambda_1$ and $\phi_c$ is the map defined in Lemma \ref{lem: cind}.
\end{thm}

\begin{rmk}
In \cite{Lee2023}, Lee constructed a bijection between a set of King tableaux of shape $\lambda$ with weight $\mu$ and $\SSOT_{g}(\oc(\lambda,g),\overline{\oc}(\mu,g))$ for any $g \geq \lambda_1$. This bijection induces a bijection 
\begin{equation*}
    \iota: \SSOT_{g}(\oc(\lambda,g),\overline{\oc}(\mu,g)) \rightarrow  \SSOT_{g+1}(\oc(\lambda,g+1),\overline{\oc}(\mu,g+1))
\end{equation*}
defined as follows. Let $T = (T_1, T_2, \dots, T_n) \in \SSOT_{g}(\oc(\lambda,g),\overline{\oc}(\mu,g))$, and for each $T_i = (\tau, \nu, \eta)$ we define 
\[
T'_i = (\tau + (1^{i-1}, 0^{n+1-i}), \nu + (1^i, 0^{n-i}), \eta + (1^i, 0^{n-i})).
\] 
Then $\iota(T)=(T'_1,T'_2,\dots,T'_n)$. In Appendix \ref{appendix: g independnt}, we further demonstrate that (Proposition \ref{prop: energy g})
\begin{equation}\label{eq: energy g independent}
    \overline{D}(\phi_c(T)) = \overline{D}(\phi_c(\iota(T))). 
\end{equation}
This shows that the right-hand side of Theorem \ref{thm: C lusztig} does not depend on $g$.
\end{rmk}

\begin{rmk}\label{rmk: g ind}
% W \to C
Let $\alpha$ be any rearrangement of $\mu$, and denote $C$ as the composition of combinatorial $R$-matrices that maps $B_{\mu}^{t}$ to $B_{\alpha}^{t}$.
It is clear that $C$ maps $\HW(B_{\mu}^{t}, \lambda^{t})$ bijectively to $\HW(B_{\alpha}^{t}, \lambda^{t})$.
Consequently, we have
\begin{equation*}
    \sum_{T \in \SSOT(\lambda, \mu)} q^{\overline{D}(\phi_c(T))} = \sum_{T \in \SSOT(\lambda, \alpha)} q^{\overline{D}(\phi_c(T))}.
\end{equation*}

In Appendix \ref{appendix: invariance ssot g}, we further establish the refined identity (Proposition \ref{prop: invariance ssot})
\begin{equation*}
    \sum_{T \in \SSOT_g(\lambda, \mu)} q^{\overline{D}(\phi_c(T))} = \sum_{T \in \SSOT_g(\lambda, \alpha)} q^{\overline{D}(\phi_c(T))}
\end{equation*}
for any $g \geq \max(\mu_1, \dots, \mu_n)$. Thus, in the right-hand side of Theorem \ref{thm: C lusztig}, we may replace $\overline{\oc}(\mu, g)$ with any rearrangement of it. For the sake of convenience in the proof, we choose $\overline{\oc}(\mu, g)$.
\end{rmk}

\begin{example}\label{ex : C lusztig q=1}
When $n=1$, $\KL^{C_1}_{(\lambda_1),(\mu_1)}(q)$ is nonzero if and only if $\lambda_1-\mu_1$ is a nonnegative even integer, and in that case, it equals $q^{\frac{\lambda_1-\mu_1}{2}}$. The set $\SSOT_{g}((g-\lambda_1),(g-\mu_1))$ consists of a single SSOT $\big((\emptyset,(r) (s))\big)$ where $r=g-\lambda_1+\frac{\lambda_1-\mu_1}{2}$ and $s=g-\lambda_1$. Via $\phi_c$, it maps to $\red(1\dots r \overline{r}\dots\overline{s+1})=1\dots s \in B^{g-\mu_1,1}(\hspace{0.4mm}\sboxeleven\hspace{0.4mm})$. Note that for any $b\in B^{g-\mu_1,1}(\hspace{0.4mm}\sboxeleven\hspace{0.4mm})$, we have $\overline{D}(b)=\frac{g-\mu_1-j}{2}$ when $b\in B(\omega_j)$ by \eqref{eq: energy function for single KR crystla}. Therefore we conclude $\overline{D}(1\dots s)=\frac{\lambda_1-\mu_1}{2}$. We see that Theorem \ref{thm: C lusztig} holds for $n=1$.
\end{example}

\begin{example}\label{ex : C lusztig}
    Let $\lambda = (1,1,0,0)$, $\mu = (0,0,0,0)$, and $g = 1$. 
    We have $\KL^{C_4}_{\lambda,\mu}(q) = q^6 + q^4 + q^2$,
    and the set $\SSOT_g(\oc(\lambda,g), \overline{\oc}(\mu,g))$ consists of
    \begin{align*}
    T^{(1)} &= \big( (\emptyset , \ydiagram{1}, \ydiagram{1}), (\ydiagram{1} , \ydiagram{1}, \emptyset), (\emptyset , \ydiagram{1}, \ydiagram{1}), (\ydiagram{1} , \ydiagram{1,1} , \ydiagram{1,1}) \big),\\
    T^{(2)} &= \big( (\emptyset , \ydiagram{1}, \ydiagram{1}), (\ydiagram{1} , \ydiagram{1,1}, \ydiagram{1,1}), (\ydiagram{1,1} , \ydiagram{1}, \ydiagram{1}), (\ydiagram{1} , \ydiagram{1,1} , \ydiagram{1,1}) \big), \\
    T^{(3)} &= \big( (\emptyset , \ydiagram{1}, \ydiagram{1}), (\ydiagram{1} , \ydiagram{1,1}, \ydiagram{1,1}), (\ydiagram{1,1} , \ydiagram{1,1}, \ydiagram{1,1,1}), (\ydiagram{1,1,1} , \ydiagram{1,1} , \ydiagram{1,1}) \big).
    \end{align*}
    Via the map $\phi_c$, we embed each $T^{(i)}$ into $(B^{1,1}(\hspace{0.4mm}\sboxeleven\hspace{0.4mm}))^{\otimes 4}$ as follows with the energy function value:  
    \begin{align*}
        \phi_c(T^{(1)}) &= 1 \otimes 1 \otimes \bar{1} \otimes 1, \qquad \overline{D}(\phi_c(T^{(1)})) = 6,\\
        \phi_c(T^{(2)}) &= 1 \otimes \bar{1} \otimes 1 \otimes 1, \qquad\overline{D}(\phi_c(T^{(2)})) = 4,\\
        \phi_c(T^{(3)}) &= \bar{1} \otimes 1 \otimes 1 \otimes 1, \qquad\overline{D}(\phi_c(T^{(3)})) = 2.
    \end{align*}
    Therefore we have $\KL^{C_4}_{\lambda,\mu}(q) =\sum_{T\in\SSOT_1(\oc(\lambda,g), \overline{\oc}(\mu,g))} q^{\overline{D}(\phi_c(T))}$.
\end{example}

For Lusztig $q$-weight multiplicities for type $B$, we introduce a $q,t$-generalization. Recall that the set of positive roots $R^{+}$ of type $B$ consists of long roots $(\varepsilon_i \pm \varepsilon_j)$ for $i < j$, and short roots $\varepsilon_i$. 

Given dominant weights $\lambda$ and $\mu$, the \emph{$q,t$-weight multiplicities} are defined as follows:
\[
    \KL^{B_n}_{\lambda,\mu}(q,t) := \sum_{w \in W} (-1)^w [e^{w(\lambda + \rho) - (\mu + \rho)}] 
    \prod_{\substack{\alpha \in R^{+} \\ \text{long root } \alpha}} \frac{1}{1 - q e^\alpha} 
    \prod_{\substack{\alpha \in R^{+} \\ \text{short root } \alpha}} \frac{1}{1 - t e^\alpha}.
\]
Clearly, we can recover the usual Lusztig $q$-weight multiplicities by setting $t = q$.

Recall that $B^{r,1}(\sboxone)$ is decomposed as $ \oplus_{0 \leq s \leq r} B( \omega_s) $ as a classical crystal.
Then for $ v \in B^{r,1}(\sboxone) $, we define the \textit{vacancy} of $ v $, denoted by $\vac(v)$, as $ (r - s) $ if $ v \in B(\omega_s) $.
More generally, given $b=b_n\otimes \dots \otimes b_1\in B^{t}_{\mu}(\sboxone)$, we define $\vac(b)=\sum_{i=1}^{n}\vac(b_i)$. 
Next, we define the \textit{$q,t$-energy function} of $b$ by
\[
    \energy_{q,t}(b) := q^{\frac{\overline{D}(b) - \vac(b)}{2}} t^{\vac(b)}.
\]
It is clear that $\vac(b) = \vac(b')$ when $b$ and $b'$ lie in the same classical component. 
Furthermore, we have $\vac(u \otimes u') = \vac(R(u \otimes u'))$ for $u \otimes u' \in B^{r_2,1}(\sboxone) \otimes B^{r_1,1}(\sboxone)$.
This result follows from the rules outlined in Appendix \ref{Sec: append A}, which verify the claim for the case $r_1 = 1$.
By extension, we conclude that $\vac(u \otimes u') = \vac(S(u \otimes u'))$ where $S$ is the splitting map.
Therefore, the assertion is established by Corollary \ref{cor: splitting invariance}. 
To summarize, the $q,t$-energy function is invariant under the combinatorial $R$-matrix and remains constant within each classical component.

We now state the main result for Lusztig $q$-weight multiplicities for type $B$:

\begin{thm}\label{thm: B lusztig}
For $\lambda, \mu \in \Par_n$, we have
\[
    \KL^{B_n}_{\lambda^{\sharp}, \mu^{\sharp}}(q,t) = \sum_{T \in \GSSOT_{g + \frac{1}{2}}(\oc(\lambda,g), \overline{\oc}(\mu,g))} \energy_{q,t}(\phi_c(T)),
\]
where $g$ is any positive integer such that $g \geq \lambda_1$, and $\phi_c$ is the map in Lemma \ref{lem: cind}.
\end{thm}

As a set, for a positive integer $g$, we have the following decomposition (from the definition of $\GSSOT$)
\begin{equation}\label{eq: gssot decomposition}
    \GSSOT_{g + \frac{1}{2}}(\lambda, \mu) = \bigcup_{v} \SSOT_g (\lambda, \mu - v),
\end{equation}
where the union is taken over all 0-1 vectors $v$. Note that when $\mu-v$ contains a negative entry, we regard $\SSOT_g (\lambda, \mu - v)$ as an empty set. By Remark \ref{rmk: energy g ind type D_N+1}, the right-hand side of Theorem \ref{thm: B lusztig} is also independent of the choice of $g$.

\begin{rmk}\label{rmk: why non spin is hard}
    For $\lambda,\mu \in \Par_n$, $\KL^{B_n}_{\lambda,\mu}(q,t)$ is not a polynomial in $q,t$ with nonnegative coefficients in general. For example, we have
    \[
        \KL^{B_2}_{(1,0), (0,0)}(q,t) = qt - q + t.
    \]
    This example illustrates that for general weights (outside the spin case), finding a combinatorial formula may require a significantly different approach.
\end{rmk}

\begin{example}
    Let $\lambda=(1,1,1), \mu = (0,0,0)$, and $g=1$. 
    We have $\KL^{B_3}_{\lambda^\sharp,\mu^\sharp}(q,t) = q^3t^3 + q^3t + q^2t + qt$,
    and the set $\GSSOT_{1+\frac{1}{2}}(\oc(\lambda,g),\overline{\oc}(\mu,g))$ consists of:
    \begin{align*}
    T^{(1)} &= \big( (\emptyset , \emptyset, \emptyset), (\emptyset , \emptyset, \emptyset), (\emptyset , \emptyset, \emptyset) \big), \\
    T^{(2)} &= \big( (\emptyset , \emptyset,\emptyset), (\emptyset,\ydiagram{1},\ydiagram{1}), (\ydiagram{1},\emptyset,\emptyset) \big), \\
    T^{(3)} &= \big( (\emptyset , \ydiagram{1},\ydiagram{1}), (\ydiagram{1},\ydiagram{1},\emptyset), (\emptyset , \emptyset, \emptyset) \big), \\
    T^{(4)} &= \big( (\emptyset , \ydiagram{1},\ydiagram{1}), (\ydiagram{1},\ydiagram{1},\ydiagram{1}), (\ydiagram{1} , \ydiagram{1}, \emptyset) \big).
    \end{align*}
    Via the map $\phi_c$, we embed each $T^{(i)}$ into $(B^{1,1}(\sboxone))^{\otimes 3}$ as follows with the $q,t$-energy function value:
   \begin{align*}
       \phi_c(T^{(1)}) &= \emptyset \otimes \emptyset \otimes \emptyset, \qquad         \energy_{q,t}(\phi_c(T^{(1)})) = q^3t^3,\\
       \phi_c(T^{(2)}) &= \bar{1} \otimes 1 \otimes \emptyset, \qquad \energy_{q,t}(\phi_c(T^{(2)})) = q^3t, \\ 
       \phi_c(T^{(3)}) &= \emptyset \otimes \bar{1} \otimes 1, \qquad \energy_{q,t}(\phi_c(T^{(3)})) = q^2t,\\
       \phi_c(T^{(4)}) &= \bar{1} \otimes \emptyset \otimes 1, \qquad\energy_{q,t}(\phi_c(T^{(4)})) = qt.
   \end{align*}
    These values coincide with the computation $\KL^{B_3}_{\lambda^\sharp,\mu^\sharp}(q,t) =  q^3t^3 + q^3t + q^2t + qt$.
\end{example}

We mainly illustrate the proof of Theorem \ref{thm: C lusztig} (Section \ref{sub: deform morris} and \ref{sub: ohs jdt}). As the proof for Theorem \ref{thm: B lusztig} is parallel, we give a sketch (Section \ref{sub: B skecth}).

\subsection{Deforming Morris type recurrence formula}\label{sub: deform morris}
Given $\lambda, \mu \in \Par_n$ for $n \geq 2$, define $\lambda^{(1)}, \lambda^{(2)}, \dots, \lambda^{(n)}$ and $\mu'$ as follows:
\begin{align}
    &\lambda^{(1)} = (\lambda_2, \ldots, \lambda_n), \hspace{40mm} \mu' = (\mu_2, \ldots, \mu_n), \label{eq: lambdais} \\
    &\lambda^{(2)} = (\lambda_1 + 1, \lambda_3, \ldots, \lambda_n), \notag \\
    &\lambda^{(3)} = (\lambda_1 + 1, \lambda_2 + 1, \lambda_4, \ldots, \lambda_n), \notag \\
    &\vdots \notag \\
    &\lambda^{(n)} = (\lambda_1 + 1, \lambda_2 + 1, \ldots, \lambda_{n-1} + 1). \notag
\end{align}

Additionally, for a partition $\lambda$ with $\ell(\lambda)\leq m$ we define $\ROHS_{\leq m}(\lambda, r)$ to be the set of rohs $T$ of length $r$ with the initial shape $\lambda$ such that $\ell(F(T))\leq m$ (note that this is distinct from the set of $m$-bounded rohs').

As proved in \cite[Theorem 3.2.1]{Lecouvey2005}, Morris type recurrence formula for type $C$ is given by
\begin{equation}\label{eq: C morris recurrence}
    \KL^{C_n}_{\lambda,\mu}(q) = \sum_{i \geq 1} (-1)^{i-1} \sum_{\substack{r + 2m = \lambda_i - \mu_1 + 1 - i \\ r, m \geq 0}} q^{r+m} \sum_{(\lambda^{(i)}, \tau, \nu) \in \operatorname{ROHS}_{\leq n-1}(\lambda^{(i)}, r)} \KL^{C_{n-1}}_{\nu, \mu'}(q).
\end{equation}
We proceed with induction on $n$ to prove Theorem \ref{thm: C lusztig}. The base case $n=1$ has been covered in Example \ref{ex : C lusztig q=1}. Now assume Theorem \ref{thm: C lusztig} is true for $\KL^{C_{n-1}}_{\lambda, \mu}$'s. Then the right-hand side of \eqref{eq: C morris recurrence} equals
\begin{equation}\label{noname}
    \sum_{i \geq 1} (-1)^{i-1} \sum_{\substack{r + 2m = \lambda_i - \mu_1 + 1 - i \\ r, m \geq 0}} q^{r+m} \sum_{(\lambda^{(i)}, \tau, \nu) \in \operatorname{ROHS}_{\leq n-1}(\lambda^{(i)}, r)} \left( \sum_{T \in \SSOT_g(\oc(\nu, g), \overline{\oc}(\mu', g))} q^{\overline{D}(\phi_c(T))} \right)
\end{equation}
where we take a sufficiently large $g$. Clearly $g>\nu_1$ for every $\nu$ appearing in \eqref{noname}, therefore $\oc(\nu,g)$ is a partition. We take the orthogonal complement of an rohs $(\lambda^{(i)},\tau,\nu)$ to obtain an ohs $(\oc(\nu,g),\oc(\tau,g),\oc(\lambda^{(i)},g))$ and append to $T\in \SSOT_g(\oc(\nu,g),\overline{\oc}(\mu',g))$ to obtain a new SSOT. The precise statement is as follows:

\begin{lem}\label{lem: add rohs}
For $1\leq i\leq n$, let $A^{(i)}$ be a set of pairs $(S,T)$ satisfying:
    \begin{itemize}
        \item $S\in \ROHS_{\leq n-1}(\lambda^{(i)},r)$ for some $r$ such that there exists $m\geq 0$ with $r+2m=\lambda_i-\mu_1+1-i$,
        \item $T\in\SSOT_g(\oc(\nu,g),\overline{\oc}(\mu',g))$ where $\nu=F(S)$.
    \end{itemize}
    Then there exists a bijection 
    \begin{equation*}
        \Phi^{(i)}: A^{(i)}\rightarrow \SSOT_{g}(\oc(\lambda^{(i)},g),\gamma^{(i)})
    \end{equation*}
    where $\gamma^{(i)}=(g-\mu_2,\dots,g-\mu_n,\lambda_i-\mu_1+1-i)$.
\end{lem}
\begin{proof}
    Pick $(S,T)\in A^{(i)}$ and denote $S=(\lambda^{(i)},\zeta,\nu)$.
    Since $g$ is sufficiently large, the integer vector $(\oc(\zeta,g),m)=(g-\zeta_{n-1},\dots,g-\zeta_1,m)$ is a partition. Then $(\oc(\nu,g),(\oc(\zeta,g),m),\oc(\lambda^{(i)},g))$ is an ohs of length $r+2m$ and appending it to $T$ yields an element of $\SSOT_{g}(\oc(\lambda^{(i)},g),\gamma^{(i)})$ which we set as $\Phi^{(i)}(S,T)$.
    
    Obviously, $\Phi^{(i)}$ is an injective map.
    Conversely, given $\overline{T}\in \SSOT_{g}(\oc(\lambda^{(i)},g),\gamma^{(i)})$, pick the last ohs of $\overline{T}$ denoted by $(\nu,\zeta,\oc(\lambda^{(i)},g))$.
    As $\ell(\nu)\leq n-1$, letting $\zeta'=(\zeta_1,\dots,\zeta_{n-1})$, $(\nu,\zeta',\oc(\lambda^{(i)},g))$ is an ohs of length $(\lambda_i-\mu_1+1-i-2\zeta_n)$.
    Then setting $S=(\lambda^{(i)},\oc(\zeta',g),\oc(\nu,g))$ and $T$ to be the SSOT obtained by deleting the last ohs from $\overline{T}$, we have $\Phi^{(i)}(S,T)=\bar{T}$.
\end{proof}

\begin{lem}\label{lem: energy r+m}
Let $A^{(i)}$ be the set in Lemma \ref{lem: add rohs}. For $(S,T)\in A^{(i)}$ such that $S\in \ROHS_{\leq n-1}(\lambda^{(i)},r)$, we have
    \begin{equation*}
        \overline{D}(\phi_c(\Phi^{(i)}(S,T)))=r+m+\overline{D}(\phi_c(T)).
    \end{equation*}
    where $m$ is given by $r+2m=\lambda_i-\mu_1+1-i$.
\end{lem}
\begin{proof}
Denote $S=(\lambda^{(i)},\zeta,\nu)$. Then the last ohs of $\Phi^{(i)}(S,T)$ given as $(\oc(\nu,g),(\oc(\zeta,g),m),\oc(\lambda^{(i)},g))$ and we denote it by $Y$. Note that $\cind(Y)$ is of the form (represented as a word) $1\dots m w_1\dots w_r\bar{m}\dots \bar{1}$ where each $w_i\succ g-\nu_1$. Since $g$ is sufficiently large, the word $w_1\dots w_r$ is admissible, therefore $\red(\cind(Y))=w_1\dots w_r$. Now the proof then follows from Proposition \ref{prop: B energy r+m}.
\end{proof}
The proof of Proposition \ref{prop: B energy r+m} requires an intricate argument related to the splitting map, postponed to Appendix \ref{appendix: energy increase}.

By Lemma \ref{lem: add rohs} and \ref{lem: energy r+m}, \eqref{noname} becomes
\begin{equation}
     \sum_{i\geq1}(-1)^{i-1}\sum_{T\in\SSOT_{g}(\oc(\lambda^{(i)},g),\gamma^{(i)})}q^{\overline{D}(\phi_c(T))}
\end{equation}
where $\gamma^{(i)}=(g-\mu_2,\dots,g-\mu_n,\lambda_i-\mu_1+1-i)$.
We may freely rearrange $\gamma^{(i)}$ as proved in Proposition \ref{prop: invariance ssot} and we choose $\beta^{(i)}=(\lambda_i-\mu_1+1-i,g-\mu_2,\dots,g-\mu_n)$ among rearrangements of $\gamma^{(i)}$.
We finally have
\begin{equation}\label{noname2}
     \sum_{i\geq1}(-1)^{i-1}\sum_{T\in\SSOT_{g}(\oc(\lambda^{(i)},g),\beta^{(i)})}q^{\overline{D}(\phi_c(T))}.
\end{equation}

\subsection{Oscillating horizontal strip jeu de taquin}\label{sub: ohs jdt}
Now our goal is to prove
\begin{equation}\label{eq: goal}
     \sum_{i\geq1}(-1)^{i-1}\sum_{T\in\SSOT_{g}(\oc(\lambda^{(i)},g) , \beta^{(i)})} q^{\overline{D}(\phi_c(T))}=\sum_{T\in\SSOT_{g}(\oc(\lambda,g),\overline{\oc}(\mu,g))}q^{\overline{D}(\phi_c(T))}.
\end{equation}
We will construct a sign-reversing involution on $\bigcup_{i\geq 1}\SSOT_{g}(\oc(\lambda^{(i)},g) , \beta^{(i)})$ such that the set of fixed elements is a subset of $\SSOT_{g}(\oc(\lambda^{(1)},g) , \beta^{(1)})$ which is in bijection with $\SSOT_{g}(\oc(\lambda,g),\overline{\oc}(\mu,g))$. The fact that there is an injection from $\SSOT_{g}(\oc(\lambda,g),\overline{\oc}(\mu,g))$ to $\SSOT_{g}(\oc(\lambda^{(1)},g) , \beta^{(1)})$ is not trivial, therefore proved by Corollary \ref{cor: ohs jdt last row}. In particular, our argument generalizes \cite{Zab1998}.
To that end, we introduce the map $\Aug$ which can be regarded as jeu de taquin for an $\SSOT$. 
In this section, $B^{t}_{\mu}=B^{t}_{\mu}(\hspace{0.4mm}\sboxeleven\hspace{0.4mm})$ unless otherwise stated.

\begin{definition}
    For $T\in \SSOT(\lambda,\mu)$, consider $\phi_c(T)\in \HW(B^{t}_{\mu},\lambda^{t})$ denoted as $b_{n}\otimes\dots \otimes b_2 \otimes b_1$. 
    Then $b_1\in B^{\mu_1,1}$ is of the form $1\dots k$ and let $b'_1=1\dots (k+r)\in B^{\mu_1+r,1}$ for any nonnegative integer $r$. 
    We define $\Aug(T,r)$ to be $T'\in \SSOT(\lambda,(\mu_1+r,\mu_2,\dots,\mu_n))$ such that $\phi_c(T)=\hw(b_n\otimes \dots \otimes b_2\otimes b'_1)$.
\end{definition}

We now present a recursive method to compute $\Aug(T, r)$ with an elegant combinatorial description (Proposition \ref{prop: computation of aug}). 
To achieve this, we define several additional terminologies.

For an ohs $T = (\mu, \nu, \lambda)$ of length $r$, let $\cind(T) = a_1 \dots a_s \bar{b}_{s+1} \dots \bar{b}_r$, where $a_1 \prec \dots \prec a_s \prec \bar{b}_{s+1} \prec \dots \prec \bar{b}_r$. The \emph{standardization} of $T$, denoted by $\std(T)$, is a sequence of partitions $S = (S^{(0)}, S^{(1)}, \dots, S^{(r)})$ defined as follows:  
\begin{itemize}
    \item Set $S^{(0)} = \mu$.  
    \item For $i \leq s$, $S^{(i)}$ is obtained from $S^{(i-1)}$ by adding a cell in the $a_i$-th column.  
    \item For $i > s$, $S^{(i)}$ is obtained from $S^{(i-1)}$ by removing a cell in the $b_i$-th column.
\end{itemize}

We also define
\[
\SP(T) := \bar{b}_r \otimes \dots \otimes \bar{b}_{s+1} \otimes a_s \otimes \dots \otimes a_1 \in (B^{1,1})^{\otimes r}.
\]
For example, given the ohs $T$ in Example \ref{ex: ohs cind}, we have
\[
\std(T) = (\ydiagram{2,1}, \ydiagram{2,2}, \ydiagram{3,2}, \ydiagram{3,1}, \ydiagram{3}) \quad \text{and} \quad
\SP(T) = \bar{1} \otimes \bar{2} \otimes 3 \otimes 2.
\]

For partitions $\lambda,\mu$ and $\zeta$ such that:
\begin{equation}
    \text{(1) $\zeta / \mu$ is a horizontal strip  \hspace{15mm} (2) $|\lambda / \mu|=1$ or $|\mu / \lambda|=1$} \label{eq: FG condition},
\end{equation}
we define $\FG(\mu,\lambda;\zeta)$\footnote{The notation $\FG$ stands for Fomin's growth diagram due to the similar construction.} as follows. 
If $|\lambda / \mu|=1$, let $a$ be the column index of a cell $\lambda / \mu$. Then choose the largest integer $b$ satisfying
\begin{itemize}
    \item $b\leq a$, 
    \item there exists a partition $\eta$ such that $|\eta / \zeta|=1$ and $\eta / \zeta$ is a cell in the $b$-th column. 
\end{itemize}
We set $\FG(\mu,\lambda;\zeta)=\eta$. If $|\mu / \lambda|=1$, let $a$ be the column index of a cell $\mu / \lambda$. Then choose the smallest integer $b$ satisfying 
\begin{itemize}
    \item $b\geq a$, 
    \item there exists a partition $\eta$ such that $|\zeta / \eta|=1$ and $\zeta / \eta$ is a cell in the $b$-th column. 
\end{itemize}
We set $\FG(\lambda,\mu;\zeta)=\eta$. Clearly there cannot be distinct partitions $\mu$ and $\mu'$ (satisfying conditions in \eqref{eq: FG condition}) with $\FG(\mu,\lambda;\zeta)=\FG(\mu',\lambda;\zeta)$.

\begin{lem}\label{lem: ohs useful1}
For partitions $\lambda, \mu$, and $\zeta$ satisfying \eqref{eq: FG condition}, the skew shape $\FG(\mu, \lambda; \zeta) / \lambda$ is also a horizontal strip.
\end{lem}

\begin{proof}
Let $\eta = \FG(\mu, \lambda; \zeta)$, then it suffices to show $\lambda_{i+1} \leq \eta_{i+1} \leq \lambda_i$ for all $i$. 
We will focus on the case $|\lambda / \mu| = 1$, noting that the case $|\mu / \lambda| = 1$ can be handled similarly.

Suppose there exists an $i$ such that $\lambda_{i+1} > \eta_{i+1}$. 
In this case, we must have $\mu_{i+1} = \zeta_{i+1} = \eta_{i+1}$ and $\lambda_{i+1} = \mu_{i+1} + 1$. 
However, this leads to a contradiction, as the conditions $\mu_{i+1} = \zeta_{i+1}$ and $\lambda_{i+1} = \mu_{i+1} + 1$ imply $\eta_{i+1} = \zeta_{i+1} + 1$, which is impossible.

Next, suppose there exists an $i$ such that $\eta_{i+1} > \lambda_i$.
In this case, we must have $\zeta_{i+1} = \lambda_i = \mu_i$ and $\eta_{i+1} = \zeta_{i+1} + 1$. 
For a $j$ such that $\lambda_j = \mu_j + 1$, the condition $\eta_{i+1} = \zeta_{i+1} + 1$ requires $\zeta_{i+1} + 1 \leq \mu_j + 1 \leq \zeta_i$. 
The inequality $\zeta_{i+1} + 1 \leq \mu_j + 1$ is equivalent to $\lambda_i + 1 \leq \lambda_j$, which implies $j < i$.
However, it is impossible for $\mu_j + 1 \leq \zeta_i$ when $j < i$.
\end{proof}

Now consider a sequence of partitions $S = (\lambda^{(0)}, \lambda^{(1)}, \dots, \lambda^{(m)})$, where $|\lambda^{(i+1)} / \lambda^{(i)}| = 1$ or $|\lambda^{(i)} / \lambda^{(i+1)}| = 1$ for all $i$. 
For a partition $\zeta$ such that $\zeta / \lambda^{(0)}$ is a horizontal strip, we abuse the notation and also denote $\FG(S; \zeta)$ by the sequence of partitions $(\zeta^{(0)}, \zeta^{(1)}, \dots, \zeta^{(m)})$, which is constructed recursively as follows: 
\begin{itemize}
    \item Set $\zeta^{(0)} = \zeta$.
    \item If $\zeta^{(i)}$ is defined, let $\zeta^{(i+1)} = \FG(\lambda^{(i)}, \lambda^{(i+1)}; \zeta^{(i)})$.
\end{itemize}

By Lemma \ref{lem: ohs useful1}, the skew shape $\zeta^{(i)} / \lambda^{(i)}$ is a horizontal strip throughout this process.
Consequently, the partitions $\lambda^{(i)}, \lambda^{(i+1)}$, and $\zeta^{(i)}$ satisfy \eqref{eq: FG condition}, ensuring that $\FG(\lambda^{(i)}, \lambda^{(i+1)}; \zeta^{(i)})$ is well-defined.

\begin{proposition}\label{prop: computation of aug}
For $T=(T_1,\dots,T_n)\in \SSOT(\lambda,\mu)$, let $S$ be a sequence of partitions obtained by gluing the sequences $\std(T_2),\dots,\std(T_n)$.
Let $F(T_1)=(k)$ for some nonnegative integer $k$. Then for a nonnegative integer $r$, $\FG(S;(k+r))$ is a sequence obtained by gluing sequences $\std(\tilde{T}_2),\dots,\std(\tilde{T}_n)$, where $\tilde{T}=(\tilde{T}_1 , \dots , \tilde{T}_n)$ is $\Aug(T,r)$.
\end{proposition}

\begin{proof}
Consider $ L \in (B^{1,1})^{\otimes \sum_{i=2}^{n} \mu_i} \otimes B^{\mu_1+r,1} $ given by 
\[
    L = \SP(T_n) \otimes \dots \otimes \SP(T_2) \otimes 1 \dots (k+r).
\]
Since the map $\red$ commutes with classical crystal operators, $\hw(L)$ is given by 
\[
    \SP(\tilde{T}_n) \otimes \dots \otimes \SP(\tilde{T}_2) \otimes 1 \dots (k+r).
\]
Thus, it suffices to show the computation of $\hw(L)$ aligns with the recursive construction of $\FG(S; (k+r))$.

For any KR crystal $ B $ of the same type (i.e., $\diamond = \sboxeleven$), let $ v \in \HW(B, \zeta^t) $ and $ \tilde{v} \in \HW(B, \tilde{\zeta}^t) $, where $\zeta  \subset\tilde{\zeta}$ are partitions such that $\tilde{\zeta} / \zeta$ forms a horizontal strip. 
Suppose we are given a classical highest weight element of the form 
\[
    a_m \otimes \dots \otimes a_1 \otimes v \in (B^{1,1})^{\otimes m} \otimes B.
\]

First consider the case where $ a_1 $ is an unbarred letter.
Let $\eta$ be the partition obtained from $\zeta$ by adding a cell in the $a_1$-th column, and let $ b_1 $ be the column index of the cell $\FG(\zeta, \eta; \tilde{\zeta}) / \tilde{\zeta}$. 
Then we have 
\[
    b_1 \otimes \tilde{v} = \hw(a_1 \otimes \tilde{v}),
\]
and $ b_1 \otimes \tilde{v} $ is obtained by applying crystal operators $ e_{b_1} \dots e_{a_1-1} $ to $ a_1 \otimes \tilde{v} $. 
We claim that applying $ e_{b_1} \dots e_{a_1-1} $ to $ a_m \otimes \dots \otimes a_1 \otimes \tilde{v} $ yields $ a_m \otimes \dots \otimes a_2 \otimes b_1 \otimes \tilde{v} $. To prove this, it suffices to show that 
\[
    e_s \left( a_m \otimes \dots \otimes a_2 \otimes (s+1) \otimes \tilde{v} \right) = a_m \otimes \dots \otimes a_2 \otimes s \otimes \tilde{v},
\]
for $ b_1 \leq s \leq a_1 - 1 $. 
When $ a_1 = b_1 $, the claim is vacuously true, so assume $ b_1 < a_1 $.
In this case, we have $\langle \wt(a_1 \otimes v), \alpha_s \rangle = 0$, which implies 
$\varphi_s(a_1 \otimes v) = \langle \wt(a_1 \otimes v), \alpha_s \rangle + \epsilon_s(a_1 \otimes v) = 0.$
Since $ a_m \otimes \dots \otimes a_2 \otimes a_1 \otimes v $ is already a classical highest weight element, by \eqref{eq: ei operator} it follows that $\varphi_s(a_1 \otimes v) \geq \epsilon_s(a_m \otimes \dots \otimes a_2).$
As $\epsilon_s(a_m \otimes \dots \otimes a_2) = 0$, the claim is proved by \eqref{eq: ei operator}.

Now assume that $ a_1 = \bar{c} $ is a barred letter.
Let $\eta$ be the partition obtained from $\zeta$ by removing a cell from the $c$-th column, and let $ d $ be the column index of the cell $\tilde{\zeta} / \FG(\zeta, \eta; \bar{\zeta})$. 
Then, we have 
\[
    \bar{d}\otimes \tilde{v} = \hw(\bar{c} \otimes \tilde{v}),
\]
and $ \bar{d}\otimes \tilde{v} $ is obtained by applying $e_{d-1} \dots e_c $ to $ a_1 \otimes \tilde{v} $.
Using the same reasoning as in the first case, applying these operators to $a_m\otimes \dots \otimes a_1 \otimes \tilde{v}$ yields $a_m\otimes \dots a_2\otimes \bar{d} \otimes \tilde{v}$.

By iteratively applying the above steps, we can compute $\hw(L)$ recursively, following the same process as the computation of $\FG(S; (k+r))$.
\end{proof}

\begin{example}
    Consider the SSOT $T = (T_1, T_2, T_3)$ given by
    \begin{equation*}
        T = \big((\emptyset, \ydiagram{2}, \ydiagram{2}), (\ydiagram{2}, \ydiagram{4}, \ydiagram{3}), (\ydiagram{3}, \ydiagram{5}, \ydiagram{2})\big).
    \end{equation*}
    Let $S$ be the sequence of partitions formed by gluing $\std(T_2)$ and $\std(T_3)$
    \begin{equation*}
        S = (\ydiagram{2}, \ydiagram{3}, \ydiagram{4}, \ydiagram{3}, \ydiagram{4}, \ydiagram{5}, \ydiagram{4}, \ydiagram{3}, \ydiagram{2}).
    \end{equation*}
    Then $\FG(S; (7))$ equals
    \begin{align*}
        (\ydiagram{7}, \ydiagram{7,1}, \ydiagram{7,2}, \ydiagram{6,2}, \\
        \ydiagram{6,3}, \ydiagram{6,4}, \ydiagram{5,4}, \ydiagram{5,3}, \ydiagram{5,2}),
    \end{align*}
    which is obtained by gluing $\std(\tilde{T}_2)$ and $\std(\tilde{T}_3)$, where
    \begin{align*}
        \tilde{T}_2 = (\ydiagram{7}, \ydiagram{7,2}, \ydiagram{6,2}), \\
        \tilde{T}_3 = (\ydiagram{6,2}, \ydiagram{6,4}, \ydiagram{5,2}).
    \end{align*}
    Hence, we have
    \begin{equation*}
        \Aug(T, 5) = \big((\emptyset, \ydiagram{7}, \ydiagram{7}), \tilde{T}_2, \tilde{T}_3\big).
    \end{equation*}
\end{example}

To complete the proof of \eqref{eq: goal}, we also require the following technical lemmas (Lemma \ref{lem: ohs useful2} and \ref{lem: ohs useful3}) and their corollaries (Corollary \ref{cor: ohs jdt last row} and \ref{cor: ohs jdt1}).

\begin{lem}\label{lem: ohs useful2}
Let $T = (\mu, \nu, \lambda)$ be an ohs with $1 \leq n = \ell(\mu) < \ell(\lambda)$.
Define the following partitions
\begin{align*}
    \mu' &= (\mu_1, \dots, \mu_{n-1}, \mu_n - r), \quad
    \nu' = (\nu_1, \dots, \nu_n, \nu_{n+1} - r), \quad\text{and}\quad
    \lambda' = (\lambda_1, \dots, \lambda_n, \lambda_{n+1} - r)
\end{align*}
for some $0 \leq r \leq \lambda_{n+1}$. Then, we have $\FG(\std(\mu', \nu', \lambda'); \mu) = \std(T)$.
\end{lem}

\begin{proof}
    This follows directly from the construction of $\FG$.
\end{proof}

\begin{corollary}\label{cor: ohs jdt last row}
    For $T\in \SSOT(\lambda,\alpha)$ such that $\ell(\lambda)=n$ and $\alpha$ is a integer vector of length $n$, there exists a unique $\SSOT$ $T'$ of the shape $(\lambda_1,\dots,\lambda_{n-1})$ such that $\Aug(T',\lambda_n)=T$.
\end{corollary}
\begin{proof}
    Denote $T=(T_1,\dots,T_n)$, then for each $T_i=(\mu,\nu,\tau)$ we must have $\ell(\mu)=i-1$ and $\ell(\tau)=i$ with $\tau_i\geq \lambda_n$. For $i\geq 2$, let $T'_i=(\mu',\nu',\tau')$ where 
    \begin{align*}
        \mu'=(\mu_1,\dots,\mu_{i-2},\mu_{i-1}-\lambda_n),\quad \nu'=(\nu_1,\dots,\nu_{i-1},\nu_{i}-\lambda_n),\quad \text{and}\quad
        \tau'=(\tau_1,\dots,\tau_{i-1},\tau_{i}-\lambda_n).
    \end{align*}
    Additionally, let $T'_1$ be the unique ohs of length $(\alpha_1 - \lambda_n)$ such that $I(T'_1) = \emptyset$ and $F(T'_1) = k - \lambda_n$, where $F(T_1) = k$.
    Then, the SSOT $T' = (T'_1, \dots, T'_n)$ satisfies $\Aug(T', \lambda_n) = T$ by the successive application of Lemma \ref{lem: ohs useful2}. 
    The uniqueness follows immediately.
\end{proof}
Consider partitions $\zeta$, $\eta$, and $\lambda$ such that $|\zeta / \eta| = 1$ or $|\eta / \zeta| = 1$. We seek to determine whether there exists a partition $\mu$ (with $|\lambda / \mu| = 1$ or $|\mu / \lambda| = 1$) satisfying $\FG(\mu, \lambda; \zeta) = \eta$. More generally, given a sequence $S = (\zeta^{(0)}, \zeta^{(1)}, \dots, \zeta^{(m)})$, where $|\zeta^{(i+1)} / \zeta^{(i)}| = 1$ or $|\zeta^{(i)} / \zeta^{(i+1)}| = 1$ for all $i$, and a partition $\lambda$, we aim to determine whether there exists a sequence $S'$ such that $\FG(S'; \zeta^{(0)}) = S$ and the final partition of $S'$ is $\lambda$. This question is addressed by Lemma~\ref{lem: ohs useful3} and Corollary~\ref{cor: ohs jdt1}.
\begin{lem}\label{lem: ohs useful3}
    Let $\zeta$ and $\eta$ be partitions such that $|\zeta /\eta|=1$ or $|\eta / \zeta|=1$. 
    The following statements hold.
    \begin{enumerate}
        \item Given a partition $\lambda$, if $\zeta_1 + 1 = \eta_1$, there exists a partition $\mu$ such that $\FG(\mu, \lambda; \zeta) = \eta$ if and only if $\lambda_1 = \eta_1$ and $\eta / \lambda$ is a horizontal strip. 
        Otherwise, such a partition $\mu$ exists if and only if $\eta / \lambda$ is a horizontal strip.
        \item Consider partitions $\lambda$ and $\lambda'$ such that both $\eta / \lambda$ and $\eta / \lambda'$ are horizontal strips.
        Assume further that these horizontal strips have the same maximal column indices, i.e., their rightmost cells coincide.
        If there exists a partition $\mu$ such that $\FG(\mu, \lambda; \zeta) = \eta$, then there is also a partition $\mu'$ such that $\FG(\mu', \lambda'; \zeta) = \eta$ (the existence of $\mu$' is ensured by (1)).
        In this case, horizontal strips $\zeta / \mu$ and $\zeta / \mu'$ have the same maximal column indices.
    \end{enumerate}
\end{lem}

\begin{proof}
    (1) Assume $|\eta / \zeta| = 1$ and let $a$ be the column index of this cell. 
    To find a partition $\mu$ such that $\FG(\mu, \lambda; \zeta) = \eta$, we need to select the smallest $b$ such that:
    \begin{itemize}
        \item $b \geq a$,
        \item deleting a cell from the $b$-th column of $\lambda$ results in a valid partition $\mu$, and $\zeta/\mu$ is a horizontal strip.
    \end{itemize}
    When $\zeta_1 + 1 = \eta_1$, we have $a = \eta_1$, and such $b$ does not exist if $\lambda_1 < \eta_1$.
    Assuming $\lambda_1 = \eta_1$ and $\eta / \lambda$ is a horizontal strip, we simply take $b=a$.
    Otherwise, if $a < \eta_1$, it is straightforward to check that a suitable $b$ exists when $\eta / \lambda$ is a horizontal strip.

    Now, assume $|\zeta / \eta| = 1$ and let $a$ be the column index of this cell. 
    We need to find the largest $b$ such that:
    \begin{itemize}
        \item $b \leq a$,
        \item adding a cell to the $b$-th column of $\lambda$ results in a valid partition $\mu$, and $\zeta/\mu$ is a horizontal strip.
    \end{itemize}
    Again it is easy to see that such $b$ always exists when $\eta / \lambda$ is a horizontal strip.

    (2) The process described in the proof of (1) outlines how to find $\mu$ such that $\FG(\mu, \lambda; \zeta) = \eta$ when given $\lambda$, $\zeta$, and $\eta$. 
    From this process, it is clear that if two horizontal strips $\eta / \lambda$ and $\eta / \lambda'$ share the same maximal column indices, horizontal strips $\zeta / \mu$ and $\zeta / \mu'$ will also have the same maximal column indices.
\end{proof}

\begin{corollary}\label{cor: ohs jdt1}
    Let $T$ be an $\SSOT$ of shape $\lambda$, and let $\Aug(T, r)$ be of shape $\tau$.
    For any partition $\lambda'$ such that the two horizontal strips $\tau / \lambda$ and $\tau / \lambda'$ have the same maximal column indices, there exists a unique $\SSOT$ $T'$ of shape $\lambda'$ satisfying
    \begin{equation*}
        \Aug(T, r) = \Aug(T', r'),
    \end{equation*}
    where $r' = |\tau| - |\lambda'|$.
\end{corollary}

\begin{proof}
    Let $T = (T_1, T_2, \dots, T_n)$ and let $S$ be the sequence of partitions obtained by gluing the sequences $\std(T_2), \dots, \std(T_n)$. 
    By Proposition \ref{prop: computation of aug}, $\Aug(T, r)$ is obtained from $\FG(S; (k + r))$, where $F(T_1) = (k)$.

    The last partition in the sequence $\FG(S; (k + r))$ is $\tau$. 
    By applying Lemma \ref{lem: ohs useful3} consecutively, there exists a unique sequence of partitions $S'$ such that:
    \begin{itemize}
        \item the last partition in $S'$ is $\lambda'$,
        \item $\FG(S'; (k + r')) = \FG(S; (k + r))$.
    \end{itemize}
    Therefore, the desired $\SSOT$ $T'$ can be uniquely constructed from $S'$.
\end{proof}

\begin{proof}[Proof of \eqref{eq: goal}]
Recall that we aim to prove
\begin{equation}\label{b}
    \sum_{i \geq 1} (-1)^{i-1} \sum_{T \in \SSOT_{g}(\oc(\lambda^{(i)}, g), \beta^{(i)})} q^{\overline{D}(T)}= \sum_{T \in \SSOT_{g}(\oc(\lambda, g), \overline{\oc}(\mu, g))} q^{\overline{D}(T)} ,
\end{equation}
where we use the shorthand $ \overline{D}(T) $ to denote $ \overline{D}(\phi_c(T)) $.

To proceed, we define sets $ G_1, \dots, G_n $ given by
\begin{equation*}
    G_i = \{ \Aug(T, g - \lambda_i + i - 1) : T \in \SSOT_g(\oc(\lambda^{(i)}, g), \beta^{(i)}) \}.
\end{equation*}
Note that for any $ T \in \SSOT_g(\oc(\lambda^{(i)}, g), \beta^{(i)})$, $ \Aug(T, g - \lambda_i + i - 1) $ is an $\SSOT$ with weight $ (g - \mu_1, \dots, g - \mu_n) = \overline{\oc}(\mu, g) $. Additionally during the computation of $ \Aug(T, g - \lambda_i + i - 1) $ there cannot be a newly created cell with column index $g+1$, hence we have $c(\Aug(T, g - \lambda_i + i - 1))\leq g$.
By Proposition \ref{prop: aug preserves energy}, we have the following equality for $ T \in \SSOT(\oc(\lambda^{(i)}, g), \beta^{(i)}) $
\begin{equation}\label{eq: augg}
    \overline{D}(T) = \overline{D}(\Aug(T, g - \lambda_i + i - 1)).
\end{equation}

We now refine each set $ G_i $ into two subsets, $ G_i^{(1)} $ and $ G_i^{(2)} $. 
For each $ T \in G_i $, let $ \tau $ be the shape of $ T $.
If $ \tau_{n - i + 1} \geq g - \lambda_i $, then $ T \in G_i^{(1)} $; otherwise, $ T \in G_i^{(2)} $. 
We claim that
\begin{equation}\label{a}
    G_i^{(2)} = G_{i+1}^{(1)}.
\end{equation}

To prove this claim, consider $ T = \Aug(T', g - \lambda_i + i - 1) \in G_i^{(2)} $, where $ \tau $ is the shape of $ T $. 
Then the skew shape $ \tau / \oc(\lambda^{(i)}, g) $ is a horizontal strip, and the condition $ \tau_{n - i + 1} \leq g - \lambda_i - 1 $ implies that $ \tau / \oc(\lambda^{(i+1)}, g) $ is also a horizontal strip. These two horizontal strips $ \tau / \oc(\lambda^{(i)}, g) $ and $ \tau / \oc(\lambda^{(i+1)}, g) $ differ at cells with column indices no greater than $g-\lambda_i-1$. As the size $| \tau / \oc(\lambda^{(i+1)}, g)|=g-\lambda_i+i-1 $ is greater than $g-\lambda_i-1$, we conclude that  $\tau / \oc(\lambda^{(i)}, g) $ and $ \tau / \oc(\lambda^{(i+1)}, g) $ share the same maximal column indices. Now Corollary \ref{cor: ohs jdt1} guarantees the existence of an SSOT $ T'' \in \SSOT(\oc(\lambda^{(i+1)}, g)$, $ \beta^{(i+1)}) $ such that $ T = \Aug(T'', g - \lambda_{i+1} + i) $. 
Since $ c(T) \leq g $, we clearly have $ c(T'') \leq g $, and therefore $ T \in G_{i+1}^{(1)} $. 
Thus we have shown that $ G_i^{(2)} \subseteq G_{i+1}^{(1)} $, and a similar argument shows that $ G_{i+1}^{(1)} \subseteq G_i^{(2)} $.

Using \eqref{eq: augg} and \eqref{a}, the left-hand side of \eqref{b} becomes
\begin{equation*}
    \sum_{i \geq 1} (-1)^{i-1} \sum_{T \in G_i} q^{\overline{D}(T)} = \sum_{T \in G_1^{(1)}} q^{\overline{D}(T)}.
\end{equation*}
Any $ T \in G_1^{(1)} $ has the shape $ \oc(\lambda, g) $, so $ G_1^{(1)} \subseteq \SSOT_{g}(\oc(\lambda, g), \overline{\oc}(\mu, g)) $.
By Corollary \ref{cor: ohs jdt last row}, the reverse inclusion holds, i.e., $ \SSOT_{g}(\oc(\lambda, g), \overline{\oc}(\mu, g)) \subseteq G_1^{(1)} $.
\end{proof}

\begin{example}
Keeping the notations from the proof of \eqref{eq: goal}, let $\lambda = (4,2)$, $\mu = (0,0)$, and $g = 7$. On the left, we list the elements $T^{(1)}, \dots, T^{(7)}$ of $\SSOT_7((5), (4,7))=\SSOT_g(\oc(\lambda^{(1)},g),\beta^{(1)})$ and their corresponding $\Aug(T, 3)$. On the right, we present the elements $T^{(8)}$ and $T^{(9)}$ of $\SSOT_7((2), (1,7))=\SSOT_g(\oc(\lambda^{(2)},g),\beta^{(2)})$ along with their $\Aug(T, 6)$. Each $\SSOT$ is represented by its image under the map $\phi_c$.
\[
\begin{array}{rl|rl}
    T^{(1)}: & 12345 \otimes \emptyset \hspace{4.5mm}\rightarrow \Aug(T^{(1)},3): 12345 \otimes 123 
    & T^{(8)}: & 23\bar{3} \otimes 1 \rightarrow \Aug(T^{(8)},6): 12\bar{7} \otimes 1234567 \\[5pt]
    T^{(2)}: & 34567\bar{7}\bar{6} \otimes 12 \rightarrow \Aug(T^{(2)},3): 12367\bar{7}\bar{6} \otimes 12345 
    & T^{(9)}: & 2 \otimes 1 \hspace{3.4mm}\rightarrow \Aug(T^{(9)},6): 1 \otimes 1234567 \\[5pt]
    T^{(3)}: & 3456\bar{6} \otimes 12 \hspace{3.5mm}\rightarrow \Aug(T^{(3)},3): 1236\bar{6} \otimes 12345 \\[5pt]
    T^{(4)}: & 345 \otimes 12 \hspace{6.9mm}\rightarrow \Aug(T^{(4)},3): 123 \otimes 12345 \\[5pt]
    T^{(5)}: & 567\bar{7}\bar{6} \otimes 1234 \rightarrow \Aug(T^{(5)},3): 123\bar{7}\bar{6} \otimes 1234567 \\[5pt]
    T^{(6)}: & 56\bar{6} \otimes 1234 \hspace{3.5mm}\rightarrow \Aug(T^{(6)},3): 12\bar{7} \otimes 1234567 \\[5pt]
    T^{(7)}: & 5 \otimes 1234 \hspace{6.9mm}\rightarrow \Aug(T^{(7)},3): 1 \otimes 1234567 \\[5pt]
\end{array}
\]
As predicted in the proof of \eqref{eq: goal}, we have $G_1^{(2)}=G_2^{(1)}=G_2$ and $G_1^{(1)}=\{\Aug(T^{(i)},3): 1\leq i\leq 5\}$ is precisely the set $\SSOT_g((\oc(\lambda,g)),\oc(\mu,g))=\SSOT_7((5,3),(7,7))$.
\end{example}

\subsection{Proof of Theorem \ref{thm: B lusztig}}\label{sub: B skecth}
In \cite[Theorem 3.2.1]{Lecouvey2006}, Morris type recurrence formula for type $B$ is derived. 
A straightforward generalization of the argument there leads to the following $q,t$ version: \footnote{For weights that are not spin weights, we do not have a recurrence for the $q,t$ version as elegant as \eqref{eq: B morris recurrence}}

\begin{equation}\label{eq: B morris recurrence}
    \KL^{B_n}_{\lambda^{\sharp},\mu^{\sharp}}(q,t) = \sum_{i \geq 1} (-1)^{i-1} \sum_{\substack{r+m = \lambda_i - \mu_1 + 1 - i \\ r, m \geq 0}} q^r t^m \sum_{(\lambda^{(i)}, \tau, \nu) \in \operatorname{ROHS}_{\leq n-1}(\lambda^{(i)}, r)} \KL^{B_{n-1}}_{\nu^{\sharp}, (\mu')^{\sharp}}(q,t).
\end{equation}
Here, $\lambda^{(i)}$ and $\mu'$ are given by \eqref{eq: lambdais}. We proceed with induction on $n$ to prove \ref{thm: B lusztig}. The base case $n=1$ is covered easily as in Example \ref{ex : C lusztig q=1}. Assuming Theorem \ref{thm: B lusztig} holds for $\KL^{B_{n-1}}_{\lambda^{\sharp}, \mu^{\sharp}}$, the right-hand side of \eqref{eq: B morris recurrence} becomes
\begin{equation}\label{noname1}
    \sum_{i \geq 1} (-1)^{i-1} \sum_{\substack{r+m = \lambda_i - \mu_1 + 1 - i \\ r, m \geq 0}} q^r t^m \sum_{(\lambda^{(i)}, \tau, \nu) \in \operatorname{ROHS}_{\leq n-1}(\lambda^{(i)}, r)} \left( \sum_{T \in \GSSOT_{g + \frac{1}{2}}(\oc(\nu, g), \overline{\oc}(\mu', g))} \energy_{q,t}(\phi_c(T)) \right)
\end{equation}
where $g$ is chosen to be sufficiently large, as discussed earlier. 
Lemma \ref{lem: add rohs to gssot} serves as the counterpart to Lemma \ref{lem: add rohs}, and Lemma \ref{lem: energy qr t m} corresponds to Lemma \ref{lem: energy r+m}.

\begin{lem}\label{lem: add rohs to gssot}
    Let $A^{(i)}$ be a set of pairs $(S,T)$ satisfying:
    \begin{itemize}
        \item $S\in \ROHS_{\leq n-1}(\lambda^{(i)},r)$ for some $r$ such that there exists $m\geq 0$ with $r+m=\lambda_i-\mu_1+1-i$,
        \item $T\in\GSSOT_{g+\frac{1}{2}}(\oc(\nu,g),\overline{\oc}(\mu',g))$ where $\nu=F(S)$.
    \end{itemize}
    Then there exists a bijection 
    \begin{equation*}
        \Phi^{(i)}: A^{(i)}\rightarrow \GSSOT_{g+\frac{1}{2}}(\oc(\lambda^{(i)},g),\gamma^{(i)})
    \end{equation*}
    where $\gamma^{(i)}=(g-\mu_2,\dots,g-\mu_n,\lambda_i-\mu_1+1-i)$.
\end{lem}
\begin{proof}
    Similar to the proof of Lemma \ref{lem: add rohs}.
\end{proof}

\begin{lem}\label{lem: energy qr t m}
Let $A^{(i)}$ be the set in Lemma \ref{lem: add rohs to gssot}. For $(S,T)\in A^{(i)}$ such that $S\in \ROHS_{\leq n-1}(\lambda^{(i)},r)$, we have
    \begin{equation*}
        \energy_{q,t}(\phi_c(\Phi^{(i)}(S,T)))=q^{r}t^m\energy_{q,t}(\phi_c(T))
    \end{equation*}
    where $m$ is given by $r+m=\lambda_i-\mu_1+1-i$.
\end{lem}
\begin{proof}
    Follows from Proposition \ref{prop:split D}
\end{proof}

By Lemma \ref{lem: add rohs to gssot} and \ref{lem: energy qr t m}, and applying the same reasoning as before, we find that \eqref{noname1} simplifies to
\begin{equation}\label{noname2}
    \sum_{i \geq 1} (-1)^{i-1} \sum_{T \in \GSSOT_{g + \frac{1}{2}}(\oc(\lambda^{(i)}, g), \beta^{(i)})} \energy_{q,t}(\phi_c(T)).
\end{equation}
Here, $\beta^{(i)} = (\lambda_i - \mu_1 + 1 - i, g - \mu_2, \dots, g - \mu_n)$. 
By Remark \ref{rmk: aug D_n}, we can rewrite \eqref{noname2} as
\begin{equation}\label{noname3}
    \sum_{i \geq 1} (-1)^{i-1} \sum_{T \in \GSSOT_{g + \frac{1}{2}}(\oc(\lambda^{(i)}, g), \beta^{(i)})} \energy_{q,t}(\phi_c(\Aug(T, g - \lambda_i + i - 1))).
\end{equation}
Note $\Aug(T,g-\lambda_i+i-1)$ is given where we regard a GSSOT $T$ as an SSOT (possibly with the different weight). At this point, the proof follows exactly as in the derivation of \eqref{eq: goal}, utilizing the decomposition given in \eqref{eq: gssot decomposition}.

\section{level-restricted $q$-weight multiplicities}\label{Sec: Ltilde}
In this section, we show Theorem \ref{thm: level formula}, which is a refinement of $X=K$ theorem for tensor products of row KR crystals \cite{S05, LS2007}.

For a dominant weight $\lambda\in P^{+}_n$, we denote the character corresponding to $\lambda$ by $s_{\lambda}^{\mathfrak{g}_n}(x)$. 
When $\mathfrak{g}_n=A_{n-1}$, $s_{\lambda}^{\mathfrak{g}_n}(x)$ is the Schur polynomial, which we simply denote by $s_{\lambda}(x)$ for brevity.
In general, according to the Weyl character formula,
\begin{align*}
    s^{\mathfrak{g}_n}_{\lambda}(x_1,x_2,\dots,x_n)=\frac{\sum_{w\in W}(-1)^w x^{w(\lambda+\rho)}}{\sum_{w\in W}(-1)^w x^{w(\rho)}}
\end{align*}
where the notation $x^{\beta}$ for a vector $\beta$ denotes $\prod_{i\geq 1}x_i^{\beta_i}$, $W$ is the Weyl group and $\rho=\sum_{\alpha\in R^{+}}\alpha/2$.
The coefficient $[x^{\mu}]s_{\lambda}^{\mathfrak{g}_n}$ is simply the weight multiplicity, therefore equals $\KL^{\mathfrak{g_n}}_{\lambda,\mu}(1)$.

We first establish the positivity of the level-restricted $q$-weight multiplicities (Proposition \ref{prop: tilted positivity}).
We begin by defining some terminologies and giving relevant background. 
For the remainder of this section, we denote $\hat{\lambda}$ for $\oc(\lambda,g)$ and $\bar{\lambda}$ for $\overline{\oc}(\lambda,g)$, where $ g $ will be clear from the context.

Given any map $ L: R^{+} \rightarrow \mathbb{Z} $, for sufficiently large $ k $ (with $ k \geq \frac{|\lambda| - |\mu|}{2} $ being sufficient), the value of $ \KL^{\mathfrak{g}_n, L}_{\lambda + (k^n), \mu + (k^n)}(q) $ stabilizes to the following expression \cite[Proposition 5]{LS2007}:
\[
    ^{\infty}\KL^{\mathfrak{g}_n, L}_{\lambda, \mu}(q) := \sum_{w \in \mathfrak{S}_n} (-1)^w[e^{w(\lambda + \rho) - (\mu + \rho)}] \prod_{\alpha \in R^+} \frac{1}{1 - q^{L(\alpha)} e^\alpha}.
\]

A straightforward generalization of the argument in \cite[Section 4]{LS2007} yields the following equation:
\begin{equation}\label{eq: level restricted stable}
    ^{\infty} \KL^{\mathfrak{g}_n, L_A}_{\lambda, \mu}(q) = \sum_{\nu} \KL^{A_{n-1}}_{\hat{\nu}, \hat{\mu}}(q) \sum_{\gamma \in P^{\mathfrak{g}_n}} c^{\hat{\nu}}_{\gamma, \hat{\lambda}},
\end{equation}
where 
\[
    P^{B_n} = P_n^{\sboxone} := \Par_n, \quad P^{C_n} = P_n^{\sboxtwo} := \{ \lambda \in \Par_n : \text{parts of } \lambda \text{ are even} \}, \quad P^{D_n} = P_n^{\sboxeleven} := \{ \lambda \in \Par_n : \text{parts of } \lambda^t \text{ are even} \},
\]
and $ c^{\lambda}_{\nu, \mu} := [s_{\lambda}](s_{\nu} s_{\mu}) $ denotes a Littlewood-Richardson coefficient.

We now present a refinement of \eqref{eq: level restricted stable} (Proposition \ref{prop: tilted positivity}).
\begin{definition}
    For $ \lambda \in P^{+}_n $, $s^{\mathfrak{g}_n}_{\lambda}$ is a Laurent polynomial. Given $g\geq \lambda_1$, $(x_1 \cdots x_n)^g s_{\lambda}^{\mathfrak{g}_n}(x_1^{-1}, x_2^{-1}, \dots, x_n^{-1})$ becomes symmetric polynomial, therefore we can expand in terms of Schur polynomials. 
    We denote the twisted branching coefficient $ d^{\mathfrak{g}_n}_{\lambda,\mu} \in \mathbb{Z}_{\geq 0} $ as the coefficient such that 
    \[
        (x_1 \cdots x_n)^g s_{\lambda}^{\mathfrak{g}_n}(x_1^{-1}, x_2^{-1}, \dots, x_n^{-1}) = \sum_{\mu} d^{\mathfrak{g}_n}_{\lambda, \mu} s_{\widehat{\mu}}(x).
    \]
    Since $ s_{\mu + (1^n)} = (x_1 \cdots x_n) s_{\mu} $, the value of $ d^{\mathfrak{g}_n}_{ \lambda, \mu} $ is independent of the choice of $ g \geq \lambda_1$.
\end{definition}

\begin{rmk}
    The character $ s_{\lambda}^{\mathfrak{g}_n} $ is invariant under the action of the Weyl group $ W $. 
    When $ \mathfrak{g}_n = B_n $ or $ \mathfrak{g}_n = C_n $, the Weyl group is $ \mathfrak{S}_n \ltimes (\mathbb{Z}/2\mathbb{Z})^n $, so we have
    $
        s_{\lambda}^{\mathfrak{g}_n}(x_1^{-1}, \dots, x_n^{-1}) = s_{\lambda}^{\mathfrak{g}_n}(x_1, \dots, x_n).
    $
    However, when $ \mathfrak{g}_n = D_n $, the Weyl group is $ \mathfrak{S}_n \ltimes (\mathbb{Z}/2\mathbb{Z})^{n-1} $, $ s_{\lambda}^{\mathfrak{g}_n}(x_1^{-1}, \dots, x_n^{-1}) $ generally differs from $ s_{\lambda}^{\mathfrak{g}_n}(x_1, \dots, x_n) $.
\end{rmk}

\begin{proposition}\label{prop: tilted positivity}
    We have the following identity
    \begin{equation*}
        \KL^{\mathfrak{g}_n,L_A}_{\lambda,\mu}(q) = \sum_{\nu} \KL^{A_{n-1}}_{\hat{\nu}, \hat{\mu}}(q) \cdot d^{\mathfrak{g}_n}_{\lambda, \nu}.
    \end{equation*}
\end{proposition}

\begin{proof}
    Let $ W $ denote the Weyl group. 
    Consider the coset representatives $[W: \mathfrak{S}_n]$, which means that any element $ w \in W $ can be uniquely written as $ w = w_1 w_2 $, where $ w_1 \in \mathfrak{S}_n $ and $ w_2 \in [W: \mathfrak{S}_n]$.
    Then we have 
    \begin{align*}
        \KL^{\mathfrak{g}_n,L_A}_{\lambda, \mu}(q) &=  \sum_{w_1 \in \mathfrak{S}_n}\sum_{w_2 \in[W: \mathfrak{S}_n]} (-1)^{w_1 + w_2} [e^{w_1 w_2 (\lambda + \rho) - (\mu + \rho)}] \prod_{\alpha \in R^{+}} \frac{1}{1 - q^{L_A(\alpha)} e^\alpha} \\
        &= \sum_{w_2 \in[W: \mathfrak{S}_n]} (-1)^{w_2} \ ^{\infty} \KL^{\mathfrak{g}_n,L_A}_{w_2(\lambda + \rho) - \rho, \mu}(q) \\
        &= \sum_{\nu} \KL_{\hat{\nu}, \hat{\mu}}^{A_{n-1}}(q) \left( \sum_{\substack{w_2 \in [W: \mathfrak{S}_n]\\ \tau = w_2 (\lambda + \rho) - \rho}} (-1)^{w_2} \sum_{\gamma \in P^{\mathfrak{g}_n}} c^{\hat{\nu}}_{\gamma, \widehat{\tau}} \right),
    \end{align*}
    where the last equality follows from \eqref{eq: level restricted stable}. 
    Moreover, we have the identity
    \begin{align}\label{eq: nnn}
        \sum_{\substack{w_2 \in [W: \mathfrak{S}_n] \\
        \tau = w_2 (\lambda + \rho) - \rho}} (-1)^{w_2} \sum_{\gamma \in P_n^{\mathfrak{g}_n}} c^{\hat{\nu}}_{\gamma, \widehat{\tau}} = [s_{\hat{\nu}}] \left( \sum_{\substack{w_2 \in [W: \mathfrak{S}_n] \\ 
        \tau = w_2 (\lambda + \rho) - \rho}} (-1)^{w_2} s_{\hat{\tau}} \right) \left( \sum_{\gamma \in P^{\mathfrak{g}_n}} s_{\gamma} \right).
    \end{align}
    Note that    
    \begin{align*}
        \sum_{\substack{w_2 \in [W: \mathfrak{S}_n] \\ 
        \tau = w_2 (\lambda + \rho) - \rho}} (-1)^{w_2} s_{\hat{\tau}} &= \sum_{\substack{w_2 \in [W: \mathfrak{S}_n] \\ 
        \tau = w_2 (\lambda + \rho) - \rho}} (-1)^{w_2} s_{\hat{\tau}}(x_n, \dots, x_1) = \frac{\sum_{\substack{w_2 \in [W: \mathfrak{S}_n] \\ 
        \tau = w_2 (\lambda + \rho) - \rho}} \sum_{w_1 \in \mathfrak{S}_n} x^{w_1 (\bar{\tau} + \rho_A)}}{\prod_{i < j} (x_i - x_j)} \\
        &= (x_1 \dots x_n)^{g + M} \left( \frac{\sum_{\substack{w_2 \in [W: \mathfrak{S}_n] \\ 
        \tau = w_2 (\lambda + \rho) - \rho}} \sum_{w_1 \in \mathfrak{S}_n} x^{-w_1 (\tau + \rho)}}{\prod_{i < j} (x_i - x_j)} \right) = (x_1 \dots x_n)^{g + M} \left( \frac{\sum_{w \in W} x^{-w (\lambda + \rho)}}{\prod_{i < j} (x_i - x_j)} \right),
    \end{align*}
    where $ \rho_A = (0, 1, \dots, n - 1) $ and $ M $ is a positive number such that $ \rho + \rho_A = (M, \dots, M) $. 
    By the Littlewood identities \cite{LW50}, we have
    \begin{equation*}
        \sum_{w \in W} x^{-w (\rho)} = \frac{\prod_{i < j} (x_i - x_j)}{(x_1 \dots x_n)^M \sum_{\gamma \in P^{\mathfrak{g}_n}} s_{\gamma}}.
    \end{equation*}
    Therefore, we conclude that
    \begin{equation*}
        \left( \sum_{\substack{w_2 \in [W:\mathfrak{S}_n] \\ \tau = w_2 (\lambda + \rho) - \rho}} (-1)^{w_2} s_{\hat{\tau}} \right) \left( \sum_{\gamma \in P^{\mathfrak{g}_n}} s_{\gamma} \right) = (x_1 \dots x_n)^{g} s_{\lambda}^{\mathfrak{g}_n}(x_1^{-1}, \dots, x_n^{-1}).
    \end{equation*}
    Combining this with \eqref{eq: nnn}, the proof is complete.
\end{proof}

In \cite{S05, LS2007}, it is shown that
\begin{equation}\label{eq: x=k}
    q^{||\mu|||+\frac{|\mu|-|\lambda|}{2}}\left(\sum_{b\in \HW(B_{\mu}(\diamond),\lambda)}q^{-\frac{|\diamond|\overline{D}(b)}{2}}\right)=\sum_{\nu}\KL^{A_{n-1}}_{\nu,\mu}(q) \sum_{\gamma\in P^{\diamond}}c^{\nu}_{\gamma,\lambda}.
\end{equation}
This result, now commonly referred to as the $X = K$ theorem in the literature, was later generalized to any tensor products of KR crystals (see \cite{LOS2012}). We now present a refinement of this identity.

\begin{thm}\label{thm: level formula}
    We have
    \begin{align}\label{eq: tilted B}
        \KL^{\mathfrak{g}_n,L_A}_{\lambda,\mu}(q)&=\sum_{\nu}\KL^{A_{n-1}}_{\hat{\nu},\hat{\mu}}(q) d^{\mathfrak{g}_n}_{\lambda,\nu}= q^{||\hat{\mu}|||+\frac{|\hat{\mu}|-|\hat{\lambda}|}{2}}\left(\sum_{\substack{b\in \HW(B_{\hat{\mu}}(\diamond),\hat{\lambda})\\\epsilon_0(b)\leq \zeta g}}q^{-\frac{|\diamond|\overline{D}(b)}{2}}\right)
    \end{align}
\end{thm}
where \begin{equation}\label{eq: diamond g_n}
\diamond=\begin{cases*}
    \sboxone \qquad&\text{if $\mathfrak{g}_n=B_n$}\\
    \sboxtwo \qquad&\text{if $\mathfrak{g}_n=C_n$}\\
    \sboxeleven \qquad&\text{if $\mathfrak{g}_n=D_n$,}
\end{cases*}\qquad \qquad \zeta=\begin{cases*}
    2 \qquad&\text{if $\mathfrak{g}_n=B_n$}\\
    1 \qquad&\text{if $\mathfrak{g}_n=C_n$}\\
    2 \qquad&\text{if $\mathfrak{g}_n=D_n$.}
\end{cases*}
\end{equation}
Note that the first identity in Theorem \ref{thm: level formula} is simply Proposition \ref{prop: tilted positivity}. Theorem holds as long as $\hat{\lambda} = \oc(\lambda, g)$ and $\hat{\mu} = \oc(\mu, g)$ are nonnegative integer vectors. Therefore, when $\lambda$ and $\mu$ are spin weights, we must take $g \in \mathbb{Z} + \frac{1}{2}$. We outline the proof strategy for Theorem \ref{thm: level formula}. In Section \ref{sub: filtering x=k}, we establish the following lemma. 

\begin{lem}\label{lem: x=k filter} For $\lambda,\mu\in \Par_n$ and a positive integer $M$, there exists a constant $\bar{d}^{\diamond}_{\lambda, \nu}(M) \in \mathbb{Z}_{\geq 0}$ such that
    \begin{equation*}
        q^{||\mu|| + \frac{|\mu| - |\lambda|}{2}} \left( \sum_{\substack{x \in \HW(B_{\mu}(\diamond), \lambda) \\ \epsilon_0(x) \leq M}} q^{-\frac{|\diamond| \overline{D}(x)}{2}} \right)
        = \sum_{\nu} \KL^{A_{n-1}}_{\nu, \mu}(q) \bar{d}^{\diamond}_{\lambda, \nu}(M).
    \end{equation*}
\end{lem}

In Section \ref{sub: level q=1}, we prove Theorem \ref{thm: level formula} at $q=1$. Then the proof for general $q$ follows as outlined below.

\begin{proof}[Proof of Theorem \ref{thm: level formula}]
We substitute $\hat{\mu}$ and $\hat{\lambda}$ for $\mu$ and $\lambda$, respectively, in Lemma \ref{lem: x=k filter}. To complete the proof, it suffices to show that $\bar{d}^{\diamond}_{\hat{\lambda}, \hat{\nu}}(\zeta g) = d^{\mathfrak{g}_n}_{\lambda, \nu}$. Since Theorem \ref{thm: level formula} holds for $q=1$ (as proved in Section \ref{sub: level q=1}), combined with Lemma \ref{lem: x=k filter}, we have the identity
    \begin{equation*}
        [x^{\mu}] s_{\lambda}^{\mathfrak{g}_n} = \sum_{\nu} \left( [x^{\hat{\mu}}] s_{\hat{\nu}} \right) \bar{d}^{\diamond}_{\hat{\lambda}, \hat{\nu}}(\zeta g).
    \end{equation*}
    This is equivalent to
    \begin{equation*}
        [x^{-\mu}] s_{\lambda}^{\mathfrak{g}_n}(x_1^{-1}, \dots, x_n^{-1}) = \sum_{\nu} \left( [x^{\hat{\mu}}] s_{\hat{\nu}} \right) \bar{d}^{\diamond}_{\hat{\lambda}, \hat{\nu}}(\zeta g),
    \end{equation*}
    and summing over all possible $\mu$ after multiplying both sides by $x^{-\mu}$ gives
    \begin{equation*}
        \sum_{\mu} x^{-\mu} \left( [x^{-\mu}] s_{\lambda}^{\mathfrak{g}_n}(x_1^{-1}, \dots, x_n^{-1}) \right) = \sum_{\mu} x^{-\mu} \sum_{\nu} \left( [x^{\hat{\mu}}] s_{\hat{\nu}} \right) \bar{d}^{\diamond}_{\hat{\lambda}, \hat{\nu}}(\zeta g).
    \end{equation*}
    The left-hand side simplifies to $s_{\lambda}^{\mathfrak{g}_n}(x_1^{-1}, \dots, x_n^{-1})$, while the right-hand side becomes
    \begin{equation*}
        \sum_{\nu} \sum_{\mu} x^{-\mu} \left( [x^{\hat{\mu}}] s_{\hat{\nu}} \right) \bar{d}^{\diamond}_{\hat{\lambda}, \hat{\nu}}(g) = \sum_{\nu} (x_1 \dots x_n)^{-g} \bar{d}^{\diamond}_{\hat{\lambda}, \hat{\nu}}(\zeta g) s_{\hat{\nu}}.
    \end{equation*}
    By the definition of $d^{\mathfrak{g}_n}_{\lambda, \nu}$, we conclude that 
        $\bar{d}^{\diamond}_{\hat{\lambda}, \hat{\nu}}(\zeta g) = d^{\mathfrak{g}_n}_{\lambda, \nu}.$
    \end{proof}

\subsection{Refining $X=K$ theorem for $B_{\mu}(\diamond)$}\label{sub: filtering x=k}
\subsubsection{$\diamond=\sboxone$ or $\sboxtwo$} \label{subsub: filtering x=k BC}
Our goal is to prove Lemma \ref{lem: x=k filter} when $\diamond = \sboxone$ or $\diamond = \sboxtwo$. The outline of the proof is as follows:

\begin{itemize}
    \item We first derive Proposition \ref{prop: GSOT LR}, which generalizes \cite[Theorem 9.4]{Sundaram1986}.
    \item We then complete the proof by utilizing the map introduced in \cite{S05}.
\end{itemize}

In particular, the derivation of Proposition \ref{prop: GSOT LR} relies only on combinatorics related to the RSK correspondence (see \cite[Chapter 7.11]{EC2} for the standard reference).
Thus we temporarily move beyond the realm of crystal theory. 
We begin by defining the necessary terminologies.  

Let $T$ be a semistandard Young tableau with distinct entries. The \emph{row-insertion} (or \emph{column-insertion}) of a number $x$ (distinct from entries in $T$) into $T$ is defined as follows:
\begin{enumerate}
    \item Insert $x$ into the first row (or column) of $T$. If all elements in the row (or column) are smaller than $x$, add $x$ to the end. Otherwise, bump the leftmost (or topmost) element that is greater than $x$.
    \item If $x$ bumps an element $x'$ from the first row (or column), insert $x'$ into the second row (or column) by repeating the same procedure.
    \item Continue this process until an element is inserted at the end of a row (or column).
\end{enumerate}
We denote the row-insertion of a number $x$ into $T$ by $(T \leftarrow x)$, and the column-insertion by $(x \rightarrow T)$. For the remainder of this section, we only consider a semistandard Young tableaux or a word with distinct entries. 

For a word of positive integers $w = w_1 \cdots w_n$, the row-insertion of $w$ into $T$ is defined as
\[
T \leftarrow w = (((T \leftarrow w_1) \leftarrow w_2) \leftarrow \cdots \leftarrow w_n),
\]
and the column-insertion of $w$ into $T$ is defined as
\[
w \rightarrow T = (w_1 \rightarrow \cdots \rightarrow (w_{n-1} \rightarrow (w_n \rightarrow T))).
\]
A classical result states that
$
\emptyset \leftarrow w = w \rightarrow \emptyset.
$
For a semistandard Young tableau $T$, row reading word of $T$ is a word given by: rows are read from bottom to top, and within each row the entries are read from right to left. Similarly column reading word of $T$ is given by: column are read from left to right, and within each column the entries are read from bottom to top. Note that we have $T=\emptyset \leftarrow w^{(1)}$=$ w^{(2)}\rightarrow\emptyset$ where $w^{(1)}$ is a column reading word and $w^{(2)}$ is a row reading word of $T$.

Now, let $y$ be a corner cell of $T$. The \emph{row-deletion} (or \emph{column-deletion}) of a cell $y$ in $T$ is defined as follows:
\begin{enumerate}
    \item Find and bump the rightmost (bottommost) entry $y'$ in the previous row (or column) that is smaller than $y$.
    \item Repeat this process until reaching the first row (or column).
\end{enumerate}
In other words, this process identifies $y'$ and $T'$ such that $T' \leftarrow y'$ (or $y' \rightarrow T'$) equals $T$, and the shape of $T'$ is obtained by removing the cell $y$ from $T$.

\begin{definition}\label{def: two line array}
    Let $\TL(r)$ denote the set of two-line arrays of positive integers of the form
    \[
    I = \begin{pmatrix}
    j_1 & \dots & j_r \\
    i_1 & \dots & i_r 
    \end{pmatrix},
    \]
    where the following conditions hold:
 \begin{itemize}
    \item $j_s \geq i_s$ for all $s$, \qquad \qquad \qquad
    $\bullet$ $j_1 < j_2 < \dots < j_r$,
    \item any two entries in $I$ are distinct except for the case $j_s = i_s$.
\end{itemize}
    For $I\in \TL(r)$, we define $\bar{I}$ as the two-line array obtained by adding columns $\begin{pmatrix}
        i_s \\
        j_s
    \end{pmatrix}$ when $j_s > i_s$, and then rearranging the columns according to the order of the first row. 
    We define $\hat{I}$ as the two-line array obtained from $I$ by removing columns $\begin{pmatrix}
        j_s \\
        i_s
    \end{pmatrix}$ when $j_s = i_s$.
    Lastly, for any two-line array $I = \begin{pmatrix}
        j_1 & \dots & j_r \\
        i_1 & \dots & i_r 
    \end{pmatrix}$, we associate the tableau
    \[
    \tab(I) := (i_r \dots i_1 \rightarrow \emptyset).
    \]
\end{definition}

Given a semistandard Young tableau $T$ with distinct entries, there exists a unique $I \in \TL(r)$ such that $\tab(\bar{I}) = T$ (see \cite[Corollary 7.13.7]{EC2}). 
For such $I$, we define $\alpha(T) = \tab(I)$ and $\beta(T) = \tab(\hat{I})$. 
The following lemma plays a key role in the proof of Proposition \ref{prop: GSOT LR}.
Its proof is provided in Appendix \ref{Sec: append B}.
The notation $T^{\leq m}$ refers to the tableau obtained by restricting $T$ to its first $m$ columns.

\begin{lem}\label{lem: useful lemma}
Let $T$ be a semistandard Young tableau with distinct entries and $k$ be a positive integer. Then we have
\[
\alpha(T^{\leq 2k+1}) = \alpha(T)^{\leq k+1}, \quad \beta(T^{\leq 2k}) = \beta(T)^{\leq k}.
\]
\end{lem}

\begin{example}
For the tableau \(T\) below, we list \(T^{\leq k}\), \(\alpha(T^{\leq k})\), and \(\beta(T^{\leq k})\) for some \(k\)'s:
\begin{align*}
T &\quad = \begin{ytableau}
1 & 2 & 7 & 9 & 10 & 13 \\
3 & 4 & 11 & 14 & 15 \\
5 & 8 & 12 \\
6
\end{ytableau}, & 
\alpha(T) &\quad = \begin{ytableau}
1 & 2 & 7 \\
3 & 4 & 12 \\
5 & 8 \\
6
\end{ytableau}, & 
\beta(T) &\quad = \begin{ytableau}
1 & 5 & 7 \\
3 & 8 \\
6  
\end{ytableau}, \\[1.5ex]
T^{\leq 5} &\quad = \begin{ytableau}
1 & 2 & 7 & 9 & 10 \\
3 & 4 & 11 & 14 & 15 \\
5 & 8 & 12 \\
6
\end{ytableau}, & 
\alpha(T^{\leq 5}) &\quad = \begin{ytableau}
1 & 2 & 7 \\
3 & 4 & 12 \\
5 & 8 \\
6
\end{ytableau}, & 
\beta(T^{\leq 5}) &\quad = \begin{ytableau}
1 & 5 \\
3 & 8 \\
6  
\end{ytableau}, \\[1.5ex]
T^{\leq 4} &\quad = \begin{ytableau}
1 & 2 & 7 & 9  \\
3 & 4 & 11 & 14  \\
5 & 8 & 12 \\
6
\end{ytableau}, & 
\alpha(T^{\leq 4}) &\quad = \begin{ytableau}
1 & 2  \\
3 & 4 \\
5 & 8 \\
6
\end{ytableau}, & 
\beta(T^{\leq 4}) &\quad = \begin{ytableau}
1 & 5 \\
3 & 8 \\
6  
\end{ytableau}.
\end{align*}
\end{example}

For a skew shape $\lambda / \mu$, a \emph{Littlewood-Richardson tableau} (LR tableau for short) is a semistandard filling of the shape $\lambda / \mu$ such that the row reading word is a lattice permutation. For the technical reason, we consider \emph{transposed-LR tableau} which is simply a transpose of an LR tableau. The set of transposed-LR tableaux on $\lambda / \mu$ with weight $\gamma$ is denoted by $\LR^{t}(\lambda /\mu;\gamma)$.

Let $\GSSOT(\lambda, (1^n))$ be denoted by $\GSOT(\lambda,n)$, where each $G \in \GSOT(\lambda, n)$ is identified with a sequence of partitions $(G_0, G_1, \dots, G_n)$ such that $G_0 = \emptyset$, $G_n = \lambda$, and two consecutive partitions differ by at most one cell. 
Specifically, $|G_i / G_{i-1}| = 1$, $|G_{i-1} / G_i| = 1$, or $G_i = G_{i-1}$. Similarly, we define the set $\SOT(\lambda, n):=\SSOT(\lambda,(1^n))$, which is naturally a subset of $\GSOT(\lambda, n)$.

We now describe the bijection
\begin{equation*}
    \Phi^{BC}: \GSOT(\lambda,n) \to \bigcup_{\mu \vdash n} \bigcup_{\gamma \in P^{\sboxone}_n} \SYT(\mu) \times \LR^{t}(\mu / \lambda; \gamma^{t}),
\end{equation*}
where $\SYT(\mu)$ is a set of standard Young tableaux on a shape $\mu$. 
This map was introduced in \cite{Sundaram1986, delest1988analogue} for the set $\SOT(\lambda, n)$\footnote{The maps in \cite{Sundaram1986, delest1988analogue} are essentially the same, except for a transpose of a tableau.}.
Specifically, restricting $\Phi^{BC}$ to $\SOT(\lambda, n)$ within $\GSOT(\lambda, n)$ gives the bijection:
\begin{equation}\label{eq: gsot to sot restriction}
    \Phi^{BC}: \SOT(\lambda,n) \to \bigcup_{\mu \vdash n} \bigcup_{\gamma \in P^{\sboxtwo}_n} \SYT(\mu) \times \LR^{t}(\mu / \lambda; \gamma^{t}).
\end{equation}
Shimozono later extended this map to the entire set $\GSOT(\lambda,n)$ in \cite{S05}.

The description of $\Phi^{BC}$ is two-fold: 
first, it involves applying the map $\Phi_1^{BC}$, followed by applying the map $\Phi_2^{BC}$. 
Given a sequence $G = (G_0, G_1, \dots, G_n) \in \GSOT(\lambda, n)$, at each step, we associate a pair $(T^{(j)}, I^{(j)})$, where $T^{(j)}$ is a semistandard Young tableau (with distinct entries) of the same shape with $G_j$, and $I^{(j)}\in \TL(r)$ for some $r$.
We initialize by setting $T^{(0)} = \emptyset$ and $I^{(0)} = \emptyset\in \TL(0)$.
At each $j$-th step, we perform the following:
\begin{enumerate}
    \item If $|G_j / G_{j-1}| = 1$, we insert a new cell on $G_j / G_{j-1}$ with the entry $j$ into $T^{(j-1)}$ to obtain $T^{(j)}$ and $I^{(j)}=I^{(j-1)}$.
    \item If $G_j = G_{j-1}$, $I^{(j)}$ is obtained by appending the column
        $\begin{pmatrix}
            j \\
            j
        \end{pmatrix}$
    to $I^{(j-1)}$, and we set $T^{(j)} = T^{(j-1)}$.
    \item If $|G_{j-1} / G_j| = 1$, we perform a row-deletion on the cell $G_{j-1} / G_j$. Let the resulting tableau be $T^{(j)}$. Denote the bumped out entry as $x$, and append the column
        $\begin{pmatrix}
            j \\
            x
        \end{pmatrix}$
    to $I^{(j-1)}$ to obtain $I^{(j)}$.
\end{enumerate}

After completing all $n$ steps, we obtain a pair $(T, I) = (T^{(n)}, I^{(n)})$, which we set as $\Phi_1^{BC}(G)$.
Next, denote the column reading word of $\tab(\bar{I})$ by $w$ and let $P = T \leftarrow w$. If $P$ has shape $\mu$, we define $Q \in \LR^{t}(\mu / \lambda)$ by assigning the integer $i$ to the newly created cells when inserting letters in the $i$-th column of $\tab(\bar{I})$ during the process $T \leftarrow w$.
Finally, we set $(P, Q) = \Phi_2^{BC}(T, I) = \Phi^{BC}(G)$.

\begin{example}\label{ex:r of rcr}
For $G\in \GSOT((3,1),9)$ given by
  \begin{align*}
        G = (\emptyset, \ydiagram{1}, \ydiagram{2}, \ydiagram{2,1}, \ydiagram{1,1}, \ydiagram{1,1}, \ydiagram{2,1}, \ydiagram{2}, \ydiagram{3}, \ydiagram{3,1}),
    \end{align*}
we go over the following process as described above
    \begin{align*}
        \emptyset\quad \ydiagram{1}\quad \ydiagram{2}\quad \ydiagram{2,1}\quad \ydiagram{1,1}\quad \ydiagram{1,1}\quad \ydiagram{2,1}\quad \ydiagram{2}\quad \ydiagram{3}\quad \ydiagram{3,1}\\
        \emptyset
        \quad
        \begin{ytableau}
            1
        \end{ytableau}
        \quad
        \begin{ytableau}
           1&2
        \end{ytableau}
        \quad
        \begin{ytableau}
            1&2\\
            3
        \end{ytableau}
        \quad
        \begin{ytableau}
            1\\
            3
        \end{ytableau}
        \quad
        \begin{ytableau}
            1\\
            3
        \end{ytableau}
        \quad
        \begin{ytableau}
            1&6\\
            3
        \end{ytableau}
        \quad
        \begin{ytableau}
            1 &3
        \end{ytableau}
        \quad
        \begin{ytableau}
            1 &3 &8
        \end{ytableau}
        \quad
        \begin{ytableau}
            1 & 3  &8\\
            9
        \end{ytableau}
    \end{align*}
    to obtain $T=\begin{ytableau} 1&3&8\\ 9\end{ytableau}$ and $I = 
    \begin{pmatrix}
        4 & 5 & 7  \\
        2 & 5 & 6 
    \end{pmatrix}$ such that $\Phi^{BC}_1(G)=(T,I)$. We have
    \begin{align*} 
    \bar{I}=\begin{pmatrix}
        2 & 4 & 5 & 6 & 7  \\
        4 & 2 & 5 & 7 & 6 
    \end{pmatrix}
    \qquad \text{and}\qquad 
     \tab(\bar{I})=\begin{ytableau}
            2 & 4  \\
        5 & 7  \\
        6\\
        \end{ytableau}.
    \end{align*}
    Column reading word of $\tab(\bar{I})$ is $65274$ and we have
    \begin{align*}
         T\leftarrow 652=\begin{ytableau}
            1&3&8\\
            9
        \end{ytableau}
        \leftarrow
        652=\begin{ytableau}
            1&2&5\\
            3\\
            6\\
            8\\
            9
        \end{ytableau}
        \qquad \text{and} \qquad \begin{ytableau}
            1&2&5\\
            3\\
            6\\
            8\\
            9
        \end{ytableau}\leftarrow 74=\begin{ytableau}
            1&2&4&7\\
            3 &5\\
            6 \\
            8 \\
            9
        \end{ytableau}.
    \end{align*}
 Therefore we obtain
 \begin{align*}
        P=\begin{ytableau}
            1&2&4&7\\
            3 &5\\
            6 \\
            8 \\
            9
        \end{ytableau}\qquad \text{and}\qquad
        Q=\begin{ytableau}
            \  &\  &\ &2\\
            \  &2\\
            1 \\
            1 \\
            1
        \end{ytableau}.
    \end{align*}
    where $\Phi_2^{BC}(T,I)=(P,Q)$.
\end{example}

Now, we state a generalization of \cite[Theorem 9.4]{Sundaram1986}.

\begin{proposition}\label{prop: GSOT LR}
   Let $G \in \GSOT(\lambda, n)$, and let $(P, Q) = \Phi^{BC}(G)$. Then $c(G) \leq g$ if and only if the following conditions hold:
   \begin{itemize}
       \item The entries $(2i+1)$ in $Q$ appear in the $(m+i)$-th column or to the left.
       \item The entries $(2i)$ in $Q$ appear in the $(m+i-k)$-th column or to the left.
       \item $\lambda_1 \leq m-k$.
   \end{itemize} 
   Here $2g = 2m - k$, where $m$ is an integer and $k = 0$ or $1$.
\end{proposition}

In Example \ref{ex:r of rcr}, we have $c(G)=3$. We can check that $1$'s in $Q$ appear in the third column or to the left and $2$'s in $Q$ appear in the fourth column or to the left.

To prove Proposition \ref{prop: GSOT LR}, we establish the following two lemmas. In particular, Lemma \ref{lem: Sundarma 9.3} generalizes \cite[Lemma 9.3]{Sundaram1986}. For a tableau $T$, we denote $(T)_1$ as the length of the first row of the shape of $T$. Additionally, for two semistandard Young tableaux $P$ and $Q$, the notation $P \leftarrow Q$ represents $P \leftarrow w$, where $w$ is a column reading word of $Q$. 

\begin{lem}\label{lem: Sundarma 9.3}
    For $G \in \GSOT(\lambda,n)$, let $(T,I)=\Phi^{BC}_{1}(G)$ and denote  $I$ by 
    $\begin{pmatrix}
        j_1 & \dots & j_r \\
        i_1 & \dots & i_r 
    \end{pmatrix}$ 
    and  $\hat{I}$ by 
    $\begin{pmatrix}
        c_1 & \dots & c_s \\
        d_1 & \dots & d_s 
    \end{pmatrix}$. 
    Consider tableaux $Z^{(1)}=(T \leftarrow i_r \cdots i_1)$ and $Z^{(2)}=(T\leftarrow d_s\dots d_1)$. Then $c(G) \leq g$ if and only if $(Z^{(1)})_1 \leq m$ and  $(Z^{(2)})_1\leq m-k$, where we represent $2g=2m-k$ for $k=0$ or $1$ and an integer $m$.
\end{lem} 

\begin{proof}
    Let $G = (G_0, \dots, G_n)$. We proceed by induction on $r$. The base case $r = 0$ is trivial. Now assume $r \geq 1$, and that the statement holds for $r - 1$. Denote $G' = (G_0, \dots, G_{j_r - 1})$.

    First, consider the case $k = 0$, i.e., $g = m \in \mathbb{Z}$, and assume $c(G) \leq m$. 

    (Case 1: $j_r = i_r$) 
      In this case, $G^{(j_r - 1)} = G^{(j_r)}$. Since $(G^{(j_r - 1)})_1 > m - 1$ would imply $c(G) \geq m + \frac{1}{2}$, we must have $(G^{(j_r - 1)})_1 \leq m - 1$. Consequently, $(G^{(j_r - 1)} \leftarrow i_r)_1 \leq m$. By the induction hypothesis, $(G^{(j_r - 1)} \leftarrow i_{r - 1} \cdots i_1)_1 \leq m$. Since $i_r$ is the largest among $\{i_1, \dots, i_r\}$, it follows that $(G^{(j_r - 1)} \leftarrow i_r \cdots i_1)_1 \leq m$.  
      Now, note that $T$ contains $G^{(j_r)}$, and $T / G^{(j_r)}$ consists of letters strictly larger than $j_r$. By the assumption $c(G) \leq m$, we conclude $(T)_1 \leq m$, and therefore, $(T \leftarrow i_r \cdots i_1)_1 \leq m$.  
      Similarly, by the induction hypothesis, $(G^{(j_r - 1)} \leftarrow d_s \cdots d_1)_1 \leq m$, and the same reasoning gives $(T \leftarrow d_s \cdots d_1)_1 \leq m$.  

    (Case 2: $j_r > i_r$)  
      Here, $G^{(j_r - 1)} = (G^{(j_r - 1)} \leftarrow i_r)$. Combining this with the induction hypothesis, we conclude $(G^{(j_r)} \leftarrow i_r \cdots i_1)_1 \leq m$ and $(G^{(j_r)} \leftarrow d_s \cdots d_1)_1 \leq m$ (note that $d_s = i_r$ in this case). By the same reasoning, $(T \leftarrow i_r \cdots i_1)_1 \leq m$ and $(T \leftarrow d_s \cdots d_1)_1 \leq m$.

    Conversely, assume $(Z^{(1)})_1 \leq m$ and $(Z^{(2)})_1 \leq m$.  
    If $j_r = i_r$, we must have $(G^{(j_r)})_1 \leq m - 1$, as $(G^{(j_r)} \leftarrow i_r)_1 > m$ would contradict $(Z^{(1)})_1 \leq m$. By the induction hypothesis, $(G^{(j_r - 1)} \leftarrow i_{r - 1} \cdots i_1)_1 \leq m$ and $(G^{(j_r - 1)} \leftarrow d_{s - 1} \cdots d_1)_1 \leq m$, so $c(G') \leq m$. Combined with $(G^{(j_r)})_1 \leq m - 1$, we conclude $c(G) \leq m$.  
    If $j_r > i_r$, we trivially have $(G^{(j_r)})_1 \leq m$. Since $(G^{(j_r - 1)} \leftarrow i_{r - 1} \cdots i_1)_1 = (G^{(j_r)} \leftarrow i_r \cdots i_1)_1 \leq (Z^{(1)})_1 \leq m$ and $(G^{(j_r - 1)} \leftarrow d_s \cdots d_1)_1 \leq (Z^{(2)})_1 \leq m$, the induction hypothesis gives $c(G') \leq m$. Combining this with $(G^{(j_r)})_1 \leq m$, we conclude $c(G) \leq m$.

    The proof for the case $k = 1$, i.e., $g \in \mathbb{Z}+\frac{1}{2}$, follows similarly.
\end{proof}

\begin{lem}\label{lem: gsot aux}
    Let $T$ be a semistandard Young tableau with distinct entries and $P$ be any semistandard Young tableau. 
    If $(T)_1$ is an odd number $2k+1$, then $(P\leftarrow T)_1-(P\leftarrow \alpha(T))_1\leq k$ and equality holds when $(P\leftarrow \alpha(T))_1-(P\leftarrow \alpha(T)^{\leq k})_1>0$.
    If $(T)_1$ is an even number $2k$, then $(P\leftarrow T)_1-(P\leftarrow \beta(T))_1\leq k$ and equality holds when $(P\leftarrow \beta(T))_1-(P\leftarrow \beta(T)^{\leq k-1})_1>0$.
\end{lem}

\begin{proof}
    We will prove the case when $(T)_1$ is an odd number $2k+1$; the proof for the even case follows similarly.   
    We take $I= 
    \begin{pmatrix}
        j_1 & \dots & j_r \\
        i_1 & \dots & i_r 
    \end{pmatrix}
    \in \TL(r)$ such that $\tab(\bar{I})=T$ and denote $\bar{I}=
    \begin{pmatrix}
        a_1 & \dots & a_s \\
        b_1 & \dots & b_s 
    \end{pmatrix}$. 
    Let $w$ be a word such that $T=\emptyset\leftarrow w$, then $(P\leftarrow T)=(\emptyset\leftarrow w b_s\dots b_1)$. 
    By Greene's Theorem \cite[A1.1.1]{EC2}, $(P\leftarrow T)_1$ equals the length of the longest increasing subsequence of $w b_s\dots b_1$ and denote this subsequence as $(w_{c_1},\dots,w_{c_m},b_{d_1},\dots,b_{d_{\ell}})$.
    If the set $\{b_{d_1},\dots,b_{d_{\ell}}\}$ contains more than $(k+1)$-many elements in $\{j_1,\dots,j_r\}\setminus\{i_1,\dots,i_r\}$ denoted by $j_{u_{1}}<\dots<j_{u_{k+1}}$ then we have the length $(2k+2)$ increasing subsequence $(i_{u_{k+1}} , \dots , i_{u_1} , j_{u_1} , \dots,j_{u_{k+1}})$ of the word $b_s\dots b_1$. 
    This is the contradiction as  $(T)_1=2k+1<2k+2$. 
    Therefore the  $\{b_{d_1},\dots,b_{d_{\ell}}\}$ contains at most $k$-many elements in $\{j_1, \dots ,j_r\}\setminus\{i_1,\dots,i_r\}$ and removing these yields an increasing subsequence of $w i_r\dots i_1$. Hence, we conclude $(P\leftarrow T)_1-(P\leftarrow \alpha(T))_1\leq k$.

    Next, note that $\alpha(T)_1 = k+1$ (see Lemma~\ref{lem: shape alpha beta}). Denote the reading word of $\alpha(T)$ as $C_1 \dots C_{k+1}$, where $C_i$ is the word of the $i$-th column. Let $M$ be the rightmost entry in the first row of $(P \leftarrow C_1 \dots C_{k+1})$. The condition 
    $
    (P \leftarrow C_1 \dots C_{k+1})_1 - (P \leftarrow C_1 \dots C_k)_1 > 0
    $  
    implies that $M$ belongs to $C_{k+1}$. Among the longest increasing subsequences of $w i_r \dots i_1$, let $u = (u_1, \dots, u_y)$ such that $u_y = M$.  

    Since $M$ is in $C_{k+1}$, during the insertion process $\emptyset \leftarrow i_r \dots i_1$, $M$ is placed in the $(k+1)$-th column of the first row (otherwise, $M$ could not remain in the $(k+1)$-th column after insertion). Hence, among the longest subsequences of $i_r \dots i_1$, we can select $(i_{z_1}, \dots, i_{z_k}, M)$. Appending the sequence $(j_{z_k}, \dots, j_{z_1})$ to $u$ forms an increasing subsequence of $w b_s \dots b_1$. Therefore, the equality 
    $
    (P \leftarrow T)_1 - (P \leftarrow \alpha(T))_1 = k
    $ 
    holds.   
\end{proof}

\begin{proof}[Proof of Proposition \ref{prop: GSOT LR}]
    We retain the notations from Lemma \ref{lem: Sundarma 9.3} and set $Y = \tab(\bar{I})$.  First, assume $c(G) \leq g$.  
    Suppose there exists $(2i+1)$ in $Q$ appearing in the $(m+i+1)$-th column or further to the right. This is equivalent to $(T \leftarrow Y^{\leq 2i+1})_1 > m + i$. By Lemma \ref{lem: useful lemma} and \ref{lem: gsot aux}, we have 
    $
    (T \leftarrow \alpha(Y)^{\leq i+1})_1 > m,
    $
    which implies 
    $
    (Z^{(1)})_1 = (T \leftarrow \alpha(Y))_1 > m.
    $ 
    This contradicts Lemma \ref{lem: Sundarma 9.3}.  
 A similar argument shows that if there exists $(2i)$ in $Q$ appearing in the $(m + i + 1 - k)$-th column or further to the right, we also reach a contradiction. Finally, the condition $\lambda_1 \leq m - k$ holds trivially.

Now assume the three conditions in Proposition \ref{prop: GSOT LR} hold.
    Suppose $(Z^{(1)})_1 > m$, then there exists $i$ such that 
    $
    (T \leftarrow \alpha(Y)^{\leq i})_1 = m$ and $(T \leftarrow \alpha(Y)^{\leq i+1})_1 = m + 1$,
    since $(T)_1 = \lambda_1 \leq m$.  
    By Lemma \ref{lem: useful lemma} and \ref{lem: gsot aux}, we have 
    $
    (T \leftarrow Y^{\leq 2i+1})_1 = (T \leftarrow \alpha(Y)^{\leq i+1})_1 + i = m + i + 1.
    $
    and this contradicts the first condition. Thus, we deduce that $(Z^{(1)})_1 \leq m$.    Similarly, we can show that $(Z^{(2)})_1 \leq m - k$ by applying the same reasoning.  
\end{proof}

\begin{proof}[Proof of Lemma \ref{lem: x=k filter} when $\diamond=\sboxone$ or $\sboxeleven$]
    For $\lambda\in \Par_n$ and $\mu\vdash n$, there exists a bijection $\eta_\mu$ with the following commutative diagram \cite[Proposition 38 and 40]{S05}:   
    $$
    \begin{tikzcd}
        \HW(B_{\mu}(\diamond),\lambda) \arrow[r, "\eta_{\mu}"] \arrow[d, "S^{\diamond}_{\mu}"'] & \bigcup_{\gamma\in P_n^{\diamond}} \bigcup_{\nu} \HW(B_{\mu}(\emptyset),\nu) \times \LR^{t}(\nu / \lambda;\gamma^{t}) \arrow[d, "S^{\emptyset}_{\mu} \times \text{id}"] \\
        \HW(B_{(1^n)}(\diamond),\lambda) \arrow[r, "\eta_{(1^n)}"] & \bigcup_{\gamma\in P_n^{\diamond}} \bigcup_{\nu} \HW(B_{(1^n)}(\emptyset),\nu) \times \LR^{t}(\nu / \lambda;\gamma^{t})
    \end{tikzcd}
    $$
    where $S^{\diamond}_{\mu}$ is the splitting map. Additionally, the following properties hold:
    \begin{itemize}
        \item For $b \in \HW(B_{\mu}(\diamond),\lambda)$, if $\eta_{\mu}(b) = (b', Q)$, then 
        $\frac{|\mu| - |\lambda|}{2} - \frac{|\diamond|}{2} \overline{D}(b) = -\overline{D}(b').$
        \item Identify $\HW(B_{(1^n)}(\sboxone),\lambda)$ with $\GSOT(\lambda,n)$, $\HW(B_{(1^n)}(\sboxtwo),\lambda)$ with $\SOT(\lambda,n)$, and $\HW(B_{(1^n)}(\emptyset),\nu)$ with $\SYT(\nu)$ via the map $\phi_r$. Then $\eta_{(1^n)}$ coincides with $\Phi^{BC}$ when $\diamond = \sboxone$ and its restriction \eqref{eq: gsot to sot restriction} when $\diamond = \sboxtwo$.
    \end{itemize}

    First, assume $\diamond = \sboxone$. 
    By Proposition \ref{prop: GSOT LR}, for any $G \in \GSOT(\lambda,n)$, the value of $c(G)$ can be determined precisely from the transposed LR tableau $Q$ obtained via $\Phi^{BC}(G) = (P, Q)$. 
    For any positive integer $M$, define a subset $\left(\LR^{t}(\nu / \lambda; \gamma^{t})\right)^{\leq M}$ of $\LR^{t}(\nu / \lambda; \gamma^{t})$ such that $c(G) \leq \frac{M}{2}$ if and only if $Q \in \left(\LR^{t}(\nu / \lambda; \gamma^{t})\right)^{\leq M}$, where $\Phi^{BC}(G) = (P, Q)$. 
    Since $\epsilon_0(S^{\sboxone}_{\mu}(b)) = \epsilon_0(b)$, the bijection $\eta_{\mu}$ restricts to:
    $$
    \{b \in \HW(B_{\mu}(\sboxone), \lambda) : \epsilon_0(b) \leq M \} \rightarrow \bigcup_{\gamma \in P_n^{\sboxone}} \bigcup_{\nu} \HW(B_{\mu}(\emptyset), \nu) \times \left(\LR^{t}(\nu / \lambda; \gamma^{t})\right)^{\leq M}.
    $$
    Thus, we have
    \begin{align*}
        q^{||\mu|| + \frac{|\mu| - |\lambda|}{2}} \left( \sum_{\substack{b \in \HW(B_{\mu}(\sboxone), \lambda) \\ \epsilon_0(b) \leq M}} q^{-\frac{|\diamond| \overline{D}(b)}{2}} \right) 
        &= \sum_{\nu} \sum_{b' \in \HW(B_{\mu}(\emptyset), \nu)} q^{||\mu|| - \overline{D}(b')} \left( \sum_{\gamma \in P^{\sboxone}_{n}} |\left(\LR^{t}(\nu / \lambda; \gamma^{t})\right)^{\leq M}| \right) \\
        &= \sum_{\nu} \KL^{A_{n-1}}_{\nu, \mu}(q) \left( \sum_{\gamma \in P^{\sboxone}_{n}} |\left(\LR^{t}(\nu / \lambda; \gamma^{t})\right)^{\leq M}| \right),
    \end{align*}
    where we used 
    $$\sum_{b' \in \HW(B_{\mu}(\emptyset), \nu)} q^{||\mu|| - \overline{D}(b')} = \KL^{A_{n-1}}_{\nu, \mu}(q)$$ 
    from \cite{NY1997}. 
    This completes the proof. The case $\diamond = \sboxtwo$ follows in exactly the same manner.
\end{proof}

\subsubsection{$\diamond=\sboxeleven$}
We begin by summarizing the results from \cite{LS2007}.
For a nonnegative integer vector $\mu$, we denote $B^{D_n}_{\mu}$ for $B^{D_n}(\mu_{\ell}\omega_1)\otimes \dots \otimes B^{D_n}(\mu_{1}\omega_1)$ where each $B^{D_n}(s \omega_1)$ is the simple $D_n$-crystal with the highest weight $s\omega_1$. 
Obviously, for partitions $\mu$ and $\nu$ such that $|\mu| = |\nu|$, there exists a bijection
\[
    \Psi: \HW(B_{\mu}(\emptyset), \nu) \rightarrow \HW(B^{D_n}_{\mu}, \nu),
\]
which is defined by interpreting each $b \in \HW(B_{\mu}(\emptyset), \nu)$ as an element of $B^{D_n}_{\mu}$. 
Next, we define the set $E_{\lambda, \mu}$ as the collection of $b \in B^{D_n}_{\mu}$ that satisfies the following conditions:
\begin{itemize}
    \item $b$ is an $A_{n-1}$-highest weight element, i.e., $e_i(b) = 0$ for $1 \leq i \leq n-1$,
    \item $\wt(b) = \lambda$,
    \item $\wt(\hw(b))$ is a partition $\nu$ such that $|\nu| = |\mu|$.
\end{itemize}

\begin{proposition}\cite[Proposition 35 and 40]{LS2007}\label{prop: LS2007}
    For $\lambda\in \Par_n$ and $\mu\vdash n$, there exists a bijection $\theta_{\mu}: \HW(B_{\mu},\lambda)\rightarrow E_{\lambda^{*},\mu}$ where $\lambda^{*}:=(-\lambda_n,\dots,-\lambda_2,-\lambda_1)$\footnote{In \cite{LS2007} they used $D_n^{\dagger}$ convention instead of the usual $D_n$.
    Given $\nu$ is a consequence of the modification. } and a classical crystal embedding $L_{\mu}: B^{D_n}_{\mu}\rightarrow B^{D_n}_{(1^n)}$ with a commutative diagram 
    \[
    \begin{tikzcd}
        \HW(B_{\mu}(\hspace{0.4mm}\sboxeleven\hspace{0.4mm}),\lambda) \arrow[r, "\theta_{\mu}"] \arrow[d, "S_{\mu}"'] & E_{\lambda^{*},\mu} \arrow[d, "L_{\mu}"] \\
        \HW(B_{(1^n)}(\hspace{0.4mm}\sboxeleven\hspace{0.4mm}),\lambda) \arrow[r, "\theta_{(1^n)}"] & E_{\lambda^{*},(1^n)}.
    \end{tikzcd}
    \]
    
    Moreover, for $b\in\HW(B_{\mu}(\hspace{0.4mm}\sboxeleven\hspace{0.4mm}),\lambda)$ we have  
    \begin{equation}
        \overline{D}(b)=\overline{D}(b')+\frac{|\mu|-|\lambda|}{2}
    \end{equation}
    where $b'\in \HW(B_{\mu}(\emptyset))$ such that $\Psi(b')=\hw(\theta_{\mu}(b))$.
\end{proposition}

In particular, we provide an explicit description of $\theta_{(1^n)}$ \cite[Proposition 33]{LS2007}.
Given an element $a_n \otimes \dots \otimes a_2 \otimes a_1 \in \HW(B_{(1^n)}(\hspace{0.4mm}\sboxeleven\hspace{0.4mm}), \lambda)$, we associate to each $a_i$ an element $a'_i$ as follows: if $a_i = c$ is an unbarred letter, then set $a'_i = \overline{n+1-c}$, if $a_i = \bar{c}$ is a barred letter, then set $a'_i = n+1-c$. We have $\theta_{(1^n)}(a_n\otimes \dots \otimes a_1)=a'_n\otimes \dots \otimes a'_1$.

Next, we introduce the Plactic monoid relation as described in \cite{Lecouvey2003}.
For an element $w = w_m \otimes \dots \otimes w_2 \otimes w_1 \in B^{D_n}_{(1^m)} = (B^{D_n}(\omega_1))^{\otimes m}$, we identify $w$ with the word $w_m \dots  w_1$. 
Lecouvey demonstrated that two elements $v, w \in B^{D_n}_{(1^m)}$ lie in the same position within two isomorphic classical components of $B^{D_n}_{(1^m)}$ if and only if $v$ and $w$ are related by the following Plactic monoid relation \cite[Definition 3.2.3]{Lecouvey2003}:

$R_1$: If $x\neq \bar{z}$
\begin{equation*}
    xzy\equiv zxy \quad \text{for $x\preceq y \prec z$} \quad \text{and} \quad yzx\equiv yxz \quad \text{for $x\prec y \preceq z$}.
\end{equation*}

$R_2$: If $1\prec x \preceq n$ and $x\preceq y \preceq \bar{x}$
\begin{equation*}
    (x-1)\overline{(x-1)}y\equiv x\overline{x}y \quad \text{and} \quad y\overline{x}x\equiv y(x-1)\overline{x-1}.
\end{equation*}

$R_3$: If $x\preceq n-1$
\begin{align*}
    \begin{cases*}
        n\bar{x}\bar{n}\equiv n\bar{n}\bar{x}\\
        \bar{n}\bar{x}n\equiv\overline{n}n\overline{x}
    \end{cases*}\quad \text{and} \quad
    \begin{cases*}
        xn\overline{n}\equiv nx\overline{n}\\
        x\bar{n}n\equiv\overline{n}xn
    \end{cases*}.
\end{align*}

$R_4$: 
\begin{align*}
    \begin{cases*}
        \bar{n}\bar{n}n\equiv \bar{n}(n-1)\overline{(n-1)}\\
        nn\bar{n}\equiv n(n-1)\overline{(n-1)}
    \end{cases*}\quad \text{and} \quad
    \begin{cases*}
         (n-1)\overline{(n-1)}\bar{n}\equiv n\bar{n}\bar{n}\\
         (n-1)\overline{(n-1)}\equiv\overline{n}nn
    \end{cases*}.
\end{align*}

\begin{rmk} 
    There is a relation $R_5$ in \cite[Definition 3.2.3]{Lecouvey2003}, referred to as the contraction relation. 
    However, since we are only considering two elements $v, w \in B^{D_n}_{(1^m)}$, words of the same length, we do not require the relation $R_5$. 
\end{rmk}

Note that we can naturally identify the two sets $\SOT(\lambda, n)$ and $\HW(B_{(1^n)}(\hspace{0.4mm}\sboxeleven\hspace{0.4mm}), \lambda)$ via the map $\phi_r$ given in Lemma \ref{lem: rind} (3).
From this viewpoint, we can describe $\epsilon_0(b)$ for $b \in \HW(B_{(1^n)}(\hspace{0.4mm}\sboxeleven\hspace{0.4mm}), \lambda)$ in terms of the associated element in $T \in \SOT(\lambda, n)$: Let $T = (\lambda^{(0)}, \dots, \lambda^{(n)}) \in \SOT(\lambda, n)$, then we have
\begin{equation}\label{eq: sot epsilon}
\epsilon_0(b) = 2c(T) = \max(\lambda^{(i)}_1 + \lambda^{(i)}_2 : 0 \leq i \leq n).
\end{equation}

The following lemma is straightforward to verify if we utilize \eqref{eq: sot epsilon}, we leave it as an exercise.

\begin{lem}\label{lem: type D plactic}
    For $b \in \HW(B_{(1^n)}(\hspace{0.4mm}\sboxeleven\hspace{0.4mm}), \lambda)$, let $v$ be an element in $B^{D_n}_{(1^n)}$ that is related to $\theta_{(1^n)}(b)$ by one of the relations $R_1, \dots, R_4$. 
    Then we have $v = \theta_{(1^n)}(b')$ for some $b' \in \HW(B_{(1^n)}(\hspace{0.4mm}\sboxeleven\hspace{0.4mm}), \lambda)$, and moreover we have $\epsilon_0(b) = \epsilon_0(b')$.
\end{lem}

\begin{example}
    Consider $T=(\emptyset,\ydiagram{1},\ydiagram{1,1},\ydiagram{1},\emptyset)\in \SOT(\emptyset,n)$ for $n=4$ which goes to $b=\bar{1}\otimes \bar{2}\otimes 2\otimes 1$ via the map $\phi_r$.
    We have 
    \begin{equation*}
        \theta_{(1^n)}(b)=n\otimes (n-1)\otimes \overline{n-1}\otimes \bar{n}\equiv n\otimes n \otimes \bar{n} \otimes\bar{n}
    \end{equation*} 
    via the relation $R_4$.
    Note that for $T'=(\emptyset,\ydiagram{1},\ydiagram{2},\ydiagram{1},\emptyset)\in \SOT(\emptyset,n)$ we have $\theta_{(1^n)}(\phi_r(T'))=n\otimes n \otimes \bar{n} \otimes\bar{n}$. 
    The value $c(T)$ is attained at the partition $\ydiagram{1,1}$ which is $\frac{2}{2}=1$ and $c(T')$ is attained at the partition $\ydiagram{2}$ which is also $\frac{2}{2}=1$.
    We have $c(T)=c(T')$ which is consistent with Lemma \ref{lem: type D plactic}.
\end{example}

\begin{proof}[Proof of Lemma \ref{lem: x=k filter} when $\diamond=\sboxeleven$]
    We retain the notation from Proposition \ref{prop: LS2007}. 
    Let $G_{\nu,\tau}$ denote the set of $A_{n-1}$-highest weight elements $b \in B^{D_n}(\nu)$ such that $\wt(b) = \tau$. 
    There exists a natural bijection
    \[
        W_{\mu}: E_{\lambda^{*},\mu} \to \bigcup_{\nu \vdash n} \HW(B_{\mu}(\emptyset), \nu) \times G_{\nu,\lambda^{*}}
    \]
    which associates to each $b \in E_{\lambda^{*},\mu}$ a pair $(b', c)$.
    Here, $b' \in \HW(B_{\mu}(\emptyset), \nu)$ is such that $\Psi(b') = \hw(b)$, and $c$ is the image of $b$ under the isomorphism between the classical component containing $b$ and $B^{D_n}(\nu)$.
    As $L_{\mu}$ is a classical crystal embedding, for each $\nu$, $L_{\mu}$ induces an embedding 
    \begin{equation*}
        \HW(B_{\mu}(\emptyset),\nu)\times G_{\nu,\lambda^{*}} \rightarrow \HW(B_{(1^n)}(\emptyset),\nu)\times G_{\nu,\lambda^{*}}
    \end{equation*}
    acting as an identity on the second component $G_{\nu,\lambda^{*}}$
    
    Now, consider two elements $b, b' \in \HW(B_{(1^n)}(\hspace{0.4mm}\sboxeleven\hspace{0.4mm}),\lambda)$.
    If the second entries of $W_{(1^{n})} \circ \theta_{(1^{n})}(b)$ and $W_{(1^{n})} \circ \theta_{(1^{n})}(b')$ are identical, then we must have $\theta_{(1^{n})}(b) \equiv\theta_{(1^{n})}(b')$.
    By Lemma \ref{lem: type D plactic}, this implies that $\epsilon_0(b) = \epsilon_0(b')$. 
    Therefore there exists a subset $G^{\leq M}_{\nu, \lambda^{*}} \subset G_{\nu, \lambda^{*}}$ such that $\epsilon_0(b) \leq M$ if and only if $Q \in G^{\leq M}_{\nu,\lambda^{*}}$, where $W_{(1^{n})} \circ \theta_{(1^{n})}(b) = (b', Q)$.
    From this point on, the proof follows identically to the case when $\diamond = \sboxone$ or $\sboxtwo$.
\end{proof}

\subsection{Proof of Theorem \ref{thm: level formula} when $q=1$}\label{sub: level q=1}
We prove Theorem \ref{thm: level formula} at $q=1$. 
By Lemma \ref{lem: rind}, we need to show
   \begin{align}\label{eq: q=1 B}
        |\GSSOT_g(\hat{\lambda},\hat{\mu})|=[x^\mu]s^{B_n}_{\lambda}  \\ 
        \label{eq: q=1 C}
        |\SSOT_g(\hat{\lambda},\hat{\mu})|=[x^\mu]s^{C_n}_{\lambda}    \\
        |\SSROT_g(\hat{\lambda},\hat{\mu})|=[x^\mu]s^{D_n}_{\lambda}   \label{eq: q=1 D}
    \end{align}
where $g$ is given so that $\hat{\lambda}=\oc(\lambda,g)$ and $\hat{\mu}=\oc(\mu,g)$ are nonnegative integer vectors. Note that \eqref{eq: q=1 C} was proved in \cite{Lee2023} by presenting an explicit bijection with a certain set of King tableaux, which is equinumerous to $[x^\mu]s^{C_n}_{\lambda}$ \cite{King75}. It remains to prove \eqref{eq: q=1 B} and \eqref{eq: q=1 D}. For each equation, we proceed separately depending on whether $\lambda$ and $\mu$ are spin weights or not. Recall that we need to take $g\in \mathbb{Z}+\frac{1}{2}$ when $\lambda$ and $\mu$ are spin weights.

Our main ingredients for the proof are Cauchy identities \cite{BG2006} and dual Pieri rules \cite{Okada2016}. 
We set up notations before presenting the proof. 
\begin{definition}
    For integers $0\leq r\leq n$, define polynomials 
    \begin{equation*}
        e^{(n)}_r(x_1,x_2,\dots,x_n) = \sum_{i=0}^{r} \left(\sum_{\substack{S\subseteq[n]\\|S|=i}} \sum_{\substack{S'\subseteq[n]\\|S'|=r-i}} (\prod_{s\in S}x_s \prod_{s'\in S'} x^{-1}_{s'}) \right).
    \end{equation*}
    Expressing $e^{(n)}_r$ in terms of characters for each type, we have
    \begin{align*}
        &e^{(n)}_r=\sum_{i=0}^{r}(-1)^{i}s^{B_n}_{({1^{r-i}})}=\sum_{i=0}^{r/2}s^{C_n}_{({1^{r-2i}})}=\begin{cases*}
            s^{D_n}_{(1^{r})} \qquad &\text{if $r<n$}\\
            s^{D_n}_{(1^{n})}+ s^{D_n}_{(1^{n-1},-1)} \qquad &\text{if $r=n$.}
        \end{cases*}       
    \end{align*}
\end{definition}
From now on, we fix positive integers $g$ and $n$. 
For a partition $\lambda$ inside a rectangle $(g^n)$, i.e., $\ell(\lambda)\leq n$ and $\lambda_1\leq g$, we denote $\tilde{\lambda}$ to be a conjugate partition of $\oc(\lambda,g)$.
For $\lambda\in \Par_n$ we let $m_{\lambda}(x)$ to be an $W$-invariant Laurent polynomial given by $m_{\lambda}(x_1,x_2,\dots,x_n):=\sum_{w\in W}w(x^{\lambda})$ where $W=\mathfrak{S}_n \ltimes (\mathbb{Z}/2\mathbb{Z})^n$ (the Weyl group for type $B_n$ or $C_n$). 

For partitions $\lambda$ and $\mu$, we define
\begin{align*}
    &\kappa^{B}(\lambda,\mu,r;g)=\text{set of $g$-bounded gohs $T$ of length $r$ such that $I(T)=\mu$ and $F(T)=\lambda$},\\
    &\kappa^{C}(\lambda,\mu,r;g)=\text{set of $g$-bounded ohs $T$ of length $r$ such that $I(T)=\mu$ and $F(T)=\lambda$},\\
    &\kappa^{D}(\lambda,\mu,r;g)=\text{set of $g$-bounded rohs $T$ of length $r$ such that $I(T)=\mu$ and $F(T)=\lambda$}.
\end{align*}

\subsubsection{Proof of \eqref{eq: q=1 B} when $\lambda$ and $\mu$ are not spin weights}
We state the Cauchy identity of type $B$ given in \cite{BG2006}.
\begin{thm}\cite[Lemma 6]{BG2006}\label{thm: Cauchy type B}
For indeterminates $x_1,\dots x_n$ and $t_1,\dots,t_g$ we have
    \begin{equation*}
        \sum_{\lambda\subseteq (g^n)}(-1)^{|\tilde{\lambda}|}m_{\lambda}(x)\prod_{i=1}^{n}e^{(g)}_{g-\lambda_i}(t)=\sum_{\lambda\subseteq (g^n)}(-1)^{|\tilde{\lambda}|}s^{B_n}_{\lambda}(x) s^{B_g}_{\tilde{\lambda}}(t).
    \end{equation*}
\end{thm}
\begin{rmk}
    In \cite{BG2006}, the left-hand side of Theorem \ref{thm: Cauchy type B} is stated as $\prod_{i=1}^{n}\prod_{j=1}^{g}(x_i+x_i^{-1}-t_{j}-t^{-1}_j).$
    It is elementary to check that this equals $\sum_{\lambda\subseteq (g^n)}(-1)^{|\tilde{\lambda}|}m_{\lambda}(x)\prod_{i=1}^{n}e^{(g)}_{g-\lambda_i}(t).$
\end{rmk}
 By Theorem \ref{thm: Cauchy type B}, we have $[m_{\mu}]s_\lambda^{B_n}= (-1)^{|\tilde{\lambda}|-|\tilde{\mu}|}[s_{\tilde{\lambda}}^{B_g}] \prod_{i=1}^n e_{g-\mu_i}^{(g)}$. Now we investigate the coefficient $[s^{B_n}_{\lambda}](e^{(n)}_r s^{B_n}_{\mu})$ (Proposition \ref{prop: type B dual pieri} and Lemma \ref{lem: 5.9 coeff}).
\begin{proposition}\cite[Theorem 4.1 (2)]{Okada2016}\label{prop: type B dual pieri}
    For $r\leq n$ and $\lambda\in \Par_n$, we have
    \begin{equation*}
        s^{B_n}_{(1^r)}s_{\mu}^{B_n}=\sum_{\lambda}K^{B_n}_{\lambda,\mu}(r)s_{\lambda}^{B_n}
    \end{equation*}
    where $K^{B_n}_{\lambda,\mu}(r)$ is the number of partitions $\xi$ satisfying:
    \begin{enumerate}
        \item $\ell(\xi)\leq n$
        \item $\xi / \mu$ and $\xi / \lambda$ are both vertical strips
        \item $|\xi / \mu|+|\xi / \lambda|=r$ or $r-1$
        \item if $\ell(\mu)<n$ one of the following holds:
        \begin{enumerate}
            \item $|\xi / \mu|+|\xi / \lambda|=r$ and $\ell(\xi)=\max(\ell(\mu),\ell(\lambda)),$
            \item $|\xi / \mu|+|\xi / \lambda|=r-1$ and $\ell(\xi)=n$.
        \end{enumerate}
    \end{enumerate}
\end{proposition}

\begin{lem}\label{lem: 5.9 coeff}
    Let $K^{B_n}_{\lambda,\mu}(r)$ be the number defined in the previous proposition. Then we have 
    \begin{equation*}
        \sum_{i=0}^{r}(-1)^{i}K^{B_n}_{\lambda^t,\mu^t}(r-i)=(-1)^{|\lambda|-|\mu|-r}|\kappa^{B}(\lambda,\mu,r;n)|.
    \end{equation*}
\end{lem}
\begin{proof}
Let $A(i) = \{\xi : \text{$(\mu^t, \xi^t, \lambda^t)$ satisfies the conditions in Proposition \ref{prop: type B dual pieri} for $i$}\}$ so that $|A(i)| = K^{B_n}_{\lambda^t, \mu^t}(i)$.

First, consider the case $\mu_1 = n$. In this scenario, condition (4) of Proposition \ref{prop: type B dual pieri} becomes redundant. As a result, each $A(i)$ represents a collection of $\xi$ such that $(\mu, \xi, \lambda)$ forms an $n$-bounded ohs of length either $i$ or $i-1$. Therefore, the statement follows immediately.

Now, consider the case $\mu_1 < n$ and assume further that $|\lambda| - |\mu| - r$ is an odd number. In this case, $A(r-2i)$ is the set of $\xi$ such that $(\mu, \xi, \lambda)$ forms an $n$-bounded ohs of length $(r-2i-1)$ with $\xi_1 = n$, while $A(r-2i-1)$ is the set of $\xi$ such that $(\mu, \xi, \lambda)$ forms an $n$-bounded ohs of length $(r-2i-1)$ with $\xi_1 = \max(\mu_1, \lambda_1)$.  For each $\xi \in A(r-2i)$, we define a mapping $\sigma(\xi) \in A(r-2j-1)$, where $\sigma(\xi) = (\max(\mu_1, \lambda_1), \xi_2, \xi_3, \dots)$. Let $B$ denote the subset of $\xi \in \bigcup_{i \geq 0} A(r-2i-1)$ such that $\xi$ is not in the image of $\sigma$. Then $\xi \in B$ if and only if
\begin{itemize}
    \item $(\mu, \xi, \lambda)$ is an $n$-bounded ohs of length less than $r$, 
    \item for $\xi' = (n, \xi_2, \xi_3, \dots)$, the ohs $(\mu, \xi', \lambda)$ has a length greater than $(r-1)$.
\end{itemize}
Observe that $\kappa^{B}(\lambda, \mu, r; n)$ is the collection of $(n-1)$-bounded ohs' $(\mu, \xi, \lambda)$ with length $r-1$. For any $(\mu, \xi, \lambda) \in \kappa^{B}(\lambda, \mu, r; n)$, we can associate an element $\xi' = (\max(\mu_1, \lambda_1), \xi_2, \xi_3, \dots) \in B$. This implies that $|B| = |\kappa^{B}(\lambda, \mu, r; n)|$.

The proof for the case where $|\lambda| - |\mu| - r$ is an even number follows similarly.
\end{proof}

\begin{proof}[Proof of \eqref{eq: q=1 B} when $\lambda$ and $\mu$ are not spin weights]
    Let $C(\lambda,\mu)=[s_{\lambda^t}^{B_g}] \prod_{i=1}^n e_{\mu_i}^{(g)}$.
    Note that $$|\GSSOT_{g}(\lambda,\mu)| = \sum_{\nu} |\kappa^{B}(\lambda,\nu,\mu_n;g)| |\GSSOT_{g}(\nu,\mu')|$$ where $\mu'=(\mu_1,\dots,\mu_{n-1})$.
    On the other hand, by Proposition \ref{prop: type B dual pieri} and Lemma \ref{lem: 5.9 coeff} we have
    $$C(\lambda,\mu)=\sum_\nu (-1)^{|\lambda|-|\nu|-|\mu_n|}|\kappa^{B}(\lambda,\nu,\mu_n;g)| C(\nu,\mu').$$ By induction it follows that $C(\lambda,\mu) = (-1)^{|\lambda|-|\mu|}|\GSSOT_g(\lambda,\mu)|$.
    Therefore, we conclude that $[m_\mu]s_\lambda^{B_n} = (-1)^{|\tilde{\lambda}|-|\tilde{\mu}|}[s_{\tilde{\lambda}}^{B_g}] \prod_{i=1}^n e_{g-\mu_i}^{(g)}= (-1)^{|\tilde{\lambda}|-|\tilde{\mu}|}C(\hat{\lambda},\hat{\mu})=|\GSSOT_g(\hat{\lambda},\hat{\mu})|$.
\end{proof}
\subsubsection{Proof of \eqref{eq: q=1 B} when $\lambda$ and $\mu$ are spin weights}
Our goal is to show 
\begin{equation*}
    |\GSSOT_{g+\frac{1}{2}}(\hat{\lambda},\hat{\mu})|=[x^{\mu^{\sharp}}]s^{B_n}_{\lambda^{\sharp}}
\end{equation*}
for $\lambda,\mu\in \Par_n$ and $\hat{\lambda}=\oc(\lambda,g)$, $\hat{\mu}=\oc(\mu,g)$. In \cite{BS1991}, it is shown that
\begin{equation*}
    [x^{\mu^{\sharp}}]s^{B_n}_{\lambda^{\sharp}}=\sum_{S\subseteq [n]}[x^{\mu+\varepsilon_S}]s^{C_n}_{\lambda}
\end{equation*}
where $\varepsilon_S=\sum_{i\in S}\varepsilon_i$ is a 0-1 vector given according to $S$. Now, \eqref{eq: q=1 C} and \eqref{eq: gssot decomposition} completes the proof.

\subsubsection{Proof of \eqref{eq: q=1 D} when $\lambda$ and $\mu$ are not spin weights}\label{subsub: D}
For $\lambda=(\lambda_1,\lambda_2,\dots,\lambda_n)\in \Par_n$, we let $\lambda^{+}=\lambda$ and $\lambda^{-}=(\lambda_1,\lambda_2,\dots,\lambda_{n-1},-\lambda_n)$
and define
\begin{equation*}
    \ps^{D_n}_{\lambda}:=s^{D_n}_{\lambda^{+}}+s^{D_n}_{\lambda^{-}}, \qquad \ns^{D_n}_{\lambda}:=s^{D_n}_{\lambda^{+}}-s^{D_n}_{\lambda^{-}}.
\end{equation*}
Note that $\ps^{D_n}_{\lambda}$ is now $W$-invariant Laurent polynomial for $W=\mathfrak{S}_n \ltimes (\mathbb{Z}/2\mathbb{Z})^n$. 
We also define $\flip(\lambda,g):=(2g-\lambda_1,\lambda_2,\dots,\lambda_n)$ and then  define
    \begin{equation*}
        P(\lambda,\mu)=|\SSROT_g(\lambda,\mu)|+|\SSROT_g(\flip(\lambda,g),\mu)|, \qquad  \quad N(\lambda,\mu)=|\SSROT_g(\lambda,\mu)|-|\SSROT_g(\flip(\lambda,g),\mu)|.
\end{equation*}
It is enough to show that $P(\hat{\lambda},\hat{\mu}) = [x^{\mu}]\ps^{D_n}_{\lambda}$ and $N(\hat{\lambda},\hat{\mu}) = [x^{\mu}]\ns^{D_n}_{\lambda}$. The outline of the proof is as follows. First, we prove Lemma \ref{lem: flip}, which enables us to write a recursive formula for $P(\lambda,\mu)$ and $N(\lambda,\mu)$. Then, we derive a recursive formula for $[x^{\mu}]\ps^{D_n}_{\lambda}$ by exploiting Theorem \ref{thm: Cauchy type D} and Proposition \ref{prop: type D dual pieri}, and a recursive formula for $[x^{\mu}]\ns^{D_n}_{\lambda}$ by exploiting \eqref{eq: weyl D}. Finally, we find that each recursive formula is identical by Lemma \ref{lem: rohs okada} and \ref{lem: dohs minus}.

\begin{lem}\label{lem: flip}
    For partitions $\lambda$ and $\mu$, assume further that $\lambda'=\flip(\lambda,g)$ and $\mu'=\flip(\mu,g)$ are partitions. Then we have $|\kappa^{D}(\lambda,\mu,r;g)|=|\kappa^{D}(\lambda',\mu',r;g)|$.
\end{lem}
\begin{proof}
Without loss of generality, we may assume $\lambda_1 \geq \mu_1$. Consider an rohs $(\mu, \xi, \lambda)$ in $\kappa^{D}(\lambda, \mu, r; g)$, and let $W = \rind(\mu, \xi, \lambda)$. Define $\psi(W)$ as the multi-set obtained from $W$ by replacing $(\lambda_1 - \mu_1)$-many $1$'s into $(\lambda_1 - \mu_1)$-many $\bar{1}$'s. We claim that $\psi(W) = \rind(T')$ for some $T' \in \kappa^{D}(\lambda', \mu', r; g)$.

Observe that $W$ contains $(\mu_1 - \xi_1)$ occurrences of $\bar{1}$, so $\psi(W)$, by construction, will have $(\lambda_1 - \xi_1)$ occurrences of $\bar{1}$. To ensure the existence of an rohs $T'$, the condition 
\[
(2g - \mu_1) - (\lambda_1 - \xi_1) \geq \max(\lambda_2, \mu_2)
\]
must hold. This condition is equivalent to the $g$-boundedness requirement for $(\mu, \xi, \lambda)$. It remains to verify that $T'$ satisfies the $g$-bounded condition. This follows from the inequality 
\[
(2g - \mu_1) + (\mu_1 - \xi_1) + \max(\mu_2, \lambda_2) \leq 2g,
\] 
which is valid because $\xi_1 \geq \max(\lambda_2, \mu_2)$. Finally, the map $\psi$ is clearly invertible, establishing a bijection.
\end{proof}

\begin{rmk}\label{rmk: D spin}
Lemma \ref{lem: flip} still holds if we replace $g$ by $g'=g+\frac{1}{2}$. The proof is identical.
\end{rmk}
\begin{thm}\cite[Lemma 5]{BG2006}\label{thm: Cauchy type D}
    For indeterminates $x_1,\dots x_n$ and $t_1,\dots,t_g$ we have
    \begin{equation*}
        \sum_{\lambda\subseteq (g^n)}(-1)^{|\tilde{\lambda}|}m_{\lambda}(x)\prod_{i=1}^{n}e^{(g)}_{g-\lambda_i}(t)=\sum_{\lambda\subseteq (g^n)}\frac{(-1)^{|\tilde{\lambda}|}\ps^{D_n}_{\lambda}(x) \ps^{D_g}_{\tilde{\lambda}}(t)}{2}.
    \end{equation*}
\end{thm}

\begin{proposition}\cite[Theorem 4.1 (3)]{Okada2016}\label{prop: type D dual pieri}
    For $r\leq n$ and $\lambda\in \Par_n$ we have
    \begin{equation*}
        e^{(n)}_r \ps_{\mu}^{D_n}=\sum_{\lambda}M(\lambda,\mu,n)K^{D_n}_{\lambda,\mu}(r)\ps_{\lambda}^{D_n}
    \end{equation*}
    where $K^{D_n}_{\lambda,\mu}(r)$ is the number of partitions $\xi$ satisfying:
    \begin{enumerate}
        \item $\ell(\xi)\leq n$
        \item $\xi / \mu$ and $\xi / \lambda$ are both vertical strips
        \item $|\xi / \mu|+|\xi / \lambda|=r$
        \item $\ell(\xi)\in\{n,\ell(\mu),\ell(\lambda)\}$
    \end{enumerate}
    and 
    \begin{align*}
        M(\lambda,\mu,n)=
        \begin{cases*}
            2 \qquad &\text{if $\ell(\mu)=n$ and $\ell(\lambda)<n$,}\\
            1 \qquad &\text{otherwise.}
        \end{cases*}
    \end{align*}
\end{proposition}

\begin{lem}\label{lem: rohs okada}
    For partitions $\lambda$ and $\mu$ with $\lambda_1,\mu_1\leq g$, we have
    \begin{equation*}
        |\kappa^{D}(\lambda,\mu,r;g)|+|\kappa^{D}(\flip(\lambda,g),\mu,r;g)| = M(\lambda^{t},\mu^{t},g)K^{D_g}_{\lambda^{t},\mu^{t}}(r).
    \end{equation*}
\end{lem}
\begin{proof}
    Let $A=\{(\mu,\xi,\lambda) : \text{ $(\mu^t,\xi^t,\lambda^t)$ satisfies conditions in  Proposition \ref{prop: type D dual pieri}} \}$. We will prove the case where $\lambda_1<g$ and $\mu_1<g$ as remaining cases can be handled similarly.

    We divide $A=A_1 \cup A_2$, where $A_1 = \{(\mu,\xi,\lambda) \in A :  \xi_1<g \}$ and $A_2 = \{(\mu,\xi,\lambda) \in A :  \xi_1 = g\}$. For each $(\mu,\xi,\lambda)\in A_2$ we consider an ohs (of the same length) $T=(\mu,\xi',\flip(\lambda,g))$ where $\xi'=(2g-\lambda_1,\xi_2,\xi_3,\dots)$. Then applying the map $\Gamma$ in Lemma \ref{lem: rohs ohs connection}, we see that $\Gamma(T)$ is an rohs of the same length as there is no paired $1$ and $\bar{1}$ in $\rind(T)$. Number of $1$'s in $\rind(\Gamma(T))$ equals $\left(2g-\lambda_1-\mu_1+\min(\xi_2-\mu_2,\xi_2-\lambda_2)\right)$, therefore we have 
    \begin{equation*}
        \mu_1+\left(2g-\lambda_1-\mu_1+\min(\xi_2-\mu_2,\xi_2-\lambda_2)\right)+\max(\mu_2,\lambda_2)=2g-\lambda_1+\xi_2\leq 2g.
    \end{equation*}
    We deduce $\Gamma(T)$ is a $g$-bounded rohs and clearly this process is invertible. Hence $|A_2|=|\kappa^{D}(\flip(\lambda,g),\mu,r;g)|$.

    Given $(\mu,\xi,\lambda)\in A_1$, we consider $T=\Gamma(\mu,\xi,\lambda)$ and $T$ is again an rohs of the same length. Also, number of $1$'s in $\rind(\Gamma(T))$ equals $\xi_1-\mu_1+ \min(\xi_2-\mu_2,\xi_2-\lambda_2)$, therefore we have 
    \begin{equation*}
        \mu_1+(\xi_1-\mu_1+ \min(\xi_2-\mu_2,\xi_2-\lambda_2))+\max(\mu_2,\lambda_2)=\xi_1+\xi_2 <2g.
    \end{equation*} 
    Again we deduce that $\Gamma(T)$ is a $g$-bounded rohs and $|A_1|=|\kappa^{D}(\lambda,\mu,r;g)|$. 
\end{proof}

By the Weyl character formula (see \cite[Equation (43) and (65)]{BG2006} for example), we have  
\begin{equation}\label{eq: weyl D}
    \ns_{\lambda}^{D_n}=\frac{\prod_{i=1}^{n}(x^2_i-1)}{\prod_{i=1}^{n}x_i}s^{C_n}_{\lambda-(1^{n})}.
\end{equation}
\begin{lem}\label{lem: dohs minus}
    For partitions $\lambda$ and $\mu$ with $\lambda_1,\mu_1\leq g$, we have
    \begin{equation*}
        |\kappa^{D}(\lambda,\mu,r;g)|-|\kappa^{D}(\flip(\lambda,g),\mu,r;g)| = |\kappa^{C}(\lambda,\mu,r;g-1)| - |\kappa^{C}(\lambda,\mu,r-2;g-1)|.
    \end{equation*}
\end{lem}
\begin{proof}
    If one of $\lambda_1$ and $\mu_1$ is $g$, then the right-hand side is trivially zero.
    We can verify that the left-hand side is also zero due to the symmetry established in Lemma \ref{lem: flip}. 
    Now we assume $\mu_1,\lambda_1 \leq g-1$. Let $A$ be a collection of an ohs $T$ in $\kappa^{C}(\lambda,\mu,r;g-1)$ such that $\bar{1}\in \rind(T)$ and $B$ be a collection of ah ohs $T=(\mu,\xi,\lambda)\in\kappa^{C}(\lambda,\mu,r-2;g-1)$ with $\xi_1=g-1$.
    Then there exists a natural bijection from $\kappa^{C}(\lambda,\mu,r-2;g-1)\setminus B$ to $\kappa^{C}(\lambda,\mu,r;g-1)\setminus A$ given by:
    $(\mu,\xi,\lambda) \rightarrow (\mu,\xi',\lambda)$ where $\xi'=(\xi_1+1,\xi_2,\xi_3,\dots)$.
    
    The map $\Gamma$ in Lemma \ref{lem: rohs ohs connection} sends $A$ bijectively to $\kappa^{D}(\lambda,\mu,r;g)$.
    Now given $(\mu,\xi,\lambda)\in B$, we associate $\Gamma(\mu,\xi',\flip(\lambda,g))$  where $\xi'=(2g-\lambda_1,\xi_2,\xi_3,\dots)$.
    This is a bijection from $B$ to $\kappa^{D}(\flip(\lambda,g),\mu,r;g)$.
\end{proof}
\begin{proof}[Proof of \eqref{eq: q=1 D} when $\lambda$ and $\mu$ are not spin weights]
  
    By Lemma \ref{lem: flip} we have 
    \begin{align*}
        P(\lambda,\mu)=\sum_{\nu}\left(|\kappa^D(\lambda,\nu,\mu_n;g)| + |\kappa^D(\flip(\lambda,g) ,\nu,\mu_n;g)|\right)P(\nu,\mu')\\
        N(\lambda,\mu)=\sum_{\nu}\left(|\kappa^D(\lambda,\nu,\mu_n;g)|-|\kappa^D (\flip(\lambda,g) ,\nu,\mu_n;g)|\right)N(\nu,\mu')
    \end{align*}
    where $\mu'=(\mu_1,\dots,\mu_{n-1})$.
    As in the proof of  \eqref{eq: q=1 B}, by Theorem \ref{thm: Cauchy type D}, Proposition \ref{prop: type D dual pieri} and Lemma \ref{lem: rohs okada}, we conclude $P(\hat{\lambda},\hat{\mu})=[m_\mu](\ps^{D_n}_{\lambda})$ by induction on $n$.
    
    By \eqref{eq: weyl D} and \eqref{eq: q=1 C}, we have  
    \begin{equation*}
        [m_\mu]\ns_{\lambda}^{D_n} = \sum_{S \subset [n-1]} (-1)^{n-|S|} |\SSOT_{g-1} (\hat{\lambda},(g-\mu_n,\hat{\mu'}-2\varepsilon_S))| - \sum_{S \subset [n-1]} (-1)^{n-|S|} |\SSOT_{g-1} (\hat{\lambda},(g-\mu_n-2,\hat{\mu'}-2\varepsilon_S))|.
    \end{equation*}
    where $\varepsilon_S$ is a 0-1 vector given by $\sum_{i\in S}{\varepsilon_i}$. 
    Then by Lemma \ref{lem: dohs minus} we deduce $N(\hat{\lambda},\hat{\mu})=[m_\mu] \ns_{\lambda}^{D_n}$. 
\end{proof}

\subsubsection{Proof of \eqref{eq: q=1 D} when $\lambda$ and $\mu$ are spin weights}
The proof parallels the arguments presented in Section \ref{subsub: D}. 
We provide the main ingredients and an outline of the proof, with details omitted.

For $\lambda=(\lambda_1,\lambda_2,\dots,\lambda_n)\in \Par_n$, we define
\begin{equation*}
    \pss^{D_n}_{\lambda}:=s^{D_n}_{\lambda^{\sharp}}+s^{D_n}_{\nu}, \qquad \nss^{D_n}_{\lambda}:=s^{D_n}_{\lambda^{\sharp}}-s^{D_n}_{\nu}
\end{equation*} 
where $\nu=(\lambda_1+\frac{1}{2},\dots,\lambda_{n-1}+\frac{1}{2},-\lambda_{n}-\frac{1}{2})$. 
Then $\pss^{D_n}_{\lambda}$ always has a factor $\prod_{i=1}^{n}(x_i^{\frac{1}{2}}+x_i^{-\frac{1}{2}})$ so we let 
\begin{equation*}
    \pss^{D_n}_{\lambda}= \prod_{i=1}^{n}(x_i^{\frac{1}{2}}+x_i^{-\frac{1}{2}}) \overline{\pss}^{D_n}_{\lambda}.
\end{equation*}
Additionally define
\begin{equation*}
        P^{\sharp}(\lambda,\mu)=|\SSROT_{g'}(\lambda,\mu)|+|\SSROT_{g'}(\flip(\lambda,g'),\mu)|, \quad   N^{\sharp}(\lambda,\mu)=|\SSROT_{g'}(\lambda,\mu)|-|\SSROT_{g'}(\flip(\lambda,g'),\mu)|
\end{equation*}
where $g'=g+\frac{1}{2}$. Then by Remark \ref{rmk: D spin} we can write a recursive formula for $P^{\sharp}(\lambda,\mu)$ and  $N^{\sharp}(\lambda,\mu)$ as before.

Lemma \ref{lem: spin cauchy} serves as the counterpart to Theorem \ref{thm: Cauchy type D}, and Lemma \ref{lem: spin okada} corresponds to Proposition \ref{prop: type D dual pieri}.
While these results are not explicitly stated in the literature, their proofs are identical in structure to those of the original versions.

\begin{lem}\label{lem: spin cauchy}
    For indeterminates $x_1,\dots x_n$ and $t_1,\dots,t_g$ we have
    \begin{equation*}
        \sum_{\lambda\subseteq (g^n)}(-1)^{|\tilde{\lambda}|}m_{\lambda}(x)\prod_{i=1}^{n}e^{(g)}_{g-\lambda_i}(t)=\sum_{\lambda\subseteq (g^n)}{(-1)^{|\tilde{\lambda}|}\overline{\pss}^{D_n}_{\lambda}(x) \overline{\pss}^{D_g}_{\tilde{\lambda}}(t)}.
    \end{equation*}
\end{lem}

\begin{lem}\label{lem: spin okada}
     For $r\leq n$ and $\lambda\in \Par_n$ we have
    \begin{equation*}
        e^{(n)}_r \overline{\pss}_{\mu}^{D_n}=\sum_{\lambda}|\kappa^{B}(\lambda^{t},\mu^{t},r;n)|\overline{\pss}_{\lambda}^{D_n}
    \end{equation*}
\end{lem}
The counterpart for \eqref{eq: weyl D} is given as \cite[Equation (63) and (65)]{BG2006}
\begin{equation*}
    \nss_{\lambda}^{D_n}=\prod_{i=1}^{n}(x_i^{\frac{1}{2}}-x_i^{-\frac{1}{2}})s^{B_n}_{\lambda}.
\end{equation*}

\begin{lem}\label{lem: spin okada connection}
     For partitions $\lambda$ and $\mu$ with $\lambda_1,\mu_1\leq g$, we have
     \begin{align*}
         |\kappa^{D}(\lambda,\mu,r;g+\frac{1}{2})|=|\kappa^{B}(\lambda,\mu,r;g)|-|\kappa^{B}(\lambda,\mu,r-1;g)| \qquad \text{if $|\lambda|-|\mu|-r$ is even},\\
         |\kappa^{D}(\flip(\lambda,g+\frac{1}{2}),\mu,r;g+\frac{1}{2})| =-|\kappa^{B}(\lambda,\mu,r;g)|+|\kappa^{B}(\lambda,\mu,r-1;g)| \qquad \text{if $|\lambda|-|\mu|-r$ is odd}.
     \end{align*}   
\end{lem}
\begin{proof}
    Proof is identical to that of Lemma \ref{lem: dohs minus}.
\end{proof}
Note that $|\kappa^{D}(\lambda,\mu,r;g+\frac{1}{2})|=0$ if $|\lambda|-|\mu|-r$ is odd and $|\kappa^{D}(\flip(\lambda,g+\frac{1}{2}),\mu,r;g+\frac{1}{2})|=0$ if $|\lambda|-|\mu|-r$ is even. Therefore Lemma \ref{lem: spin okada connection} can be rephrased as
\begin{align*}
        (-1)^{|\lambda|-|\mu|-r}(|\kappa^{D}(\lambda,\mu,r;g+\frac{1}{2})|+|\kappa^{D}(\flip(\lambda,g+\frac{1}{2}),\mu,r;g+\frac{1}{2})|) =|\kappa^{B}(\lambda,\mu,r;g)|-|\kappa^{B}(\lambda,\mu,r-1;g)|,\\
        |\kappa^{D}(\lambda,\mu,r;g+\frac{1}{2})|-|\kappa^{D}(\flip(\lambda,g+\frac{1}{2}),\mu,r;g+\frac{1}{2})|=|\kappa^{B}(\lambda,\mu,r;g)|-|\kappa^{B}(\lambda,\mu,r-1;g)|.
\end{align*}
From here, we can complete the proof for \eqref{eq: q=1 D} when $\lambda$ and $\mu$ are spin weights, in the same manner as in Section \ref{subsub: D}.

\section{Concluding remarks and Future directions}\label{sec: future}
We outline some further questions for exploration. 
\begin{itemize}
    \item It is natural to try to complete the missing parts of Table \ref{table: summary}. For partial results in those missing parts, see \cite{JK2021}.
    For Lusztig $q$-weight multiplicity for type $D$, we suspect that a formula related to the column KR crystal of the affine type $C_N^{(1)}$ (kind $\diamond=\sboxtwo$) might exist. 
    For non-spin weights in type $B$, we may need to step beyond the current framework of crystal theory (Remark \ref{rmk: why non spin is hard}), which is a big and fascinating challenge.
    
    \item In \cite{Lee2023}, the crystal structure on $\SSOT$ was given. We may further try to establish crystal structures on $\GSSOT$ or $\SSROT$.\\
    
    \item Theorem \ref{thm: C lusztig} naturally implies the monotonicity of the Lusztig $q$-weight multiplicities, namely, $\KL^{C_n}_{\lambda+(1^n),\mu+(1^n)}(q)-\KL^{C_n}_{\lambda,\mu}(q)$ is a polynomial in $q$ with nonnegative coefficients. 
    By Theorem \ref{thm: B lusztig}, monotonicity property holds for type $B$ with spin weights.
    We conjecture that the monotonicity properties hold for other types.
    Additionally it would be an interesting problem to find a geometric explanation (possibly related to the intersection homology theory) for the monotonicity.

    \item In \cite{LecouLenart2020}, they gave a combinatorial formula for $\KL^{C_n}_{\lambda,(0^n)}$ by giving a suitable statistic on King tableaux.
    Transferring their statistic on $\SSOT$ via the bijection given in \cite{Lee2023}, we observed that their statistic is the same as ours. 
    It is an interesting problem to unveil a connection between these two.  
    
    \item
    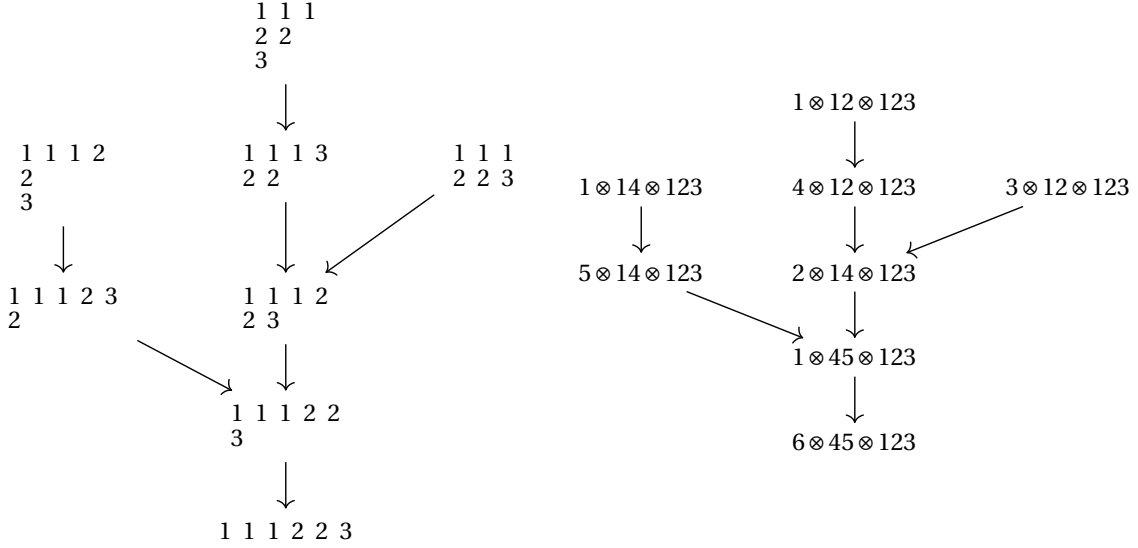
\begin{figure}
        \begin{tikzcd}
        	& \usebox\youngA \\
        	\usebox\youngB & \usebox\youngC & \usebox\youngD \\
        	\usebox\youngE & \usebox\youngF \\
        	& \usebox\youngG \\
        	& \usebox\youngH
        	\arrow[from=1-2, to=2-2]
        	\arrow[from=2-1, to=3-1]
        	\arrow[from=2-2, to=3-2]
        	\arrow[from=2-3, to=3-2]
        	\arrow[from=3-1, to=4-2]
        	\arrow[from=3-2, to=4-2]
        	\arrow[from=4-2, to=5-2]
        \end{tikzcd}
        \quad
        \begin{tikzcd}
        	&  1 \otimes 12 \otimes 123 \\
        	1 \otimes 14 \otimes 123 & 4 \otimes 12 \otimes 123 & 3 \otimes 12 \otimes 123 \\
        	5 \otimes 14 \otimes 123 & 2 \otimes 14 \otimes 123 \\
        	& 1\otimes 45 \otimes 123 \\
        	& 6\otimes 45 \otimes 123
        	\arrow[from=1-2, to=2-2]
        	\arrow[from=2-1, to=3-1]
        	\arrow[from=2-2, to=3-2]
        	\arrow[from=2-3, to=3-2]
        	\arrow[from=3-1, to=4-2]
        	\arrow[from=3-2, to=4-2]
        	\arrow[from=4-2, to=5-2]
        \end{tikzcd}
        \caption{A cyclage graph for weight $(3,2,1)$ and corresponding elements in $B^t_{(3,2,1)}(\emptyset)$.}
        \label{fig: cyclage}
    \end{figure}
    
    The left side of the Figure \ref{fig: cyclage} represents the cyclage graph \cite[Example 3.2]{Las1989} for the weight $(3,2,1)$. 
    The two tableaux connected by the directed edge are in a cyclage relation.
    On the other hand, each vertex on the right side represents the corresponding crystal element.
    The edge $T \to T'$ means that $\hw(e_0(T))=T'$.
    For example, let $T= 1 \otimes 12 \otimes 123 \in B_{(3,2,1)}^t(\emptyset)$.
    Then $e_0(T) = N \otimes 12 \otimes 123 $ and $\hw(e_0(T))=4 \otimes 12 \otimes 123$, which matches the above diagram.
    Therefore, the right side of the diagram can be extended to cyclage graphs for $\mathfrak{g}= B, C, D$, and there are natural embedding from a cyclage graph for a given weight to a cyclage graph for another weight. 
    
    \item The Catalan function (type $A$), introduced by Blasiak, Morse, Pun and Summers, is defined by applying a sequence of certain operators $R_{ij}$ to the Hall-Littlewood polynomial according to the associated root ideal \cite{BMPS19}. We may define the Catalan function for type $C$ in a similar manner: apply a sequence of operators $R_{ij}$ to $H_{\mu} := \sum_{\lambda \in P^{+}_n} \KL^{C_n}_{\lambda,\mu}(q) s^{C_n}_{\lambda}$ (according to the associated root ideal). By Theorem \ref{thm: C lusztig}, we can study the Catalan function for type $C$ in the realm of KR crystals. Natural questions, including positivity and geometric meaning, are under exploration by the third author and Shimozono.
\end{itemize}

\appendix
\section{Combinatorial $R$-matrix and the Splitting map}\label{Sec: append A}
We prove several results whose proofs were deferred to this section, although they were previously utilized in Section \ref{Sec: L}.
The key ingredient in these proofs is the description of the combinatorial $R$-matrix for $B^{k,1}(\diamond) \otimes B^{1,1}(\diamond)$, as presented in Appendix \ref{app:comb R B}. The reader might skip the rules in Appendix \ref{app:comb R B} at first, returning to them when they are used later on.

For the remainder of this section, we denote by $R$ the combinatorial $R$-matrix and by $S$ the splitting map $S^{t}_{\mu}: B^{t}_{\mu} \rightarrow (B^{1,1})^{\otimes |\mu|}$. 
Additionally, we restrict our focus to the cases where $\diamond = \sboxone$ or $\sboxeleven$.

\subsection{Combinatorial R matrix for $B^{k,1}(\diamond) \otimes B^{1,1}(\diamond)$}\label{app:comb R B}
We fix an integer $k \geq 2$ and consider an element $u \otimes u' \in B^{k,1}(\diamond) \otimes B^{1,1}(\diamond)$.
Below, we provide a complete description of $R(u \otimes u')$.
Recall that we typically represent $u$ as a word consisting of letters, but we also identify it with the set of letters contained in that word.
For instance, we identify the word $123$ with the set $\{1, 2, 3\}$. 
The rules outlined below can be established by induction, and we omit the details of the proof.

 (a) $\diamond=\sboxeleven$

\begin{itemize}
    \item (Case 1) $u' \prec u_1$  or $u=\emptyset$.

    Consider a word $v=u'u$ and let $w=\red(v)$.
    \begin{itemize}
        \item (Case 1-1)  $w$ is empty 

        $R(u\otimes u')=\overline{1}\otimes 1$.
        \item (Case 1-2) $w\neq v$, i.e., $v$ is not admissible 

        Let $d$ be the largest letter (in $\prec$ order) in $w$, then $R(u\otimes u')=d\otimes (w\setminus\{d\})$.

        \item (Case 1-3) $w= v$ and the length of $w$ equals $k+1$

        Let $d$ be the largest letter in $w$, then $R(u\otimes u')=d\otimes (w\setminus\{d\})$.

        \item (Case 1-4)  $w= v$ and the length of $w$ is smaller than $k+1$

        Find the smallest $i$ such that: $i, \overline{i} \notin w$ and $s=w\cup\{i,\overline{i}\}\setminus \{d\}$ is admissible where $d$ is the largest letter in $w\cup\{i,\overline{i}\}$. We have $R(u\otimes u')=d\otimes s$.
    \end{itemize}    
\end{itemize}

\begin{itemize}
    \item (Case 2) $c \nprec u_1$ where $u'=c$ is an unbarred letter

    Let $u_i$ be the largest letter (among letters in $u$) such that $u_i\preceq c$.

    \begin{itemize}
        \item (Case 2-1) $\overline{u_i}\notin u$

          $R(u\otimes u')=u_i\otimes (u\setminus\{u_i\}\cup \{c\})$.
        \item (Case 2-2) $\overline{u_i}\in u$

          Find the maximal $p\prec u_i$ such that $\overline{p}\notin u$.
          Then $R(u\otimes u')=p\otimes (u\setminus\{u_i,\overline{u_i}\}\cup \{c,\overline{p}\})$.
    \end{itemize}
\end{itemize}

\begin{itemize}
    \item (Case 3) $\overline{c} \nprec u_1$ where $u'=\overline{c}$ is a barred letter and $c\notin u$

    Find the maximal $u_i$ such that $u_i\preceq\overline{c}$.
    \begin{itemize}
        \item (Case 3-1) $u_i\prec c$ and $\overline{u_i}\in u$

        Pick the maximal $p\preceq u_i$ such that $\overline{p}\notin u$. 
        Then $R(u\otimes u')=p \otimes (u\setminus\{u_i,\overline{u_i}\}\cup \{\overline{c},\overline{p}\}) $.

        \item (Case 3-2) otherwise

        $R(u\otimes u')=u_i\otimes(u\setminus\{u_i\}\cup \{\overline{c}\})$.
    \end{itemize}
\end{itemize}

\begin{itemize}
    \item (Case 4)  $\overline{c} \nprec u_1$ where $u'=\overline{c}$ is a barred letter and $c\in u$

    Let $I=\{u_i\in u: c \preceq u_i \preceq \overline{c}\}$ and 
    \begin{equation*}
        I'=\red(\{d-c+1:d\in I\}\cup\{\overline{d-c+1}:\overline{d}\in I\}).
    \end{equation*}
    For example, if $c=6$ and $I=\{6,7,\overline{7}\}$ we have $I'=\red(\{1,2,\overline{2}\})=\{1\}$.
    
    \begin{itemize}
        \item (Case 4-1) $I'$ is empty

        Note that we must have $c,\overline{c}\in u$. 
        Find the maximal $p\preceq c$ such that $\overline{p}\notin u$. Then $R(u\otimes u')=p \otimes (u\setminus\{c\}\cup \{\overline{p}\}) $.

        \item (Case 4-2) $I'$ is not empty

        Let $d$ be the maximal letter in $I'$ and set $d'=d+c-1$ if $d$ is an unbarred letter and $d'=\overline{r+c-1}$ if $d=\overline{r}$ is a barred letter.
        If $d'=c$ then  $R(u\otimes u')=c\otimes(u\setminus\{c\}\cup \{\overline{c}\})$.
        If $d'\neq c$ then find the minimal $p>c$ such that $p,\overline{p}\notin I\setminus\{d'\}$. 
        We have $R(u\otimes u')=d'\otimes(u\setminus\{c,d'\}\cup \{p,\overline{p}\})$
    \end{itemize}
\end{itemize}

(b) $\diamond=\sboxone$
\begin{itemize}
    \item (Case 1) $u' \prec u_1$ or $u=\emptyset$, or $u'=\emptyset$.

    Consider a word $v=u'u$ (if $u'=\emptyset$ is empty then $v=u$). Let $w=\red(v)$.
    \begin{itemize}
        \item (Case 1-1) $w=v$ and length of $w$ is smaller than $k+1$

        $R(u\otimes u')=\emptyset\otimes w$.
        \item (Case 1-2) otherwise

        $R(u\otimes u')=d\otimes (v\setminus\{d\})$ where $d$ is the largest letter in $v$.
    \end{itemize}    
\end{itemize}

Case 2, 3 and 4 follow the same rules given when $\diamond=\sboxeleven$.

\begin{example}
    We provide examples for each case. Throughout this example, we consider $u\otimes u'\in B^{7,1}\otimes B^{1,1}$.
    
    $\diamond=\sboxeleven$
    \begin{itemize}
        \item (Case 1-1) $u\otimes u'=\overline{1}\otimes 1$ and $R(u\otimes u')=\overline{1}\otimes 1$
        \item (Case 1-2) $u\otimes u'=24\overline{2}\otimes 1$ and $R(u\otimes u')=4\otimes 1$
        \item (Case 1-3) $u\otimes u'=2345678\otimes 1$ and $R(u\otimes u')=8\otimes 1234567$
        \item (Case 1-4) $u\otimes u'=23456\otimes 1$ and $R(u\otimes u')=\overline{7}\otimes 1234567$   
    \end{itemize}
    
    \begin{itemize}
        \item (Case 2-1) $u\otimes u'=245\otimes 3$ and $R(u\otimes u')=2\otimes 345$
        \item (Case 2-2) $u\otimes u'=24\overline{2}\otimes 3$ and $R(u\otimes u')=1\otimes 34\overline{1}$
        \item (Case 3-1) $u\otimes u'=13\overline{3}\otimes \overline{4}$ and $R(u\otimes u')=2\otimes 1\bar{4}\bar{2}$
        \item (Case 3-2) $u\otimes u'=13\overline{2}\otimes \overline{4}$ and $R(u\otimes u')=3\otimes 1\bar{4}\bar{2}$   
    \end{itemize}
    
    \begin{itemize}
        \item (Case 4-1) $u\otimes u'=1 3\overline{3}\otimes \overline{3}$ and $R(u\otimes u')=2\otimes 1\bar{3}\bar{2}$
        \item (Case 4-2) $u\otimes u'=1 23  \otimes \overline{3}$ and $R(u\otimes u')=3\otimes 12\overline{3}$, $u\otimes u'=1 3  5 \otimes \overline{3}$ and $R(u\otimes u')=5\otimes 14\overline{4}$.
    \end{itemize}
    
    $\diamond=\sboxone$
    \begin{itemize}
        \item (Case 1-1) $u\otimes u'=234\otimes 1$ and $R(u\otimes u')=\emptyset\otimes 1234$
        \item (Case 1-2) $u\otimes u'=2345678\otimes 1$ and $R(u\otimes u')=8\otimes 123467$, $u\otimes u'=2 \overline{2}\otimes 1$ and $R(u\otimes u')=\overline{2}\otimes 12$.
    \end{itemize}
\end{example}

\subsection{Description of the splitting map}
We explicitly describe image of the splitting map $S(u \otimes u')$ for $u \otimes u' \in \HW(B^{a,1} \otimes B^{b,1})$. In particular, we show that the splitting map commutes with the combinatorial $R$-matrix, i.e., $S(u \otimes u') = S(R(u \otimes u'))$ (Corollary \ref{cor: splitting invariance}).

We adopt the following notation from \cite[Appendix D]{Sch2005}. Given a word $w_1 \dots w_2 w_3 \dots w_4\dots$, where $w_1 \dots w_2$ denotes consecutive letters from $w_1$ to $w_2$, we represent the word as $[w_1 \dots w_2 | w_3 \dots w_4 | \dots]$. For instance, the word $123 \bar{4} \bar{3} \bar{2}$ is written as $[1 \dots 3 | \overline{4} \dots \overline{2}]$. Note that an element $u \otimes u' \in \HW(B^{a,1} \otimes B^{b,1})$ is of the form 
$ [1 \dots k | n+1 \dots p | \overline{p} \dots \overline{q}] \otimes [1 \dots n], $
where the parameter ranges are $p \geq n\geq k \geq 0$ and $q \leq p + 1$. Any of the intervals in $u$ or $u'$ can be empty. For example, if $p = n$, then $u = [1 \dots k | \overline{n} \dots \overline{q}]$.

Next, we introduce some statistics for $u \otimes u'$. Let $\tol(u \otimes u')$ denote the maximal nonnegative integer, no greater than $n - k$, such that the word 
$$ [1 \dots k + \tol(u\otimes u') | n + 1 \dots p | \overline{p} \dots \overline{q}] $$
is admissible. We also define $\midd(u \otimes u') = \min(\tol(u \otimes u'), a - \ell(u))$, where $\ell(u) = k + (p - n) + (p + 1 - q)$ is the length of $u$ as a word. Lastly, the notation
\[
\underbrace{\overline{1} \otimes 1 \otimes \dots \otimes \overline{1} \otimes 1}_{2m}
\]
denotes the element in $(B^{1,1})^{\otimes 2m}$ formed by aligning $m$ pairs of $\overline{1} \otimes 1$.

\begin{proposition}\label{prop: split B}
   For $u \otimes u'\in\HW(B^{a,1}(\hspace{0.4mm}\sboxeleven\hspace{0.4mm})\otimes B^{b,1}(\hspace{0.4mm}\sboxeleven\hspace{0.4mm}))$, denoted by $[1 \dots k | n+1 \dots p | \overline{p} \dots \overline{q}] \otimes [1 \dots n]$, we have the following.
    \begin{itemize}
        \item \textbf{(Case 1)} If the number of barred letters in $u$, i.e., $p-q+1$, is an odd number and $\tol(u\otimes u')=2 \lfloor\frac{\midd(u\otimes u')}{2}\rfloor$, we have 
        \begin{equation*}
            S(u\otimes u')=k\otimes \dots \otimes1\otimes \underbrace{\overline{1} \otimes 1\otimes  \dots \otimes\overline{1}\otimes 1}_{2 \lfloor\frac{\midd(u\otimes u')}{2}\rfloor} \otimes \overline{q} \otimes \dots\otimes\overline{(p-1)} \otimes (p-1)\otimes \dots \otimes 1  \otimes \underbrace{\overline{1} \otimes 1\otimes  \dots \otimes\overline{1}\otimes 1}.
        \end{equation*}
        \item \textbf{(Case 2)} If the number of barred letters in $u$ is an even number and for $v=u\cup\{p+1,\overline{p+1}\}$, $v$ is admissible and $v\otimes u'$  corresponds \textbf{(Case 1)}, we have
        \begin{equation*}
            S(u\otimes u')=k\otimes \dots \otimes1\otimes \underbrace{\overline{1} \otimes 1\otimes  \dots \otimes\overline{1}\otimes 1}_{2 \lfloor\frac{\midd(u\otimes u')}{2}\rfloor} \otimes \overline{q} \otimes \dots\otimes\overline{(p+1)} \otimes (p+1)\otimes \dots \otimes 1  \otimes \underbrace{\overline{1} \otimes 1\otimes  \dots \otimes\overline{1}\otimes 1}.
        \end{equation*}
        
        \item \textbf{(Case 3)} Otherwise, we have
       \begin{equation*}
            S(u\otimes u')=k\otimes \dots \otimes1\otimes \underbrace{\overline{1} \otimes 1\otimes  \dots \otimes\overline{1}\otimes 1}_{2 \lfloor\frac{\midd(u\otimes u')}{2}\rfloor} \otimes \overline{q} \otimes \dots\otimes\overline{p} \otimes p\otimes \dots \otimes 1  \otimes \underbrace{\overline{1} \otimes 1\otimes  \dots \otimes\overline{1}\otimes 1}.
        \end{equation*}
    \end{itemize}
    In all cases the rightmost factors $\underbrace{\overline{1} \otimes 1\otimes  \dots \otimes\overline{1}\otimes 1}$ of $S(u\otimes u')$ are filled so that $S(u\otimes u')\in (B^{1,1})^{\otimes (a+b)}$. 
\end{proposition}

\begin{proof} 
We denote $\ell = \lfloor \frac{\midd(u \otimes u')}{2} \rfloor$. First, we split the right factor $u' = [1 \dots n]$ and start with the expression  
\begin{equation*}
    u\otimes S(u')=u \otimes n \otimes \dots \otimes 1 \otimes \underbrace{\overline{1} \otimes 1 \otimes \dots \otimes \overline{1} \otimes 1}_{b-n}.
\end{equation*}  
Next, we track the transformations that occur when moving the leftmost factor $u$ to the right by applying combinatorial $R$-matrices. Initially, we apply the combinatorial $R$-matrices $k$ times using the rule (Case 2-1) in Appendix~\ref{app:comb R B}. After $k$ applications, the front becomes $k \otimes \dots \otimes 1$, leaving the remaining term  
\begin{equation*}
    [n-k+1 \dots p | \overline{p} \dots \overline{q}] \otimes n-k \otimes \dots \otimes 1 \otimes \underbrace{\overline{1} \otimes 1 \otimes \dots \otimes \overline{1} \otimes 1}_{b-n}.
\end{equation*}  

If $\ell > 0$, then by repeatedly applying the rules (Case 1-4) and (Case 2-1) in Appendix~\ref{app:comb R B}, we derive  
\begin{align*}
    R([n-k+1 \dots p | \overline{p} \dots \overline{q}] \otimes n-k) &= \overline{1} \otimes [1 | n-k \dots p | \overline{p} \dots \overline{q}],\\  
    R([1 | n-k \dots p | \overline{p} \dots \overline{q}] \otimes n-k-1) &= 1 \otimes [n-k-1 \dots p | \overline{p} \dots \overline{q}].
\end{align*}  
Repeating this process $\ell$ times, we produce $\underbrace{\overline{1} \otimes 1 \otimes \dots \otimes \overline{1} \otimes 1}_{2\ell}$ at the front, with the remaining term  
\begin{equation}\label{eq: splitting B front}
    [n-k-2\ell+1 \dots p | \overline{p} \dots \overline{q}] \otimes n-k-2\ell \otimes \dots \otimes 1 \otimes \underbrace{\overline{1} \otimes 1 \otimes \dots \otimes \overline{1} \otimes 1}_{b-n}.
\end{equation}
From this point onward, we proceed separately for each case as outlined in the statement.

\textbf{(Case 1)} The word $[n-k-2\ell \dots p | \overline{p} \dots \overline{q}]$ is not admissible and maps to $[n-k-2\ell \dots p-1 | \overline{p-1} \dots \overline{q}]$ via the map $\red$. When $p>q$, applying the rule (Case 1-2) in Appendix~\ref{app:comb R B}, we obtain  
\begin{equation*}
    R([n-k-2\ell+1 \dots p | \overline{p} \dots \overline{q}] \otimes n-k-2\ell) = \overline{q} \otimes [n-k-2\ell \dots p-1 | \overline{p-1} \dots \overline{q+1}].
\end{equation*}  
 Then using the rule (Case 1-4) in Appendix~\ref{app:comb R B}, we have  
\begin{equation*}
    R([n-k-2\ell \dots p-1 | \overline{p-1} \dots \overline{q+1}] \otimes n-k-2\ell-1) = \overline{q+1} \otimes [n-k-2\ell-1 \dots p | \overline{p} \dots \overline{q+2}].
\end{equation*}  
Since $p-q+1=2m+1$ is an odd number, repeating this process $2m+1$ times results in $\overline{q} \otimes \dots \otimes \overline{p-1}\otimes p-1$ at the front, leaving the remaining term  
\begin{equation}\label{eq: dddd}
    [s+1 \dots p-2] \otimes s \otimes \dots \otimes 1 \otimes \underbrace{\overline{1} \otimes 1 \otimes \dots \otimes \overline{1} \otimes 1}_{b-n}
\end{equation}
where $s=n-k-2\ell-2m-1$. From this point, it is straightforward to confirm that \eqref{eq: dddd} splits into $p-2 \otimes \dots \otimes 1 \otimes \overline{1} \otimes 1 \otimes \dots \otimes \overline{1} \otimes 1$, as required.  

\textbf{(Case 2)} Using the rule (Case 1-4) in Appendix~\ref{app:comb R B}, we find  
\begin{equation*}
    R([n-k-2\ell+1 \dots p | \overline{p} \dots \overline{q}] \otimes n-k-2\ell) = \overline{q} \otimes [n-k-2\ell \dots p+1 | \overline{p+1} \dots \overline{q+1}],
\end{equation*}  
and applying the rule (Case 1-1) in Appendix~\ref{app:comb R B} we have  
\begin{equation*}
    R([n-k-2\ell \dots p+1 | \overline{p+1} \dots \overline{q+1}] \otimes n-k-2\ell-1) = \overline{q+1} \otimes [n-k-2\ell-2 \dots p | \overline{p} \dots \overline{q+2}].
\end{equation*}  
Since $p-q+1=2m$ is an even number, repeating this process $2m+1$ times produces $\overline{q} \otimes \dots \otimes \overline{p+1}$ at the front, leaving the remaining term  
\begin{equation}\label{eq: ddddd}
    [s+1 \dots p+1] \otimes s \otimes \dots \otimes 1 \otimes \underbrace{\overline{1} \otimes 1 \otimes \dots \otimes \overline{1} \otimes 1}_{b-n}
\end{equation}
where $s=n-k-2\ell-2m-1$.
The remainder of the proof proceeds analogously.  

\textbf{(Case 3)} Applying combinatorial $R$-matrices $p-q+1$ times to \eqref{eq: splitting B front} results in $\overline{q} \otimes \dots \otimes \overline{p}$ at the front, with the remaining term  
\begin{equation}\label{eq: ddddd}
    [s+1 \dots p] \otimes s \otimes \dots \otimes 1 \otimes \underbrace{\overline{1} \otimes 1 \otimes \dots \otimes \overline{1} \otimes 1}_{b-n}.
\end{equation} 
where $s=n-k-2\ell-(p-q+1)$. The remainder of the proof follows similarly.
\end{proof}

\begin{proposition}\label{prop: split D}
   For $u \otimes u'\in\HW(B^{a,1}(\sboxone)\otimes B^{b,1}(\sboxone))$, denoted by $[1 \dots k | n+1 \dots p | \overline{p} \dots \overline{q}] \otimes [1 \dots n]$, we have 
   \begin{equation*}
            S(u\otimes u')=k\otimes \dots \otimes1\otimes \underbrace{\emptyset\otimes \dots \otimes \emptyset}_{ \midd(u\otimes u')} \otimes \overline{q} \otimes \dots\otimes\overline{p} \otimes p\otimes \dots \otimes 1  \otimes \underbrace{\emptyset\otimes \dots \otimes \emptyset}_{}
        \end{equation*}
    where the rightmost factors $\underbrace{\emptyset\otimes \dots \otimes \emptyset}_{}$ of $S(u\otimes u')$ are filled so that $S(u\otimes u')\in (B^{1,1})^{\otimes (a+b)}$.      
\end{proposition}
\begin{proof}
    The proof is similar to the proof for \textbf{(Case 3)} in Proposition \ref{prop: split B}.
\end{proof}

\begin{example}
We consider $u\otimes u'\in \HW(B^{4,1}(\hspace{0.4mm}\sboxeleven\hspace{0.4mm})\otimes B^{3,1}(\hspace{0.4mm}\sboxeleven\hspace{0.4mm}))$.
\begin{itemize}
    \item \textbf{(Case 1)} Let $u\otimes u'=4\overline{4}\otimes 123$. Then the word $124\overline{4}$ is admissible while $1234\overline{4}$ is not. We have $\tol(u\otimes u')=2$ and $\midd(u\otimes u')=\max(2, 4-\ell(u))=2$. To compute $S(u\otimes u')$ we start from 
    \begin{equation*}
        4\overline{4}\otimes 3\otimes 2\otimes 1\in B^{4,1}\otimes (B^{1,1})^{\otimes 3}
    \end{equation*}
    and apply combinatorial $R$-matrices using the rules in Appendix \ref{app:comb R B}:
    \begin{align*}
        4\overline{4}\otimes 3\otimes 2\otimes 1 \rightarrow \overline{1} \otimes 134\overline{4}\otimes 2\otimes 1 \rightarrow \overline{1} \otimes 1\otimes 23\otimes 1 \rightarrow \overline{1} \otimes 1\otimes 3\otimes 12.
    \end{align*}
    Lastly splitting the rightmost factor $12\in B^{4,1}$ we obtain 
    \begin{equation*}
        \overline{1} \otimes 1\otimes 3\otimes 2\otimes 1\otimes \overline{1} \otimes 1.
    \end{equation*}

    \item \textbf{(Case 2)} Let $u\otimes u'=\emptyset\otimes 123$. The computation of $S(u\otimes u')$  is given as:
     \begin{align*}
        \emptyset\otimes 3\otimes 2\otimes 1 \rightarrow \overline{1} \otimes 13 \otimes 2\otimes 1 \rightarrow \overline{1} \otimes 1\otimes 23\otimes 1 \rightarrow \overline{1} \otimes 1\otimes \overline{4}\otimes 1234.
    \end{align*}
    Lastly splitting the rightmost factor $1234\in B^{4,1}$ we obtain 
    \begin{equation*}
        \overline{1} \otimes 1\otimes \overline{4}\otimes 4\otimes3\otimes 2\otimes 1.
    \end{equation*}
\end{itemize}
\end{example}

\begin{corollary}\label{cor: splitting invariance}
    We have $S(R(u\otimes u'))=S(u\otimes u')$ for $u\otimes u'\in B^{a,1}(\diamond)\otimes B^{b,1}(\diamond)$ when $\diamond=\sboxone$ or $\sboxeleven$.
\end{corollary}
\begin{proof}
It suffices to verify when \( u \otimes u' \) is a classical highest weight element, since the splitting map \( S \) is a classical crystal embedding.  
For \( \diamond = \sboxeleven \), the description of \( R(u \otimes u') \) for \( u \otimes u' \in \HW(B^{a,1}(\diamond) \otimes B^{b,1}(\diamond)) \) is  fully provided in \cite[Appendix D]{Sch2005}.  
By conducting a thorough verification using Proposition \ref{prop: split B}, we conclude that \( S(u \otimes u') = S(R(u \otimes u')) \).

When \( \diamond = \sboxone \), we describe \( R(u \otimes u') \) for \( u \otimes u' = [1 \dots k | n+1 \dots p | \overline{p} \dots \overline{q}] \otimes [1 \dots n] \in \HW(B^{a,1}(\diamond) \otimes B^{b,1}(\diamond)) \).  
Without loss of generality, we assume \( a \geq b \).

Case 1: If \( b+1 \notin u \), then \( R(u \otimes u') = u \otimes u' \).

Case 2: If \( b+1 \in u \), then   \( R(u \otimes u') = [1 \dots k | n+m+1 \dots p | \overline{p} \dots \overline{q}] \otimes [1 \dots n+m] \) where \( m = \min(a-b, p-b) \).

The description for $R(u\otimes u')$ can be easily proved by induction and it is straightforward to verify that \( \midd(u \otimes u') = \midd(R(u \otimes u')) \). Therefore, Proposition \ref{prop: split D} completes the proof.
\end{proof}

\subsection{Invariance of $\SSOT_{g}$ under combinatorial $R$-matrix}\label{appendix: invariance ssot g}
\begin{proposition}\label{prop: invariance ssot}
    Let $\lambda$ be a partition and $\alpha$ a nonnegative integer vector. 
    Let $\beta$ be a rearrangement of $\alpha$, and let $g$ be a positive integer such that $g \geq \max(\alpha_1, \dots, \alpha_n)$.
    Then we have
    \begin{align}\label{eq: ssot invariance}
        \sum_{T \in \SSOT_g(\lambda, \alpha)} q^{\overline{D}(\phi_c(T))} &= \sum_{T \in \SSOT_g(\lambda, \beta)} q^{\overline{D}(\phi_c(T))}, \\
        \sum_{T \in \GSSOT_{g + \frac{1}{2}}(\lambda, \alpha)} \energy_{q,t}(\overline{D}(\phi_c(T)) &= \sum_{T \in \GSSOT_{g + \frac{1}{2}}(\lambda, \beta)} \energy_{q,t}(\overline{D}(\phi_c(T)). \label{eq: gssot invariance}
    \end{align}
\end{proposition}

\begin{proof}
We prove \eqref{eq: ssot invariance}, as \eqref{eq: gssot invariance} follows in a similar manner. 
It suffices to show the case when $\beta$ is obtained from $\alpha$ by switching $\alpha_i$ and $\alpha_{i+1}$.

Let $T \in \SSOT_g(\lambda, \alpha)$, and let $\phi_c(T) = x_n \otimes \dots \otimes x_1$. Each $x_j$ does not contain a letter $m$ or $\overline{m}$ for $m > g$.
Since we also have $g \geq \max(\alpha_i, \alpha_{i+1})$, by following the rules given in Appendix \ref{app:comb R B}, we conclude that any tensor factor of $S(x_{i+1} \otimes x_i)$ do not contain a letter $m$ or $\overline{m}$ for $m > g$.
Now, let $y_{i+1} \otimes y_i = R(x_{i+1} \otimes x_i)$. 
If either $y_i$ or $y_{i+1}$ contains a letter $m$ or $\overline{m}$ for $m > g$, then $m$ or $\overline{m}$ must appear in a tensor factor of $S(y_{i+1} \otimes y_i)$.
This contradicts $S(x_{i+1} \otimes x_i) = S(y_{i+1} \otimes y_i)$, as shown by Corollary \ref{cor: splitting invariance}. 
Therefore, there exists a $T' \in \SSOT_g(\lambda, \beta)$ such that
\[
    \phi_c(T) = x_n \otimes \dots \otimes x_{i+2} \otimes y_{i+1} \otimes y_i \otimes x_{i-1} \otimes \dots \otimes x_1.
\]
This provides an energy-preserving bijection between $\SSOT_g(\lambda, \alpha)$ and $\SSOT_g(\lambda, \beta)$.
\end{proof}

\begin{rmk}
    The assumption $g\geq \max(\alpha_1,\dots,\alpha_n)$ is crucial in Proposition \ref{prop: invariance ssot}. 
    Let $\lambda=(2,0)$, $\alpha=(3,1)$, $\beta=(1,3)$ and $g=1$.
    Then $\SSOT_g(\lambda,\alpha)=\{T\}$ and $\SSOT_g(\lambda,\beta)=\{T'\}$ where 
    \begin{equation*}
        T=\big((\emptyset,\ydiagram{2},\ydiagram{1}),(\ydiagram{1},\ydiagram{2},\ydiagram{2} )\big), \qquad T'=\big((\emptyset,\ydiagram{1},\ydiagram{1}),(\ydiagram{1},\ydiagram{2,1},\ydiagram{2} )\big).
    \end{equation*}
    However, $\overline{D}(\phi_c(T))=2$  and $\overline{D}(\phi_c(T'))=1$
\end{rmk}

\subsection{Energy is independent of $g$}\label{appendix: g independnt}
In this section, we prove \eqref{eq: energy g independent},  which ensures that the right-hand side of Theorem \ref{thm: C lusztig} is well-defined. 
Throughout this section, $B^{t}_{\mu} = B^{t}_{\mu}(\hspace{0.4mm}\sboxeleven\hspace{0.4mm})$, unless stated otherwise.
We begin by defining some key terminology. 
Let $v = v_1 \dots v_r \in B^{r', 1}$ (with the assumption that $0 \notin v$).
We define $\iota(v) = 1 \, u_1 \dots u_r \in B^{r' + 1, 1}$, where $u_i = c + 1$ if $v_i = c$ is an unbarred letter, and $u_i = \overline{c + 1}$ if $v_i = \overline{c}$ is a barred letter.
For example, if $v = 24\bar{4}\bar{3}$, then $\iota(v) = 135\bar{5}\bar{4}$. It is easy to check that if $x_n\otimes\dots\otimes x_1\in \HW(B^{t}_{\mu})$, then $\iota(x_n)\otimes \dots \otimes\iota(x_1)$ is also a classical highest weight element. The goal of this section is to prove the following statement, which is precisely \eqref{eq: energy g independent}.
\begin{proposition}\label{prop: energy g}
    Let $x \in \HW(B_{\mu}^{t})$, where $x = x_n \otimes \dots \otimes x_1$. Then we have $\overline{D}(x) = \overline{D}(\tilde{x})$, where $\tilde{x} = \iota(x_n) \otimes \dots \otimes \iota(x_1)$.
\end{proposition}

To prove Proposition \ref{prop: energy g}, we need an auxiliary lemma (Lemma \ref{lem: energy independent}). Before stating the lemma, we establish some useful notation.

\begin{definition}
    For nonnegative integers $a$ and $b$, we define $\inc(a,b)$ as the set of elements $x \in (B^{1,1})^{\otimes a+2b}$ of the form
    \begin{equation*}
        x = v_a \otimes v_{a-1} \otimes \dots \otimes v_1 \otimes \underbrace{\overline{1} \otimes 1 \otimes \dots \otimes \overline{1} \otimes 1}_{2b}
    \end{equation*}
    where $v_a \succ v_{a-1} \succ \dots \succ v_1 = 1$ if $a > 0$, and $v_2 \otimes v_1 \neq \overline{1} \otimes 1$ if $a = 2$. We also define $\overline{\inc}(a,b)$ as the set of elements $x$ satisfying the same conditions as in $\inc(a,b)$, except for the condition that $v_1 = 1$. It is evident that $\inc(a,b) \subset \overline{\inc}(a,b)$.
\end{definition}

We now list some useful properties that will be applied throughout the rest of the paper:
\begin{itemize}
    \item If $x \in \inc(a,b)$, and $x'$ is obtained from $x$ by applying a sequence of classical crystal operators, then it follows that $x' \in \overline{\inc}(a,b)$.
    \item For $x \otimes y$ where $x \in \inc(a_1,b_1)$ and $y \in \inc(a_2,b_2)$, with the condition that $a_2 > 0$, we can compute $\overline{D}(x \otimes y)$ exactly, as the local energy functions are fully determined:
    \begin{equation}\label{eq: inc energy}
    \overline{D}(x \otimes y) = \sum_{i=1}^{a_1-1} i + \sum_{i=1}^{b_1} 2(a_1 + (2i - 1)) + \sum_{i=1}^{a_2-1} (a_1 + 2b_1 + i) + \sum_{i=1}^{b_2} 2(a_1 + 2b_1 + a_2 + 2(i-1)).
    \end{equation}
\end{itemize}

\begin{lem}\label{lem: energy independent}
    For $u \otimes v \in B^{r_2,1}\otimes B^{r_1,1}$, denote $R(u\otimes v)=u'\otimes v'$. We have 
    \begin{enumerate}
        \item $R(\iota(u)\otimes \iota(v))=\iota(u')\otimes \iota(v')$, 
        \item $\overline{D}(u\otimes v)=\overline{D}(\iota(u)\otimes \iota(v))$.
    \end{enumerate}
\end{lem}
\begin{proof}
    Let $\hat{u} \otimes \hat{v} = \hw(u \otimes v)$, which is given by
    $
        \hat{u} \otimes \hat{v} = e_{i_n} \dots e_{i_1} (u \otimes v).$
    Then, we have
    $
        \iota(\hat{u}) \otimes \iota(\hat{v}) = e_{i_n + 1} \dots e_{i_1 + 1} (\iota(u) \otimes \iota(v)),
    $
    and additionally, $\iota(\hat{u}) \otimes \iota(\hat{v}) = \hw(\iota(u) \otimes \iota(v))$. Thus, we can assume that $u \otimes v$ is a classical highest weight element and proceed with the proof.

    Our objective is to describe $S(\iota(u) \otimes \iota(v))$ in terms of $S(u \otimes v)$. To do this, we introduce the following notation. Let $x \in \inc(a, b)$ be given by
    \begin{equation*}
        x = v_a \otimes \dots \otimes v_1 \otimes \underbrace{\overline{1} \otimes 1 \otimes \dots \otimes \overline{1} \otimes 1}_{2b}.
    \end{equation*}
    We define $P(x) \in \inc(a + 1, b)$ as
    \begin{equation*}
        P(x) = v'_a \otimes \dots \otimes v'_1 \otimes 1 \otimes \underbrace{\overline{1} \otimes 1 \otimes \dots \otimes \overline{1} \otimes 1}_{2b},
    \end{equation*}
    where $v'_i = c + 1$ if $v_i = c$ is an unbarred letter, and $v'_i = \overline{c + 1}$ if $v_i = \overline{c}$ is a barred letter. 

    By Proposition \ref{prop: split B}, $S(u \otimes v)$ has one of the following forms: 
    \begin{enumerate}
        \item $x \otimes y$, where $x \in \inc(a_1, b_1)$ and $y \in \inc(a_2, b_2)$ for some $a_1 \geq 0$ and $a_2 > 0$,
        \item $x$, where $x \in \inc(0, b_1)$.
    \end{enumerate}

    For the first case, by Proposition \ref{prop: split B}, we have $S(\iota(u) \otimes \iota(v)) = P(x) \otimes P(y)$.
    For the second case, we have
    \begin{equation*}
        S(\iota(u) \otimes \iota(v)) = 1 \otimes 1 \otimes \underbrace{\overline{1} \otimes 1 \dots \otimes \overline{1} \otimes 1}_{2b_1}.
    \end{equation*}

    In either case, a direct computation of the energy function using \eqref{eq: energy g independent} shows that
    \begin{equation*}
        \overline{D}(S(\iota(u) \otimes \iota(v))) - \overline{D}(S(u \otimes v)) = 2(r_1 + r_2),
    \end{equation*}
    which implies that $\overline{D}(\iota(u) \otimes \iota(v)) = \overline{D}(u \otimes v)$. 

    Furthermore, we have shown that $S(\iota(u) \otimes \iota(v))$ can be uniquely constructed from the information of $S(u \otimes v)$. Thus, we conclude that
    \begin{equation*}
        S(\iota(u) \otimes \iota(v)) = S(\iota(u') \otimes \iota(v')),
    \end{equation*}
    since $S(u \otimes v) = S(u' \otimes v')$ by Corollary \ref{cor: splitting invariance}. Again by Corollary \ref{cor: splitting invariance}, we conclude  $R(\iota(u) \otimes \iota(v)) = \iota(u') \otimes \iota(v')$.
\end{proof}

\begin{example}
    For $14\otimes 123 \in \HW(B^{4,1}\otimes B^{3,1})$ we have
    \begin{equation*}
        S(14\otimes 123)=\underbrace{1 \otimes \overline{1}\otimes 1}_{x} \otimes\underbrace{ 4\otimes 3\otimes 2 \otimes 1}_{y}
    \end{equation*}
    which is of the form $x\otimes y$ where $x\in \inc(1,1)$ and $y\in \inc(4,0)$. For $\iota(14)\otimes \iota(123)=125 \otimes 1234 \in \HW(B^{5,1}\otimes B^{4,1})$ we have 
    \begin{equation*}
        S(125\otimes 1234)=\underbrace{2\otimes 1 \otimes \overline{1}\otimes 1}_{P(x)} \otimes\underbrace{5\otimes 4\otimes 3\otimes 2 \otimes 1}_{P(y)}
    \end{equation*}
    where $P$ is the map defined in the proof of Lemma \ref{lem: energy independent}.
\end{example}

\begin{proof}[Proof of Proposition \ref{prop: energy g}]
    For $i \leq j$, let $x_j^{(i)}$ denote the element of the $i$-th tensor factor obtained by applying the composition of combinatorial $R$-matrices that places the $j$-th tensor factor in the $i$-th position in $x$. Similarly, we define $\tilde{x}_j^{(i)}$ for the tensor $\tilde{x} = \tilde{x}_n \otimes \dots \otimes \tilde{x}_1$ where each $\tilde{x}_i=\iota(x_i)$. According to Lemma \ref{lem: energy independent} (1), we have $\iota(x_j^{(i)}) = \tilde{x}_j^{(i)}$.

    Next, by Lemma \ref{lem: energy independent} (2), we know that
$
    \overline{D}(x_j^{(i+1)} \otimes x_i) = \overline{D}(\tilde{x}_j^{(i+1)} \otimes \tilde{x}_i),
    $
    which can be rewritten as
    \[
    \overline{D}(x_j^{(i)}) + \overline{H}(x_j^{(i+1)} \otimes x_i) + \overline{D}(x_i) = \overline{D}(\overline{x}_j^{(i)}) + \overline{H}(\tilde{x}_j^{(i+1)} \otimes \tilde{x}_i) + \overline{D}(\tilde{x}_i).
    \]
    Note that for any $v \in B^{r,1}$ and $\iota(v) \in B^{r+1,1}$, we have $\overline{D}(v) = \overline{D}(\iota(v))$. Therefore, we conclude that
    \[
    \overline{H}(x_j^{(i+1)} \otimes x_i) = \overline{H}(\tilde{x}_j^{(i+1)} \otimes \tilde{x}_i).
    \]
    Using the definition of $\overline{D}$, we now obtain
    \[
    \overline{D}(x) = \sum_{1 \leq i < j \leq n} \overline{H}(x_j^{(i+1)} \otimes x_i) + \sum_{i=1}^{n} \overline{D}(x_i^{(1)}) = \sum_{1 \leq i < j \leq n} \overline{H}(\tilde{x}_j^{(i+1)} \otimes \tilde{x}_i) + \sum_{i=1}^{n} \overline{D}(\tilde{x}_i^{(1)}) = \overline{D}(\tilde{x}).
    \]
\end{proof}

\begin{rmk}\label{rmk: energy g ind type D_N+1}
    Keeping the notation from Proposition \ref{prop: energy g}, assume instead that $B_{\mu}^{t} = B_{\mu}^{t}(\sboxone)$.
    By applying the same argument as in Proposition \ref{prop: energy g}, and using Proposition \ref{prop: split D} in place of Proposition \ref{prop: split B}, we can show that $\overline{D}(x) = \overline{D}(\tilde{x})$. 
    We also need to modify the definition of $\inc(a,b)$ as follows: collection of $x = v_a \otimes \dots \otimes v_1 \otimes \emptyset \otimes \dots \otimes \emptyset$, where $v_a \succ \dots \succ v_1 = 1$ if $a > 0$. 
    
    Since we trivially have $\vac(x) = \vac(\tilde{x})$, we conclude $\energy_{q,t}(x) = \energy_{q,t}(\tilde{x}).$
\end{rmk}

\subsection{$\Aug$ preserves the energy}
We prove the following which was used in the proof of \eqref{eq: goal}. As before, $B^{t}_{\mu} = B^{t}_{\mu}(\hspace{0.4mm}\sboxeleven\hspace{0.4mm})$, unless stated otherwise.
\begin{proposition}\label{prop: aug preserves energy}
    Let $\mu = (\mu_1, \dots, \mu_n)$ be a nonnegative integer vector such that $\mu_1 \leq \min(\mu_2, \dots, \mu_n)$, and let $m$ be a nonnegative integer such that $\mu_1 + m \leq \min(\mu_2, \dots, \mu_n)$. 
    Let $x = x_n \otimes \cdots \otimes x_2 \otimes x_1 \in \HW(B_{\mu}^{t})$, where $x_1 = [1 \dots k]$ for some $k$, and define $x' = x_n \otimes \cdots \otimes x_2 \otimes x'_1$, where $x'_1 = [1 \dots (k+m)] \in B^{\mu_1+m, 1}$. 
    Then, we have the equality
    \begin{equation*}
        \overline{D}(x) = \overline{D}(x').
    \end{equation*}
\end{proposition}

We first establish the special cases of Proposition \ref{prop: aug preserves energy} (Lemmas \ref{lem: aux1} and \ref{lem: aux2}), and then use the local argument to complete the proof.

\begin{lem}\label{lem: aux1}
    Let $r_2 > r_1$ be nonnegative integers.
    Let $x = x_2 \otimes x_1 \in \HW(B^{r_2,1} \otimes B^{r_1,1})$.
    Consider the element $x' = x_2 \otimes x'_1 \in B^{r_2,1} \otimes B^{r_1+1,1}$, where $x_1 = [1 \dots n]$ and $x'_1 = [1 \dots (n+1)]$ for some integer $n$. 
    Then, we have the equality
    \begin{equation*}
        \overline{D}(x) = \overline{D}(x').
    \end{equation*}
\end{lem}
\begin{proof}
    When $x = \emptyset \otimes \emptyset$, the claim is trivially verified. Therefore, we assume that $x \neq \emptyset \otimes \emptyset$. By Proposition \ref{prop: split B}, $S(x)$ is of the form $y \otimes z$, where $y \in \inc(a_1, b_1)$ and $z \in \inc(a_2, b_2)$ for some $a_1, b_1, b_2 \geq 0$ and $a_2 > 0$. Let $x'' = \hw(x')$. We will show that $S(x'')$ is of one of the following forms:
    \begin{enumerate}[(a)]
        \item $y' \otimes z'$, where $y' \in \inc(a_1+1, b_1)$ and $z' \in \inc(a_2, b_2)$,
        \item $y' \otimes z'$, where $y' \in \inc(a_1, b_1)$ and $z' \in \inc(a_2+3, b_2-1)$,
        \item $y' \otimes z'$, where $y' \in \inc(a_1, b_1+1)$ and $z' \in \inc(a_2-1, b_2)$.
    \end{enumerate}
    In each of these cases, a direct computation of the energy function via \eqref{eq: inc energy} reveals that $\overline{D}(S(x'')) = \overline{D}(S(x)) + r_1 + r_2-1$, which implies that $\overline{D}(x) = \overline{D}(x')$.

    Since $x\in \HW(B^{r_2,1} \otimes B^{r_1,1})$, we denote $x = [1 \dots k | n+1 \dots p | \overline{p} \dots \overline{q}] \otimes [1 \dots n]$. If $n+1 \in x_1$, then we have $x'' = [1 \dots k+1 | n+2 \dots p | \overline{p} \dots \overline{q}] \otimes [1 \dots n+1]$. It is straightforward to check that $\midd(x) = \midd(x'')$ and $\tol(x) = \tol(x'')$, so we conclude that $S(x'') = (k+1) \otimes S(x)$, corresponding to (a).

    Next, assume that $n+1 \notin x_1$. In this case, $x_1 = [1 \dots k | \overline{n} \dots \overline{q}]$, and we have $x'' = [1 \dots k | \overline{n+1} \dots \overline{q+1}] \otimes [1 \dots n+1]$. We have $\tol(x) = q - k - 1$, and since $r_2 > r_1 \geq n$, it follows that $\midd(x) = \tol(x)$. By a similar argument, we conclude that $\midd(x'') = \tol(x'') = q - k$.

    First, consider the case where $x_1$ has an odd number of barred letters. 
    If $q - k - 1$ is even, then $x$ and $x''$ correspond to \textbf{(Case 1)} and \textbf{(Case 3)} in Proposition \ref{prop: split B}, respectively, implying that $S(x'')$ corresponds to (b). If $q - k - 1$ is odd, then $x$ and $x''$ correspond to \textbf{(Case 3)} and \textbf{(Case 1)} in Proposition \ref{prop: split B}, respectively, implying that $S(x'')$ corresponds to (c).

    Second, consider the case where $x_1$ has an even number of barred letters. 
    If $q = k + 1$, then $x$ and $x''$ correspond to \textbf{(Case 3)} and \textbf{(Case 2)} in Proposition \ref{prop: split B}, respectively, implying that $S(x'')$ corresponds to (b). 
    If $q > k + 1$, then both $\hat{x}_1 = [1 \dots k | n+1 | \overline{n+1} \dots \overline{q}]$ and $\hat{x''}_1 = [1 \dots k | n+2 | \overline{n+2} \dots \overline{q+1}]$ are admissible. With the same reasoning as before, exactly one of $\hat{x}_1 \otimes [1 \dots n]$ and $\hat{x''}_1 \otimes [1 \dots n+1]$ corresponds to \textbf{(Case 1)} in Proposition \ref{prop: split B}. Therefore exactly one of $x$ and $x''$ corresponds to \textbf{(Case 2)} in Proposition \ref{prop: split B}.
    In both cases, we can verify that $S(x'')$ corresponds to either (b) or (c).
\end{proof}

\begin{example}
For $3456\otimes 12 \in \HW(B^{4,1}\otimes B^{2,1})$ we have 
\begin{align*}
    S(3456\otimes 12)&=6\otimes 5\otimes4\otimes 3\otimes 2 \otimes 1\\
    S(\hw(3456\otimes 123))=S(1456\otimes 123)&=1\otimes 6\otimes 5\otimes4\otimes 3\otimes 2 \otimes 1.
\end{align*}
which goes to (a) in the proof of Lemma \ref{lem: aux1}.

For $12\otimes 12 \in \HW(B^{4,1}\otimes B^{2,1})$ we have 
\begin{align*}
    S(12\otimes 12)&=2\otimes 1\otimes2\otimes 1\otimes \overline{1}\otimes 1\\
    S(\hw(12\otimes 123))=S(12\otimes 123)&=2\otimes 1\otimes\overline{4}\otimes 4\otimes 3 \otimes 2\otimes 1.
\end{align*}
which goes to (b).

For $\bar{3}\bar{2}\otimes 123 \in \HW(B^{4,1}\otimes B^{3,1})$ we have 
\begin{align*}
    S(\bar{3}\bar{2}\otimes 123)&=\overline{2}\otimes \overline{3}\otimes\overline{4}\otimes 4\otimes 3\otimes 2\otimes 1\\
    S(\hw(\bar{3}\bar{2}\otimes 1234))=S(\bar{4}\bar{3}\otimes 1234)&=\overline{1}\otimes 1\otimes \overline{3}\otimes\overline{4}\otimes 4\otimes 3\otimes 2\otimes 1
\end{align*}
which goes to (c).
\end{example}

\begin{lem}\label{lem: aux2}
    Let $r_2 > r_1$ be nonnegative integers, and let $m$ be a nonnegative integer such that $r_2 \geq r_1 + m$. Consider $x = x_2 \otimes x_1 \in \HW(B^{r_2,1} \otimes B^{r_1,1})$. Define $x' = x_2 \otimes x'_1 \in \HW(B^{r_2,1} \otimes B^{r_1+m,1})$, where $x_1 = [1 \dots n]$ and $x'_1 = [1 \dots (n+m)]$ for some $n$. Then, we have
    $
        \overline{D}(x) = \overline{D}(x').
    $
\end{lem}

\begin{proof}
    Let $x'' = \hw(x')$. We can obtain $x''$ by applying the following iterative process. Initially, set $y^{(0)} = x$, and for each $i$, produce $y^{(i+1)} \in \HW(B^{r_2,1} \otimes B^{r_1+i+1,1})$ from $y^{(i)} \in \HW(B^{r_2,1} \otimes B^{r_1+i,1})$ as follows. Since $y^{(i)}$ is of the form $z \otimes [1 \dots (n+i)]$, we set $y^{(i+1)} = \hw(z \otimes [1 \dots (n+i+1)])$.

    After $m$ iterations, we obtain $x'' = y^{(m)}$. By Lemma \ref{lem: aux1}, we then have
    $
        \overline{D}(y^{(i)}) = \overline{D}(y^{(i+1)})
    $ which gives $\overline{D}(x)=\overline{D}(x'')$.
\end{proof}

\begin{proof}[Proof of Proposition \ref{prop: aug preserves energy}]
    For $2 \leq i < j \leq n$, let $x_j^{(i)}$ denote the element of the $i$-th tensor factor obtained by applying the composition of combinatorial $R$-matrices that places the $j$-th tensor factor in the $i$-th tensor factor in $x$. Then, we can express the energy function as
    \begin{align*}
        \overline{D}(x) &= \sum_{2 \leq i < j \leq n} \overline{H}(x_j^{(i+1)} \otimes x_i) + \sum_{i=2}^{n} \overline{D}(x_i^{(2)} \otimes x_1)-(n-2)\overline{D}(x_1) \\
        &= \sum_{2 \leq i < j \leq n} \overline{H}(x_j^{(i+1)} \otimes x_i) + \sum_{i=2}^{n} \overline{D}(x_i^{(2)} \otimes x'_1)-(n-2)\overline{D}(x'_1)  = \overline{D}(x')
    \end{align*}
    where we used the fact that $\overline{D}(x_i^{(2)} \otimes x_1) = \overline{D}(x_i^{(2)} \otimes x'_1)$, which follows from Lemma \ref{lem: aux2} and $\overline{D}(x_1)=\overline{D}(x'_1)$ which is trivial.
\end{proof}

\begin{rmk}\label{rmk: aug D_n}
    Proposition \ref{prop: aug preserves energy} is true when $B_{\mu}^{t}=B_{\mu}^{t}(\sboxone)$. The proof is parallel. 
    We also conjecture that it is true for $B_{\mu}^{t}=B_{\mu}^{t}(\sboxtwo)$.
    It would be nice to have a uniform proof for Proposition \ref{prop: aug preserves energy}, possibly exploiting rigged configurations \cite{OSS2018}.
\end{rmk}
\subsection{Proof of Lemma \ref{lem: energy r+m}}\label{appendix: energy increase}
Our objective is to prove Proposition \ref{prop: B energy r+m} which directly implies Lemma \ref{lem: energy r+m}. 
We set $B_{\mu}^{t}=B_{\mu}^{t}(\hspace{0.4mm}\sboxeleven\hspace{0.4mm})$, unless otherwise stated.

\begin{lem}\label{lem: splitting ssot support}
    Given $c\in B^{r,1}$ for $r\geq 2$, let $x\in \inc(a,b)$ for $a>0$ be $v_a\otimes \dots v_1\otimes \underbrace{\overline{1}\otimes 1 \otimes\dots  \otimes \overline{1}\otimes 1}_{2b}$ such that $v_a \neq \overline{1}$.
    We have the following.
    
    (1) If $c\neq \emptyset$ and $c_1=1$, $S(c\otimes x)$ is of the form $y\otimes z$ where $y\in \inc(a_1,b_1)$ and $z\in \inc(a_2,b_2)$ for some $a_1,a_2>0$.

    (2) If $c=\emptyset$, $S(c\otimes x)$ is of the form $y\otimes z$ where $y\in \inc(0,b_1)$ and $z\in \inc(a_2,b_2)$ for some $a_2>0$.

    (3) If $c\neq \emptyset$ and $c_1\succ v_a $, $S(c\otimes x)$ is of the form $y\otimes z$ where $y\in \inc(0,b_1)$ and $z\in \inc(a_2,b_2)$ for some $a_2>0$.
\end{lem}
\begin{proof}
    We proceed by induction on $a$. First, consider the base case when $a = 1$ or $a = 2$. In this case, the word $v := v_1 \dots v_a$ is admissible, and we have $S(v) = x$, where we regard $v \in B^{a+2b,1}$.

   (1) We have $S(c \otimes x) = S(c \otimes v)$ by the definition of the splitting map. Since both $c$ and $v$ contain the letter $1$, by Proposition \ref{prop: split B}, $S(\hw(c \otimes v))$ is of the form $y' \otimes z'$ for $y' \in \inc(a_1, b_1)$ and $z' \in \inc(a_2, b_2)$, where $a_1, a_2 > 0$. Therefore, $S(c \otimes v)$ is obtained from $y' \otimes z'$ by applying classical crystal operators. We conclude that there exist $y \in \overline{\inc}(a_1, b_1)$ and $z \in \overline{\inc}(a_2, b_2)$ such that $S(c \otimes v) = y \otimes z$. Since $1 \in c, v$, the weight consideration implies that $y \in \inc(a_1, b_1)$ and $z \in \inc(a_2, b_2)$.

    (2) By Proposition \ref{prop: split B}, $S(\hw(c \otimes v))$ is of the form $y' \otimes z'$ for $y' \in \inc(0, b_1)$ and $z' \in \inc(a_2, b_2)$, where $a_2 > 0$. We conclude that there exists $z \in \overline{\inc}(a_2, b_2)$ such that $S(c \otimes v) = y' \otimes z$. The weight consideration then shows that $z \in \inc(a_2, b_2)$.

    (3) The argument for this case is similar to the one in (2), and follows the same reasoning.

    Now, assume $a > 2$, and that claims (1), (2), and (3) hold for smaller values of $a$. We define $v'_a \otimes c' = R(c \otimes v_a)$ and $v'_{a-1} \otimes c'' = R(c' \otimes v_{a-1})$. Finally, we denote by $x'$ the result of erasing the left-most factor $v_a$ from $x$, and by $x''$ the result of erasing the left-most two factors $v_a$ and $v_{a-1}$ from $x$.

    (1) If $v'_a = 1$, then either $c'_1 \succ v_{a-1}$ or $c' \neq \emptyset$. By the induction hypothesis for (2) or (3), we have $S(c' \otimes x') = y' \otimes z'$ for some $y' \in \inc(0, b_1)$ and $z' \in \inc(a_2, b_2)$ with $a_2 > 0$. Since $S(c \otimes x) = 1 \otimes S(c' \otimes x')$, the claim follows. If $v'_a \neq 1$, then we have $v'_a \succ v'_{a-1}$ and $1 \in c'$. In this case, the induction hypothesis for $S(c' \otimes x')$ says $S(c' \otimes x') = y' \otimes z'$ for some $y' \in \inc(a_1, b_1)$ and $z' \in \inc(a_2, b_2)$ with $a_1, a_2 > 0$. As we have $v'_a\otimes y'\in \inc(a_1+1,b_1)$ the proof follows.

    (2) We must have $v'_a \otimes v'_{a-1} = \overline{1} \otimes 1$ and $c'' = v_{a-1} v_a$. Since $S(c \otimes x) = \overline{1} \otimes 1 \otimes S(c'' \otimes x'')$, the induction hypothesis for $S(c'' \otimes x'')$ completes the claim.

    (3) If $v'_a \otimes v'_{a-1} = \overline{1} \otimes 1$, then $c''_1 \succ v_{a-2}$, and the induction hypothesis for $S(c'' \otimes x'')$ completes the claim. Otherwise, we have $v'_a \succ v'_{a-1}$ and $c'_1 \succ v_{a-1}$. In this case, the induction hypothesis for $S(c' \otimes x')$ completes the claim.
\end{proof}
\begin{lem}\label{lem: claim to lem}
     Let $y^{(i)} \in \inc(a'_i, b'_i)$ for $a'_i > 0$, such that the leftmost factor of $y^{(i)}$ is not $\overline{1}$. Then for $c \in B^{r, 1}$, where $r > 1$, we have the following.
    \begin{enumerate}
        \item If $c_1 = 1$, then $S(c \otimes y^{(n)} \otimes \dots \otimes y^{(1)})$ is of the form $ x^{(n+1)} \otimes \dots \otimes x^{(1)}, $ where each $x_i \in \inc(a_i, b_i)$ for some $a_i > 0$, and the leftmost factor of each $x_i$ is not $\overline{1}$.
        \item If $c = \emptyset$, then $S(c \otimes y^{(n)} \otimes \dots \otimes y^{(1)})$ is of the form $ z \otimes x^{(n)} \otimes \dots \otimes x^{(1)}$, where each $x_i \in \inc(a_i, b_i)$ for some $a_i > 0$, with the leftmost factor of each $x_i$ not equal to $\overline{1}$, and $z \in \inc(0, b_{n+1})$.
    \end{enumerate}
\end{lem}
\begin{proof}
    We proceed by induction on $n$. The base case, $n = 1$, follows directly from Lemma \ref{lem: splitting ssot support} (1) and (2). Now consider the case $n > 1$ and assume that the claim holds for $n-1$. 

(1) If \( c_1 = 1 \), then by Lemma \ref{lem: splitting ssot support} (1), \( S(c \otimes y^{(n)}) \) is expressed as  
\[
z^{(2)} \otimes z^{(1)} = \underbrace{v_{a''_2} \otimes \dots \otimes v_1 \otimes \bar{1} \otimes 1 \otimes \dots \otimes \bar{1} \otimes 1}_{z^{(2)}} \otimes \underbrace{w_{a''_1} \otimes \dots \otimes w_1 \otimes \bar{1} \otimes 1 \otimes \dots \otimes \bar{1} \otimes 1}_{z^{(1)}},
\]
where each \( z^{(i)} \in \mathrm{Inc}(a''_i, b''_i) \) and \( a''_i > 0 \).  Next, let \( y'^{(n)} \otimes c' \) denote the result of moving \( c \) to the rightmost position in \( c \otimes y^{(n)} \) by applying the combinatorial \( R \)-matrices. This results in one of the following two cases:  
\begin{itemize}
        \item $y'^{(n)}=z^{(2)} \otimes w_{a''_1}\otimes\dots\otimes w_1 \otimes\underbrace{\bar{1}\otimes 1 \otimes \dots \ \otimes\bar{1}\otimes 1}_{2m}$ where $m<b''_1$ and $c'=\emptyset$,
        \item $y'^{(n)}=z^{(2)} \otimes w_{a''_1}\otimes\dots\otimes w_s$ where $s>1$ and $c'=w_1\dots w_{s-1}$.
    \end{itemize}
In the first case, the induction hypothesis for (2) completes the proof.  
In the second case, observe that the leftmost factor of  
$
S(c' \otimes y^{(n-1)} \otimes \dots \otimes y^{(1)})
$  
must be smaller (in the \( \succ \) ordering) than \( w_s \). Hence, the induction hypothesis for (1) completes the proof.

    (2) The case $c=\emptyset$ follows similarly.
\end{proof}

\begin{corollary}\label{prop: splitting of ssot full length}
    Let $\lambda$ be a partition of length $n$, and let $\mu$ be a nonnegative integer vector of length $n$. For $T \in \HW(B_{\mu}^{t}, \lambda^{t})$, we have that $S(T)$ is of the form
    \[
    S(T) = x^{(n)} \otimes \dots \otimes x^{(1)},
    \]
    where each $x^{(i)} \in \inc(a_i, b_i)$ for some $a_i > 0$ and $b_i \geq 0$.
\end{corollary}
\begin{proof}
The proof easily follows from the induction on $n$ applying Lemma \ref{lem: claim to lem} (1).
\end{proof}

\begin{lem}\label{lem: splitting form}
    Let the assumptions be the same as those in Corollary \ref{prop: splitting of ssot full length}. Then the leftmost $\lambda_n$ factors of $S(T)$ are of the form $\lambda_n \otimes \dots \otimes 1$.
\end{lem}

\begin{proof}
 We proceed by induction on $n$. The base case, $n = 1$, is trivial.  Now, assume $n > 1$ and that the claim holds for $n - 1$.  
 
 Let $T = T_n \otimes \dots \otimes T_1$ and let $\nu$ be a partition such that $\wt(T_{n-1}\otimes\dots \otimes T_1)=\nu^{t}$. Then we must have $\ell(\nu)=n-1$ and $\nu_{n-1}\geq \lambda_n$. By the induction hypothesis, we know that 
    \[
    S(T_{n-1} \otimes\dots \otimes T_1) = \nu_{n-1} \otimes \dots \otimes 1 \otimes \cdots,
    \]
    where the leftmost $\nu_{n-1}$ factors are of the desired form.  Let $T_n$ be represented by the word $w$. For this word, we have $w_i = i$ for $1 \leq i \leq \lambda_n$, and $w_{\lambda_n + 1} \succ \nu_{n-1}$. 
    By applying the rule (Case 2-1) in Appendix \ref{app:comb R B}, we have $R(w \otimes \nu_{n-1}) = \lambda_n \otimes \left( w \setminus \{\lambda_n\} \cup \{\nu_{n-1}\} \right).$
    Applying this rule iteratively $\lambda_n$ times completes the proof.
\end{proof}

\begin{example}
    Let $\lambda = (2,2,2)$ and $\mu = (4,4,4)$. 
    Consider the element $ T = 12\bar{4}\bar{3} \otimes 12 \otimes 1234 \in \HW(B_{\mu}^{t}, \lambda^t).$
    Then, image of the splitting map is given by
    \[
    S(T) = \underbrace{2 \otimes 1}_{x^{(3)}} \otimes \underbrace{\overline{3} \otimes \overline{4} \otimes 2 \otimes 1 \otimes \overline{1} \otimes 1}_{x^{(2)}} \otimes \underbrace{4 \otimes 3 \otimes 2 \otimes 1}_{x^{(1)}}.
    \]
    This is of the form $x^{(3)} \otimes x^{(2)} \otimes x^{(1)}$, where $x^{(3)}\in \inc(2,0)$,  $x^{(2)}\in \inc(4,1)$ and $x^{(1)}\in \inc(4,0)$ as predicted in Proposition \ref{prop: splitting of ssot full length}.
    Furthermore, the leftmost two factors of $S(T)$ are $2 \otimes 1$, which is consistent with the result in Lemma \ref{lem: splitting form}.
\end{example}

\begin{lem}\label{lem: comb R for special}
    Let $x \in \inc(a,b)$ be written as $v_a \otimes \dots \otimes v_1 \otimes \overline{1} \otimes 1 \otimes \dots \otimes \overline{1} \otimes 1$, where $v_i = i$ for $1 \leq i \leq k$ and $v_a \prec \overline{k}$ for some positive integer $k$. 
    Let $c = [1 \dots r] \in B^{r',1}$, where $r' < k$.
    Then, we have
    \[
        S(c \otimes x) = S(c) \otimes x = r \otimes \dots \otimes 1 \otimes \underbrace{\overline{1} \otimes 1 \otimes \dots \otimes \overline{1} \otimes 1}_{r' - r} \otimes x.
    \]
\end{lem}

\begin{proof}
    Note that any word $w = w_1 \dots w_\ell$ of length no greater than $k$, where each $w_i \prec \overline{k}$, is admissible. 
    To prove the claim, we move the leftmost factor $c$ of $c \otimes x$ to the right by applying combinatorial $R$-matrices according to the rules (Case 2-1), (Case 1-4), (Case 1-2), and (Case 3-2) specified in Appendix \ref{app:comb R B}.
\end{proof}

\begin{proposition}\label{prop: B energy r+m}
    For a partition $\lambda$ of length $n$ and a nonnegative integer vector $\mu$ of length $n$, let $T\in\HW(B_{\mu}^{t},\lambda^{t})$. 
    For nonnegative integers $r$ and $m$ such that $r+2m<\lambda_n$, let $v=v_1\dots v_r\in B^{r+2m,1}$ such that $v\otimes T\in \HW(B^{t}_{(\mu,r+2m)})$ and $v_1\succ \lambda_n$. 
    Then we have
    \begin{equation*}
         S(v\otimes T)=\underbrace{\overline{1}\otimes 1\otimes \dots \otimes \overline{1} \otimes 1}_{2m} \otimes v_r\otimes \dots \otimes v_1\otimes S(T).
    \end{equation*}
    In particular, we have $\overline{D}(v\otimes T)=\overline{D}(T)+r+m$.
\end{proposition}
\begin{proof}
    For simplicity, denote $k = \lambda_n$. By Corollary \ref{prop: splitting of ssot full length}, we have
    $
        S(T) = x^{(n)} \otimes x^{(n-1)} \otimes \dots \otimes x^{(1)},
    $
    where each $x^{(i)} \in \inc(a_i, b_i)$ with $a_i > 0$. In particular, by Lemma \ref{lem: splitting form}, we have
    $
        x^{(n)} = k \otimes \dots \otimes 1 \otimes \underbrace{\overline{1} \otimes 1 \dots \overline{1} \otimes 1}_{2b_n}.
    $
    If $m > 0$, applying the $R$-matrix gives
    $
        R(v_1 \cdots v_r \otimes k) = \overline{1} \otimes 1 k  v_1 \cdots v_r,
    $
    and
    $
        R(1  k  v_1 \cdots v_r \otimes (k-1)) = 1 \otimes (k-1)  k v_1 \cdots v_r.
    $
  We repeat this process until we obtain $m$ pairs of $\overline{1} \otimes 1$. The element in $B^{r+2m,1}$ is then
    $
        (k - 2m + 1) \dots k  v_1 \dots v_r.
    $
   Again applying the $R$-matrix $r$ times using the rule (Case 1-3) in Appendix \ref{app:comb R B}, we obtain $v_r \otimes \cdots \otimes v_1$ as the front part, with the remaining term
    \begin{equation}\label{eq:dddddddddd}
        v' \otimes (k - r - 2m) \otimes \dots \otimes 1 \otimes \underbrace{\overline{1} \otimes 1 \otimes \dots \otimes \overline{1} \otimes 1}_{2b_n} \otimes x^{(n-1)} \otimes \dots \otimes x^{(1)},
    \end{equation}
    where $v' = (k - r - 2m + 1) \dots k$. Since the weight of $S(T)$ is $\lambda^t$, each $x^{(i)}$ must satisfy the conditions in Lemma \ref{lem: comb R for special}. In other words, representing $x^{(i)}=v_{a_i}\otimes \dots \otimes v_1\otimes \overline{1} \otimes 1 \dots \overline{1} \otimes 1$, we have that $v_i=i$ for $1\leq i\leq k$ and $v_{a_i}\prec \overline{k}$. We now check that
    $$
        S(v' \otimes (k - r - 2m) \otimes \dots \otimes 1 \otimes \underbrace{\overline{1} \otimes 1 \otimes \dots \otimes \overline{1} \otimes 1}_{2b_n}) = k \otimes \dots \otimes 1 \otimes \underbrace{\overline{1} \otimes 1 \otimes \dots \otimes \overline{1} \otimes 1}_{2b_n},
    $$
     by Proposition \ref{prop: split B}. Therefore in \eqref{eq:dddddddddd}, if we move $v'$ to the right until it encounters $x^{(n-1)}$ by consecutively applying combinatorial $R$-matrices, we obtain $[1\dots (r+2m-2b_n)]\in B^{r+2m,1}$. From here we apply Lemma \ref{lem: comb R for special} consecutively, concluding that the image of the splitting map to the expression \eqref{eq:dddddddddd} equals in $S(T)$.

    Now, we prove the claim for the energy function. Denote $S(T) = y_{|\mu|} \otimes \dots \otimes y_1$, where each $y_i \in B^{1,1}$. By Corollary \ref{prop: splitting of ssot full length}, we have
    $
        \sum_{i=1}^{|\mu|-1} \overline{H}(y_{i+1} \otimes y_i) = |\mu| - n.
    $
Thus, we compute the energy function as follows
    \begin{align*}
        \overline{D}(S(v\otimes T))&=\sum_{i=1}^{m}2(2i-1)+\sum_{i=2m+1}^{r+2m}i+\sum_{i=1}^{|\mu|-1}(|\mu|+r+2m-i)\overline{H}(y_{i+1}\otimes y_i)\\
        &=\sum_{i=1}^{m}2(2i-1)+\sum_{i=2m+1}^{r+2m}i+(r+2m)(|\mu|-n)+\overline{D}(S(T))
    \end{align*}
    which gives $\overline{D}(v\otimes T)=\overline{D}(T)+r+m$.
\end{proof}

\begin{example}
For $T=12345\overline{6}\otimes 123456\in \HW(B^{6,1}\otimes B^{6,1})$ and $v\otimes T=\overline{5}\otimes 12345\overline{6}\otimes 123456\in \HW(B^{3,1}\otimes B^{6,1}\otimes B^{6,1})$, we have 
\begin{align*}
S( T)&= 5\otimes 4\otimes3\otimes 2\otimes1\otimes 5\otimes4\otimes 3\otimes 2 \otimes 1\otimes \overline{1}\otimes 1\\
    S(v\otimes T)&=\overline{1}\otimes 1\otimes \overline{5}\otimes 5\otimes 4\otimes3\otimes 2\otimes1\otimes 5\otimes4\otimes 3\otimes 2 \otimes 1\otimes \overline{1}\otimes 1.
\end{align*}
As proved in Proposition \ref{prop: B energy r+m}, we have $S(v\otimes T)=\overline{1}\otimes 1\otimes \overline{5}\otimes S(T)$.
\end{example}

We state the $\diamond=\sboxone$ version of Proposition \ref{prop: B energy r+m}. The proof is parallel so we omit the details. 

\begin{proposition}\label{prop:split D}
    For partition $\lambda$ of length $n$ and a length $n$ integer vector $\mu$, let $T\in\HW(B_{\mu}^{t},\lambda^{t})$. For nonnegative integer $r$ and $m$ such that $r+m<\lambda_n$, let $v=v_1\dots v_r\in B^{r+m,1}$ such that $v\otimes T\in \HW(B^{t}_{(\mu,r+m)})$ and $v_1\succ \lambda_n$. Then we have
    \begin{equation*}
         S(v\otimes T)=\underbrace{\emptyset\otimes \dots \otimes\emptyset}_{m} \otimes v_r\otimes \dots \otimes v_1\otimes S(T).
    \end{equation*}
    In particular we have $\overline{D}(v\otimes T)=\overline{D}(T)+2r+m$, therefore
    \begin{equation*}
        \energy_{q,t}(v\otimes T)=q^r t^m \energy_{q,t}(T).
    \end{equation*}
\end{proposition}

\section{Proof of Lemma \ref{lem: useful lemma}}\label{Sec: append B}
We now prove Lemma \ref{lem: useful lemma}. Specifically, we will establish the claim for $\alpha(T)$, namely, $\alpha(T)^{\leq k+1} = \alpha(T^{\leq 2k+1})$. The analogous claim for $\beta(T)$ can be proved in a similar manner. As outlined in Section \ref{subsub: filtering x=k BC}, we restrict our attention to words or semistandard Young tableaux with distinct entries.

We present Burge's algorithm \cite{Burge1974}, which constructs $\tab(\bar{I})$ directly from $I= \begin{pmatrix}
    j_1 & \dots & j_r \\
    i_1 & \dots & i_r
\end{pmatrix}\in\TL(r)$\footnote{In \cite{Burge1974}, the algorithm was described for $I \in \TL(r)$ under the assumption that no column contains duplicate entries. The generalization to any $ I\in \TL(r)$ is straightforward and was noted in \cite{S05}.}. 
We construct a sequence of tableaux $T_0, \dots, T_r$ recursively as follows:
\begin{itemize}
    \item Initialize by setting $T_0$ as the empty tableau.
    \item Given $T_{k-1}$, construct $T_k$ according to the following rules:
    \begin{itemize}
        \item If $j_k > i_k$, let $T'_k = i_k \rightarrow T_{k-1}$ and denote by $s$ the newly created cell. Add a new cell at the bottom of the column immediately to the right of the column containing $s$, with entry $j_k$.
        \item If $j_k = i_k$, set $T_k = i_k \rightarrow T_{k-1}$.
    \end{itemize}
\end{itemize}
Finally, we have $T_r = \tab(\bar{I})$.
\begin{example}
    Let $I=\begin{pmatrix}
        4 &5 &7 \\
        2 &5 &3 
    \end{pmatrix}$.
    Then we have
    \[
    T_1' = \begin{ytableau} 2 \end{ytableau} \quad
    T_1 = \begin{ytableau} 2 & 4 \end{ytableau} \quad 
    T_2 = \begin{ytableau} 2 & 4\\ 5 \end{ytableau} \quad
    T_3' = \begin{ytableau} 2 & 4\\ 3 & 5 \end{ytableau} \quad
    T_3 = \begin{ytableau} 2 & 4 & 7\\ 3 & 5 \end{ytableau}.
    \]
Note that $T_3$ equals $35247\rightarrow \emptyset=\tab(\bar{I})$.
\end{example}

If $I \in \TL(r)$ has no column with duplicate entries, then the shape of $\tab(\bar{I})$ is a partition $\lambda$ where each $\lambda_i$ is an even number, and we also have $\alpha(T) = \beta(T)$. In this case, it has been shown that the shape of $\alpha(T)$ is given by $\left(\frac{\lambda_1}{2}, \frac{\lambda_2}{2}, \dots \right)$ \cite[Lemma 10.7]{Sundaram1986}. The argument in \cite{Sundaram1986} can be straightforwardly generalized to any $I \in \TL(r)$, leading to the following lemma.
\begin{lem}\cite[Lemma 10.7, generalized]{Sundaram1986} \label{lem: shape alpha beta}
For $I \in \TL(r)$, let the shape of $\tab(\bar{I})$ be denoted by $\lambda$. Then, the following holds:

(1) The shape of $\alpha(T)$ equals $(\lfloor\frac{\lambda_1+1}{2}\rfloor,\lfloor\frac{\lambda_2+1}{2}\rfloor,\dots)$. \hspace{5mm}
(2) The shape of $\beta(T)$ equals $(\lfloor\frac{\lambda_1}{2}\rfloor,\lfloor\frac{\lambda_2}{2}\rfloor,\dots)$.
\end{lem}

We construct lemmas to complete the proof of Lemma \ref{lem: useful lemma}, starting with the introduction of some notations.

For a word $w = w_1 w_2 \dots w_n$ with distinct entries, we define $\second(w)$ as the sequence of letters that are bumped from the first row during the insertion process $\emptyset \leftarrow w$. For example, the insertion process $\emptyset \leftarrow 83215476$ proceeds as follows:
\[
\begin{ytableau} 8 \end{ytableau} \quad
\begin{ytableau} 3 \\8 \end{ytableau} \quad 
\begin{ytableau} 2\\3\\8 \end{ytableau} \quad
\begin{ytableau} 1\\2\\3\\8 \end{ytableau} \quad
\begin{ytableau}  1 &5\\2\\3\\8 \end{ytableau}\quad
\begin{ytableau}  1 &4\\2&5\\3\\8 \end{ytableau}\quad
\begin{ytableau}  1 &4&7\\2&5\\3\\8 \end{ytableau}\quad 
\begin{ytableau}  1 &4&6\\2&5&7\\3\\8 \end{ytableau}.
\]
In this example, $8$ is bumped from the first row, followed by $3$, $2$, $5$ and then $7$. Thus, $\second( 83215476) = (8,3,2,5,7)$.

Next, we define $\lseq(w)$ as a subsequence $(w_{i_1}, \dots, w_{i_\ell})$ determined recursively. Let $T_j = \emptyset \leftarrow w_1 \dots w_j$. If $w_n$ is placed in the $\ell$-th column of $T_n$, initialize by setting $i_\ell = n$. Then, for $k = \ell, \ell-1, \dots, 2$, once $i_k$ is defined, set $i_{k-1}$ to be the smallest index satisfying the following conditions:
\begin{itemize}
    \item $w_{i_{k-1}} < w_{i_k}$,
    \qquad \qquad $\bullet$ $i_{k-1}$ is placed in the $(k-1)$-th column of $T_{i_{k-1}}$.
\end{itemize}
Note that this process ensures $i_{k-1} < i_k$. Continuing with the example $w=83215476$, we have $\lseq(w)=(3,5,6)$.

\begin{lem}\label{lem: insert to the second row}
Let $w = w_1 w_2 \dots w_n$ be a word with distinct entries, and let $\lseq(w) = (w_{i_1}, \dots, w_{i_\ell})$. For each $j$, define $v_j$ as the subword of $w$ given by 
$
v_j = w_{i_{j-1}+1} w_{i_{j-1}+2} \dots w_{i_j-1},
$
where we regard $i_0 = 0$. Consider the word 
\[
w' = w_{i_1} v_1 w_{i_2} v_2 \dots w_{i_\ell} v_\ell.
\]
Then, $\second(w)$ is obtained by appending a letter $b$ to $\second(w')$, where $b$ is the letter bumped from the first row during the insertion of $w_n$ in the process $\emptyset \leftarrow w$. In particular, if $w_n$ is placed at the end of the first row, we have $\second(w) = \second(w')$ therefore $\emptyset \leftarrow w=\emptyset \leftarrow w'$.
\end{lem}

\begin{proof}
We prove the lemma by induction on $\ell$. For the base case $\ell = 1$, note that $w_n$ is the smallest letter in $w$, and thus the claim follows easily. Now let $\ell>1$ and assume the claim holds for $\ell - 1$.  Define
\[
u = w_1 \dots w_{i_{\ell-1}}, \quad \text{and} \quad u' = w_{i_1} v_1 w_{i_2} v_2 \dots w_{i_{\ell-1}} v_{\ell-1},
\]
then by the induction hypothesis, $\second(u)$ is obtained by appending $b'$ to $\second(u')$, where $b'$ is the letter bumped from the first row when inserting $w_{i_{\ell-1}}$. Note that $b'$ must be greater than $w_{i_\ell}$; otherwise, this would contradict the construction of $\lseq(w)$. Therefore, when $w_{i_\ell}$ is inserted into $\emptyset \leftarrow u'$, $b'$ is bumped out of the first row.

Let $S$ denote the set of letters in the first row of $\emptyset \leftarrow u$. After inserting $w_{i_\ell}$ into $\emptyset \leftarrow u'$, the first row of the resulting tableau consists of letters in the set $S \cup \{w_{i_\ell}\}$. Since $v_\ell$ does not contain any letters between $w_{i_{\ell-1}}$ and $w_{i_\ell}$, from here the insertion process are parallel giving $\second(w')=\second(w_1\dots w_{n-1})$. The proof is complete. 
\end{proof}

\begin{example}
Let $w = 83215476$, then we have $\second(w) = (8, 3, 2, 5, 7)$ and $\lseq(w) = (3, 5, 6)$. Consider $w' = 38521647$, then $\second(w') = (8, 3, 2, 5)$ as predicted in Lemma \ref{lem: insert to the second row}.
\end{example}

\begin{lem}\label{lem: lseq description for I}
Let $I \in \TL(r)$, and denote
$
I = \begin{pmatrix}
    j_1 & \dots & j_r \\
    i_1 & \dots & i_r 
\end{pmatrix}$ and $\bar{I} = \begin{pmatrix}
    a_1 & \dots & a_s \\
    b_1 & \dots & b_s 
\end{pmatrix}$. For a number $M$ larger than any entry in $I$, and a word $w = b_s \dots b_1 M$, the sequence $\lseq(w)$ is one of the following forms:
\begin{itemize}
    \item $(i_{u_1}, \dots, i_{u_k}, j_{u_k}, \dots, j_{u_1}, M)$, if the first row of $\tab(\bar{I})$ has length $2k$,
    \item $(i_{u_1}, \dots, i_{u_k}, j_{u_{k+1}}, \dots, j_{u_1}, M)$, where $i_{u_{k+1}} = j_{u_{k+1}}$, if the first row of $\tab(\bar{I})$ has length $2k+1$.
\end{itemize}
\end{lem}

\begin{proof}
We prove the case where the first row of $\tab(\bar{I})$ has length $2k$; the remaining case follows similarly. 

First, note that any longest increasing subsequence of $b_s \dots b_1$ is of the form $(i_{v_k}, \dots, i_{v_1}, j_{u_1}, \dots, j_{u_k})$. Let this be the first $2k$ components of $\lseq(w)$. If there exists an index $c$ such that $i_{v_c} \neq i_{u_c}$, we choose the smallest such $c$. In this case, we can replace $j_{u_c}$ with $j_{v_c}$, therefore the sequence $(i_{v_k}, \dots, i_{v_1}, j_{u_1}, \dots, j_{u_k})$ cannot be part of $\lseq(w)$. This completes the proof.
\end{proof}

\begin{proof}[Proof of Lemma \ref{lem: useful lemma}] We prove the equality $\alpha(T^{\leq 2k+1}) = \alpha(T)^{\leq k+1}$ by induction on $k$. The base case, $k=0$, is trivial, as we have $\alpha(T)^{\leq 1} = T^{\leq 1}$ by Burge's algorithm. Now consider $k\geq 1$ and assume the claim holds for $k-1$, i.e., $\alpha(T^{\leq 2k-1}) = \alpha(T)^{\leq k}$.

From here we proceed by induction on the size of $T$. The base case, when $T$ is the empty tableau, is trivial. For a non-empty semistandard Young tableau $T$, we assume the claim holds for tableaux of smaller size. To prove the claim for $T$, we reverse the step of Burge's algorithm. First, we identify the largest entry $M$ in $T$ and erase the corresponding cell. Then, for a cell located at the bottom of a column immediately to the left of the column containing $M$, we apply column deletion. Let the resulting tableau be $T_1$, and let $d$ be the letter that was bumped out after this process. Then $\alpha(T_1)$ is obtained from $\alpha(T)$ by applying column deletion to the cell $c$, which is determined by Lemma \ref{lem: shape alpha beta}, and the bumped-out letter is $d$.

By the induction hypothesis, we have $\alpha(T_1^{\leq 2k+1}) = \alpha(T_1)^{\leq k+1}$. We then apply Burge's algorithm again to $T_1^{\leq 2k+1}$ by column inserting the letter $d$ and adding a new cell with an entry $M$ at the bottom of the column immediately to the right of the column containing a newly created cell in $d\rightarrow T_1^{\leq 2k+1}$. Let the resulting tableau be $T_2$, and let $P = d \rightarrow \alpha(T_1)^{\leq k+1}$. Clearly, we have $\alpha(T_2) = P$.

(Case 1: $M$ is located in the $(2k+1)$-th column or to the left in $T$): In this case, we have $T_2 = T^{\leq 2k+1}$, and $P = \alpha(T)^{\leq k+1}$.

(Case 2: $M$ is located in the $(2k+2)$-th column): Here, $T_2$ is obtained by adding a cell with entry $M$ at the end of the first row of $T^{\leq 2k+1}$, and $P = \alpha(T)^{\leq k+1}$. Let $I \in \TL(r)$ be such that $\tab(\bar{I}) = T^{\leq 2k+1}$; we denote $I$ and $\bar{I}$ as follows:
\begin{equation}\label{eq: IIbar}
    I= 
    \begin{pmatrix}
        j_1 & \dots & j_r \\
        i_1 & \dots & i_r 
    \end{pmatrix}
    \quad\text{and}\quad   \bar{I}=
    \begin{pmatrix}
        a_1 & \dots & a_s \\
        b_1 & \dots & b_s 
    \end{pmatrix}.
\end{equation} 
By Lemma \ref{lem: lseq description for I}, the sequence $\lseq(b_s \dots b_1 M)$ is given by $(i_{u_1}, \dots, i_{u_k}, j_{u_{k+1}}, \dots, j_{u_1}, M)$ where $i_{u_{k+1}}=j_{u_{k+1}}$. We now define $J \in \TL(r)$, obtained from $I$ by replacing columns
$
    \begin{pmatrix}
        j_{u_m} \\
        i_{u_m}
    \end{pmatrix}
$
with
$
    \begin{pmatrix}
        j_{u_{m-1}} \\
        i_{u_m}
    \end{pmatrix},
$
where we regard $j_{u_0} = M$. By Lemma \ref{lem: insert to the second row}, we have 
$
    \tab(\bar{J}) = \emptyset \leftarrow (b_s \dots b_1 M) = T_2,
$
and 
$
    \tab(J) = \emptyset \leftarrow (i_r \dots i_1) = P.
$
Thus, we conclude that $\alpha(T_2) = \alpha(T^{\leq 2k+1}) = P$.

(Case 3: $M$ is located in the $(2k+3)$-th column or to the right): Let $T_3 = T_2^{\leq 2k+2}$. Since the proofs for (Case 1) and (Case 2) have already been established, we now deduce that $\alpha(T^{\leq 2k+1}) = \alpha(T_3)$. Let $I \in \TL(r)$ such that $\tab(\bar{I}) = T_3$. Denote $I$ and $\bar{I}$ as in  \eqref{eq: IIbar}. By Lemma \ref{lem: lseq description for I}, we express the sequence $\lseq(b_s \dots b_1 M)$ as $(i_{u_1}, \dots, i_{u_{k+1}}, j_{u_{k+1}}, \dots, j_{u_1}, M)$. Now, we define $J \in \TL(r+1)$ obtained from $I$ by modifying the columns as follows. Replace columns
$
    \begin{pmatrix}
        j_{u_m} \\
        i_{u_m}
    \end{pmatrix}
$
with
$
    \begin{pmatrix}
        j_{u_{m-1}} \\
        i_{u_m}
    \end{pmatrix},
$
where we set $j_{u_0} = M$, and insert a new column
$
    \begin{pmatrix}
        j_{u_{k+1}} \\
        j_{u_{k+1}}
    \end{pmatrix}.
$
By Lemma \ref{lem: insert to the second row}, we obtain
$
    \tab(\bar{J}) = \emptyset \leftarrow (b_s \dots b_1 M) = T_3 \leftarrow M = T_2,
$
and
$
    \tab(J) = \emptyset \leftarrow (i_r \dots i_1 j_{u_{k+1}}) = P.
$
Note that when inserting $j_{u_{k+1}}$ into $\emptyset \leftarrow i_r \dots i_1$, $j_{u_{k+1}}$ is placed at the end of the first row, so we have $P^{\leq k+1} = \emptyset \leftarrow i_r \dots i_1$. Therefore, we conclude that
\[
    \alpha(T^{\leq 2k+1}) = \alpha(T_3) = P^{\leq k+1} = \alpha(T)^{\leq k+1}.
\]
\end{proof}

\bibliographystyle{alpha}  
\bibliography{main.bib} 

\end{document}